\documentclass[11pt]{article}

\usepackage[utf8]{inputenc}
\usepackage[T1]{fontenc}
\usepackage{mathtools}
\usepackage[english]{babel}
\usepackage{amsmath}
\usepackage{amsfonts}
\usepackage{amssymb}
\usepackage{graphicx}
\usepackage[usenames,dvipsnames]{color}
\usepackage{dsfont}
\usepackage{hyperref}

\usepackage{fullpage}
\usepackage{amsthm}
\usepackage{authblk}

\newtheorem{theorem}{Theorem}[section]
\newtheorem{corollary}[theorem]{Corollary}
\newtheorem{claim}[theorem]{Claim}
\newtheorem{proposition}[theorem]{Proposition}

\newtheorem{lemma}[theorem]{Lemma}

\newtheorem*{fact*}{Fact}

\theoremstyle{definition}
\newtheorem{remark}[theorem]{Remark}

\allowdisplaybreaks

\usepackage[commentmarkup=todo]{changes}
\definechangesauthor[name={AI},color=purple]{ai}
\definechangesauthor[color=blue]{mk}
\definechangesauthor[color=red]{dm}
\definechangesauthor[color=violet]{pw}
\definechangesauthor[color=cyan]{ll}
\definechangesauthor[name={Proceedings},color=orange]{xx}

\makeatletter
\def\namedlabel#1#2{\begingroup
    #2%
    \def\@currentlabel{#2}%
    \phantomsection\label{#1}\endgroup
}

\newcommand{\N}{{\mathbb N}}
\newcommand{\Gnp}{\ensuremath{G(n,p)}}

\newcommand{\E}{\mathbb E}
\newcommand{\Prob}{\mathbb{P}}
\newcommand{\eps}{\varepsilon}
\newcommand{\cA}{\mathcal A}

\newcommand{\cX}{\mathcal X}

\newcommand{\qc}{\Lambda} 

\newcommand{\onemp}{q} 
  
\newcommand{\indic}[2]{\mathds{1}_{#2\in #1}}                                   
\newcommand{\Lvl}[1]{#1}
\newcommand{\Lvltl}[3]{{#1}_{#2}({#3})}
\newcommand{\sLvl}[1]{|#1|}
\newcommand{\sLvltl}[3]{|{#1}_{#2}({#3})|}

\newcommand{\evnt}[2]{\mathcal{#1}^{\mathrm{#2}}}

\date{}

\begin{document}



\title{Consensus in label propagation on binomial random graphs\footnote{This manuscript is the full version of the extended abstract submitted for publication in the proceedings of the ACM–SIAM Symposium on Discrete Algorithms (SODA 2027).}}
\author[1]{Marcos Kiwi\footnote{Marcos Kiwi has been partially supported by grant GrHyDy ANR-20-CE40-0002 and BASAL funds for centers of excellence from ANID-Chile (FB210005) and STIC AmSud Project ECO-DIST (AMSUD 240005).}}
\author[2]{Lyuben Lichev\footnote{Lyuben Lichev has been supported by the Austrian Science Fund (FWF) [10.55776/ESP624]. For open access purposes, the authors have applied a CC BY public copyright license to any author-accepted manuscript version arising from this submission.}}
\author[3]{Dieter Mitsche\footnote{Dieter Mitsche has been supported by ANID-FONDECYT regular grant 1261336.}}
\author[4]{Pawe\l{} Pra\l{}at\footnote{Pawe\l{} Pra\l{}at has been partially supported by NSERC Discovery Grant. Part of this work was done while the author was visiting the Simons Institute for the Theory of Computing.}}
\affil[1]{Univ.~Chile, Santiago, Chile}
\affil[2]{TU Wien, Vienna, Austria}
\affil[3]{IMC, Pontif\'{i}cia Univ. Cat\'{o}lica, Santiago, Chile}
\affil[4]{Toronto Metropolitan University, Toronto, Canada}

\maketitle



\begin{abstract}
Understanding how local update rules amplify small random fluctuations into global consensus is a central challenge in the study of network dynamics and decentralized processes. A prominent family of examples is given by label propagation algorithms, widely used heuristics for community detection whose analysis is complicated by the strong dependencies created as the process evolves. Despite the simplicity and practical success of this class of algorithms, a rigorous understanding of their convergence remains limited. In this paper, we study a specific variant, which we refer to as LPA,
on the binomial random graph $\Gnp$---a canonical null model with no planted community structure.

We show that LPA reaches consensus asymptotically almost surely when $np\ge n^{5/8+\varepsilon}$ for any fixed $\varepsilon>0$, 
thus breaking the current best $n^{3/4+\varepsilon}$ analytical barrier.
More remarkably, we uncover a phase transition governing \emph{which} label survives: 
when $np\gg n^{2/3}$, the surviving label is asymptotically almost surely the smallest initial label, whereas for $n^{5/8+\varepsilon}\le np\ll n^{2/3}$, the surviving label is asymptotically almost surely not the smallest one. 
We further show that, in the critical window $np=\Theta(n^{2/3})$, both outcomes have probability $\Theta(1)$.

Our analysis has several important ingredients. 
We analyze a staged exploration that progressively reveals information about the random graph while preserving some randomness needed to track the amplification of label advantages across multiple rounds. 
This analysis is complemented by novel coupling and decoupling arguments that control the dependencies inherent to the process. 
Along the way, we establish a monotonicity result for competing binomial random variables that we believe is of independent interest.
\end{abstract}


\section{Introduction and main contributions}
%
Label propagation algorithms form a broad family of simple heuristics for community detection. In their most common form, vertices carry labels and repeatedly update them according to the labels observed in their neighborhoods. 
Since these algorithms require no prior knowledge of the community structure or the number of communities, and rely only on local information, they are widely used in graph mining and network analysis.

Despite their practical success, the theoretical behavior of the class of label propagation algorithms remains poorly understood. Even on the binomial random graph $\Gnp$, arguably the simplest possible test case, basic questions about convergence, convergence time, and the identity of the surviving label remain largely open.

Label propagation algorithms are instances of a more general and recurring phenomenon in the study of network dynamics and decentralized processes: simple local update rules can amplify small random fluctuations into large-scale, global agreement. Understanding when and how this happens is often far from straightforward, since even on a simple random network, the evolution of the process quickly creates strong dependencies across vertices and across rounds, making rigorous analysis technically challenging. 

In this paper, we analyze one instance of the abstract class of label propagation processes, which we call LPA.
Here, in a decentralized manner, the vertices are initially assigned the ranks $1,\ldots,n$ according to a bijection chosen uniformly at random.
Then, at each round, every vertex synchronously adopts the most frequent label among the vertices in its closed neighborhood. 
Ties in the first round are broken in favor of smaller labels, while ties in all subsequent rounds are broken uniformly at random; note that this rule is relatively natural since all labels are initially distinct, justifying the need for adopting a different convention in the first round. 

Our host graph will be the binomial random graph~$G(n,p)$, a canonical null model with no planted communities.
Thus, any dominant label that emerges is generated by the random graph and by the dynamics themselves, rather than by an underlying structure. 
This makes $G(n,p)$ a natural setting to test whether LPA can amplify purely random local fluctuations into global consensus and, if so, which label ultimately survives.

At first sight, one might expect consensus to emerge rapidly on $\Gnp$ in the 
regime of $p$ considered in this paper. 
Indeed, throughout the range of edge probabilities that we study, the graph has asymptotically almost surely\footnote{Often abbreviated a.a.s.; it means with
probability tending to one as $n\to\infty$.} diameter two, suggesting that information can spread very quickly throughout the network. 
Moreover, the drift of the tie-breaking rule towards smaller labels at the first round creates a persistent bias towards such smaller labels.
However, these heuristics are difficult to turn into a rigorous analysis due to the  correlations that arise between the labels present at different vertices.
As a result, understanding which label ultimately survives requires controlling a complex web of dependencies generated by the dynamics itself.
We note that it was shown in~\cite{KLPSS19} that LPA on random graphs with sufficiently large diameter (corresponding to $np=O(n^{\varepsilon})$ for some small but fixed $\varepsilon>0$) typically stabilizes with polynomially many labels, therefore exhibiting different qualitative behavior.

Our main contributions are summarized in the following theorem.
Throughout the paper, we write $f\ll g$ or $g\gg f$ to denote $f=o(g)$.
\begin{theorem}\label{thm 1}
Let $\varepsilon\in(0,1/24)$ and suppose that
$n^{5/8+\varepsilon}\le np\ll n$.
Then, a.a.s., LPA reaches a consensus configuration after five rounds.
Moreover, the surviving label coincides with the label that is most represented after the first round of the process. In addition,
\begin{itemize}
\item if $n^{2/3}\ll np\ll n$, then, a.a.s., the surviving label is $1$;

\item if $n^{5/8+\varepsilon}\le np\ll n^{2/3}$, then, a.a.s., the surviving label is different from~$1$.
\end{itemize}
\end{theorem}

A few remarks are in order. First, it follows from our proof that, when $n^{5/8+\varepsilon}\le np=O(n^{2/3})$, each of the first $\Theta((np^3)^{-1/2})$ labels has a comparable probability to become the last survivor.
Second, we observe that the conclusion in
the regime $n^{2/3}\ll np\ll n$
remains valid for the deterministic variant where ties are always broken in
favor of smaller labels. This is the tie-breaking convention considered
in~\cite{KPS13}, which analyzed the regime $np\ge n^{3/4+\varepsilon}$.
Third, the exponent $5/8$ in Theorem~\ref{thm 1} should not be interpreted as a conjectural threshold; we stop there since it corresponds to the range where the multi-stage exposure argument developed in this paper yields a clean rigorous analysis. 
We expect that the same method could be extended, with substantially more technical estimates, to the regime $np\ge n^{1/2+\eta}$ for every sufficiently small fixed $\eta>0$. 
We believe that, as one approaches the threshold $n^{1/2+\eta}$, the relevant bottleneck is no longer determined solely by the first-round advantage of the largest label class over its competitors.
Instead, one may need to track the number of labels that remain sufficiently represented after later rounds as well. Since this would require a considerably more technical multi-round analysis, 
we do not pursue it here.
Breaking the $n^{1/2}$ barrier, however, seems to require more significant new ideas and remains a challenging open problem, especially in view of the simulations in~\cite{KLPSS19} suggesting a threshold closer to~$np=\Theta(n^{1/5})$.

We also show a complementary result describing the transition window between the two regimes identified in the previous theorem, where neither of the two asymptotic regimes dominates completely.
\begin{proposition}\label{prop:criticalRegime}
If $np=cn^{2/3}$ for some constant $c>0$, then the probability that the surviving label is $1$ is bounded away from both $0$ and $1$.
\end{proposition}

Together, Theorem~\ref{thm 1} and Proposition~\ref{prop:criticalRegime} reveal a transition in the mechanism governing consensus formation. 
Our analysis shows that the eventual winner is typically determined by the label carried by the largest number of vertices after the first round. When $np\gg n^{2/3}$, the drift induced by the tie-breaking rule gives label~1 a sufficiently strong advantage for it to dominate a.a.s. In contrast, when $n^{5/8+\varepsilon}\le np\ll n^{2/3}$, random fluctuations among the labels carried by the largest numbers of vertices after the first round become decisive, and label~1 no longer enjoys such a deterministic advantage. In the transition window $np=\Theta(n^{2/3})$, these two effects are of comparable magnitude: consensus is still reached after five rounds a.a.s., but label~1 wins with probability bounded away from both $0$ and $1$.
This change in behavior highlights the interplay between the systematic bias created by the tie-breaking rule and the random fluctuations generated during the early dynamics.

In a more practical direction, we note that the rapid convergence of label propagation algorithms has long been observed empirically. 
For example, Raghavan, Albert and Kumara~\cite{raghavan2007near} report that their procedure typically approaches convergence after only a handful of iterations on a variety of real-world networks, while noting that a mathematically formal explanation for this behavior appears difficult. 
We partially fill this theoretical gap by providing a rigorous explanation of this phenomenon for a broad range of parameters in $G(n,p)$.

\smallskip

The main technical contribution of this paper is twofold. 
First, we introduce an exploration of the graph in several stages tailored to the analysis of LPA. 
Progressively revealing the states of carefully chosen edges over several rounds allows us to track the evolution of the dominant label (and estimate the gap to the other labels) over the first five rounds.
Second, the guided 
exploration is complemented by a collection of coupling and decoupling arguments that permit a quantitative comparison between competing labels despite the strong correlations generated by the dynamics. Together, these ingredients make it possible to identify small advantages at early stages of the process and to follow their amplification across subsequent rounds.

A key auxiliary contribution is a monotonicity theorem for competing binomial random variables.
Roughly speaking, it shows that, for two binomial random variables $X,Y$ with the same parameter~$p$ but different means, conditioning on $\max\{X,Y\}\ge M$ for some threshold $M>0$ increases the advantage of the dominant variable and, moreover, this advantage increases with $M$.
This result plays a crucial role in our analysis and appears to capture a phenomenon that is intrinsic to the study of the tie-breaking rule used in LPA; we believe that this result may be useful beyond the present application.

\subsection{Outline of the proof}
%
At a high level, the proof is organized around a carefully designed exploration that progressively reveals the random graph while preserving enough fresh randomness to analyze later rounds of the dynamics. The central difficulty is that, after only a few rounds, the labels carried by different vertices become highly correlated. A naive exposure of the graph would quickly consume all available randomness, making it difficult to control the evolution of the process. Our approach avoids this obstacle by revealing only small portions of the graph at each stage, while ensuring that the decisions taken coincide a.a.s.\ with those of LPA.

More precisely, we partition the vertex set into three subsets, denoted $A$,~$B$ and~$C$. 
The set $A$ consists of vertices initially carrying the first few smallest labels,
$B$ consists of their neighbors outside $A$, and $C$ contains all remaining vertices. 
In particular, vertices in levels $\Lvl{A}$ and $\Lvl{C}$ are not adjacent. 
We note that level~$\Lvl{C}$ contains, a.a.s., all but $o(n)$ vertices. 

Associated with each label represented in $A$ is a first-round basin (or simply basin) consisting of the vertices in $B$ which carry this label after round 1. 
The sizes of these basins constitute the key quantities governing the subsequent evolution of the process. A central part of the analysis is to identify the largest basin and to quantify the gap between its size and those of its competitors.

The remainder of the proof shows that the advantage of the largest basin is progressively amplified across subsequent rounds of the dynamics. The analysis first establishes that the largest basin exceeds all competitors by a sufficiently large margin. This initial advantage is then shown to propagate through later rounds, producing a dominant label after three rounds that occupies all but~$o(n)$ vertices. 
Two additional rounds then suffice to obtain consensus.
A key feature of the argument is that the edges within $C$ remain unexplored during the early stages of the analysis: recalling that $|C|=n-o(n)$, this leaves almost the entire graph unexplored during the first two rounds.

While this amplification mechanism is common to both regimes, the analysis of the largest basin differs substantially depending on the density of the graph. In the regime $n^{2/3}\ll np\ll n$, the bias induced by the tie-breaking rule dominates the relevant fluctuations, and we show that the basin associated with label $1$ is a.a.s.\ the largest. The subsequent amplification argument then implies that label $1$ eventually propagates throughout the entire graph.

In contrast, throughout the regime $n^{5/8+\varepsilon}\le np\ll n^{2/3}$, the advantage created by tie-breaking becomes smaller than
the fluctuations of the basin sizes. Consequently, label $1$ is no longer guaranteed to generate the largest basin in one round, although a.a.s.\ the eventual winner still emerges from the $\Theta((np^3)^{-1/2})$ smallest labels. 
Ensuring amplification of the most represented label is a more difficult task in this regime, as the gap between the largest and the second-largest basin decreases.

The critical window $np=\Theta(n^{2/3})$ lies precisely between these two behaviors. 
Here, the fluctuations and the tie-breaking bias are of the same order, yielding nontrivial probabilities for the identity of the winning label.
Thus, establishing Proposition~\ref{prop:criticalRegime} requires a refinement of the analysis used in the denser regime.

\subsection{Technical tools}\label{sec technical tools}
%
The proof strategy outlined above encounters several technical obstacles. 
The most fundamental of these is the complex dependency of the labels generated by the dynamics. 
Indeed, even quantities that appear simple at first sight, such as the number of vertices with fixed labels after two rounds, become highly correlated once conditioning on earlier rounds of the process.
A substantial part of our work is therefore devoted to developing a guided exploration procedure
together with coupling and decoupling arguments that preserve enough independence for the analysis to proceed.

A key technical ingredient is Lemma~\ref{lem:bis}.
It establishes a monotonicity property for competing binomial random variables.
The phenomenon captured by the result can already be seen in the
following basic setting: consider two independent binomial random
variables $X$ and $Y$ with the same success probability but different
means. In our application, $X$ and $Y$ correspond to suitably
coupled upper or lower bounds on the numbers of neighbors carrying
two competing labels.
To determine which label becomes dominant, however, one must
compare $X$ and $Y$ not in isolation, but taking into account
the counts associated with all other competing labels.
Roughly speaking, conditioning on $\max\{X,Y\}\ge M$ increases
the relative advantage of the variable with larger mean,
and this advantage is nondecreasing in the threshold $M$.
More precisely, Lemma~\ref{lem:bis} establishes this comparison
for the probabilities that $X$ and $Y$ win against each other
and the remaining competitors, taking into account the
possibility of ties.
Despite its natural formulation, this inequality appears to
be novel. While it plays a crucial role in our analysis,
its general formulation may allow for other applications.

The determination, via coupling, of the number of vertices that receive each of the relevant small labels after the first round leads to questions concerning the asymptotic distribution of the maximum of independent binomials whose means decrease with the label index, reflecting the first-round bias toward smaller labels.
This requires an analysis of order statistics of independent but non-identically distributed random variables; while the analogous question for i.i.d.~variables is significantly simpler and extensively studied, we were unable to find this more general extreme-value analysis adapted to our setting.
To address this question, we develop an approach that first determines the asymptotic behavior of the maximum, over a suitable interval of label indices, 
of independent binomial proxy variables whose means decrease with the label index; see Lemma~\ref{lem:max} and Lemma~\ref{lem:generalization}. 
By comparing the resulting asymptotics for different intervals, we identify the ranges of label indices where the first and second maxima are attained; see Corollary~\ref{cor second max}. We then estimate the gap between these two maxima in Lemma~\ref{lem fst snd max}.

Finally, several auxiliary estimates rely on refined probabilistic approximations, in particular Berry--Esseen bounds and Slud-type inequalities.
These estimates allow us to obtain quantitative control over differences of binomial random variables in regimes where standard concentration inequalities are not sufficiently precise.

\subsection{Related work}
%
Label propagation algorithms were introduced as scalable heuristics for community detection and have inspired numerous variants; see, for example, \cite{gregory2010finding,raghavan2007near} and the surveys \cite{bedi2016community,harenberg2014community,yang2016comparative}. 
Their appeal stems from their simplicity, efficiency, and ability to operate with little prior information about the network. Despite their widespread use in graph mining and network analysis, formal mathematical analyses remain relatively scarce.

To the best of our knowledge, there are only two rigorous analyses of label propagation algorithms on random graph models. 
Kothapalli, Pemmaraju and Sardeshmukh~\cite{KPS13} studied, on the stochastic block model, a deterministic tie-breaking variant in which ties are always broken in favor of smaller labels.
In the special case of a single block, corresponding to the binomial random graph $\Gnp$, they showed that, for every $\varepsilon>0$, if $np\ge n^{3/4+\varepsilon}$, then a.a.s.\ all vertices eventually receive label~1. Later, Knierim, Lengler, Pfister, Schaller and Steger \cite{KLPSS19} showed that for sufficiently sparse graphs consensus no longer occurs, disproving a conjecture from \cite{KPS13}. 
Simulations reported in \cite{KLPSS19} suggest that the consensus threshold, in terms of the average degree $d=np$, may be of order $d=\Theta(n^{1/5})$.

Our results reveal a transition between regimes in which a small systematic bias induced by tie-breaking is dominated by  or competes with 
the random fluctuations of the process. Related questions have recently attracted considerable attention in the study of majority dynamics on $\Gnp$. 
In a seminal paper, Benjamini, Chan, O'Donnell, Tamuz and Tan~\cite{BCO2016} initiated the rigorous study of majority dynamics on random graphs and conjectured a sharp consensus phenomenon that was subsequently established by Fountoulakis, Kang and Makai \cite{FKM20}; see also \cite{CN23,sah2021majority,TV20,Zeh20}. 
Although the mechanisms governing majority dynamics differ substantially from those underlying LPA, these works highlight the richness of local interaction processes on random graphs and the subtle interplay between deterministic bias and stochastic fluctuations.

A natural, widely studied, and arguably more applicable random graph model with planted community structure is the \emph{stochastic block model}: here, the vertex set is first partitioned into communities, and edges are then placed independently, with a larger probability within communities than between communities.
As already observed in~\cite{KPS13}, understanding label propagation type algorithms on the stochastic block model ultimately requires understanding the mechanisms governing its behavior within individual blocks. For this reason, the case of a single block, corresponding to the binomial random graph $\Gnp$, is not merely a simplified special case but a fundamental object of study. From this perspective, $\Gnp$ plays the role of a natural null model: it contains no planted community structure, and therefore any dominant labels or apparent communities produced by LPA arise purely from random fluctuations and from the update rule itself. We are not aware of a direct transfer from results on $\Gnp$ to results on general stochastic block models. Nevertheless, many of the phenomena that determine the success or failure of LPA already appear in the one-block setting. Moreover, while community structure in the stochastic block model often provides a strong signal that assists community detection, our goal is to understand the communities and dominant labels that arise purely from random fluctuations. We therefore focus on the binomial random graph, where these effects can be isolated and studied directly.

\section{Preliminaries}

\paragraph{Notation.}
We use mostly standard asymptotic notation. Apart from the classical $O$, $\Omega$, $\Theta$ and~$o$, for any two functions $f,g\colon\mathbb N\to (0,\infty)$, 
we write $f(n) \sim g(n)$ if $f(n)=(1+o(1))g(n)$. 

We use $\log n$ to denote the natural logarithm of $n$ and the following extension of the notation $[n]=\{1,\ldots,n\}$: For $a,b\in\N$ with $a\leq b$, we let $[a,b]=\{a,\ldots,b\}$. For $a\in\mathbb R$ and $\epsilon>0$, we let~$a\pm \epsilon$ denote the interval $[a-\epsilon,a+\epsilon]$.

For a vertex $v \in V$, we write $N(v)$ for the set of neighbors of $v$ in $\Gnp$, and $N[v] = N(v)\cup \{v\}$. 
For any $Z \subseteq V$, let also $N(Z) = \bigcup_{v \in Z} N(v)$ and $N[Z] = N(Z)\cup Z$.

As is typical in the field of random graphs, for expressions that clearly have to be integers, we round up or down without specifying when this choice does not affect the argument.
Finally, throughout the paper, absolute constants are chosen for convenience and are not optimized.

\paragraph{Probabilistic tools.}\label{ssec prelim}
Appendix~\ref{ap A} collects the standard concentration and approximation estimates used throughout the paper, including Chernoff bounds, Slud's inequality, and quantitative bounds for differences of binomial random variables. 
Here, we single out one additional lemma whose role is more structural: it allows us to compare two competing labels even after conditioning on the performance of all other competitors.

As explained in Section~\ref{sec technical tools}, the setting of this lemma arises repeatedly in our analysis, for example, when counting the edges from a vertex $v$ towards the first-round basins (or slight modifications of these).
Two of these sets (say $\Omega_i,\Omega_j$) 
play a distinguished role, one larger than the other, and we wish to compare the probabilities that the vertex $v$ selects the label associated with each of them. The remaining competing sets are relevant only through the random variable $\max_{\ell\neq i,j} |\Omega_\ell\cap N(v)|$
denoted below by $M$, and by the number $\rho$ of such sets attaining this maximum.

The following lemma shows that this additional conditioning can only strengthen the relative advantage of the larger set. Equivalently, the ratio between the probability of selecting the label associated with the larger set and the probability of selecting the competing label is minimized when no additional competing sets are present.

\begin{lemma}[Monotonicity for tie-weighted binomial comparisons]\label{lem:bis}
Let $0<p<1$. Fix integers $n > n'\ge 0$.
Consider two independent random variables $X\sim\mathrm{Bin}(n,p)$ and $Y\sim\mathrm{Bin}(n',p)$.

For an integer $M\ge -1$ and a real number $\rho\ge 0$, define the (tie-weighted) probability that $X$ wins against $Y$ and these additional competitors by
\begin{align*}
p_X(M,\rho) & = \mathbb P(X > M, X > Y) + \tfrac{1}{2} \mathbb P(X = Y > M) 
\\
& \qquad +\tfrac{1}{\rho+1} \mathbb P(X = M > Y) + \tfrac{1}{\rho+2} \mathbb P(X = Y = M),
\end{align*}
and define $p_Y(M,\rho)$ similarly by exchanging $X$ and $Y$.
Then, for every integer $M\in [0,n]$,
\[
\frac{p_X(M,\rho)}{p_Y(M,\rho)}
\ge
\frac{p_X(-1,\rho)}{p_Y(-1,\rho)}
=
\frac{\mathbb P(X > Y) + \tfrac{1}{2} \mathbb P(X = Y)}
{\mathbb P(Y > X) + \tfrac{1}{2} \mathbb P(Y = X)}.
\]
\end{lemma}

The proof of Lemma~\ref{lem:bis} requires a significant technical effort and is given in Appendix~\ref{ap B}.

\section{Technical overview of the staged exploration}
\label{sec:technical-overview}

This section discusses the structural reductions and probabilistic estimates that drive the proof of Theorem~\ref{thm 1}; the detailed statements and proofs are deferred to Appendix~\ref{ap C}.

Throughout this overview, we fix $\eps\in(0,1/24)$ and work in the range 
$n^{5/8+\eps}\le np\ll n$. Set
\[
k=\Big\lceil 15p^{-2}(n^{-1}\log n)^{1/2}\Big\rceil.
\]
The starting point of the analysis of LPA's dynamics is to isolate the first $2k$ labels, 
whose influence can be tracked accurately through the first rounds of the process. 
The parameter $k$ is chosen to balance two competing requirements. 
On the one hand, $k$ must be chosen large enough so that almost all labels remaining after the first round are in $[k]$. 
On the other hand, $k$ must remain small enough to ensure that almost all vertices remain outside the neighborhood of the set initially labeled by $1,\ldots,2k$, thus ensuring that most of the edges are exposed in the third exploration round.
This choice ensures that 
the union of these neighborhoods has a predictable size, which is roughly $knp$, and that the first-round basin sizes can be estimated and compared with sufficient precision, ensuring that a.a.s.\ all but $o(n)$ vertices belong to level~$\Lvl{C}$.

Next, we define three \emph{levels}: $\Lvl{A}=\{v_1,\ldots,v_{2k}\}$, where $v_\ell$ is the vertex initially carrying label~$\ell$,
$\Lvl{B}=N(\Lvl{A})\setminus \Lvl{A}$ and
$\Lvl{C}=V\setminus(\Lvl{A}\cup \Lvl{B})$.
Thus, level~$\Lvl{B}$  consists of the vertices that obtain one of the first~$2k$ labels at the first round, while level~$\Lvl{C}$ contains the vertices whose interaction with these labels occurs through level~$\Lvl{B}$.

For $\ell\in[2k]$, the first-round basin of label~$\ell$, denoted $\Lvltl{B}{1}{\ell}$, consists of those vertices of level~$\Lvl{B}$ for which $v_\ell$ is the smallest-indexed neighbor in $\Lvl{A}$, that is,
\[
\Lvltl{B}{1}{\ell}
=
N(v_\ell)\setminus
\Big(
\Lvl{A}\cup \bigcup_{i=1}^{\ell-1}N(v_i)
\Big).
\]
The basin sizes are therefore the first quantitative trace of the tie-breaking rule: although all basins have order~$np$, 
the expectations of their sizes decrease with~$\ell$.

The following analysis has three components.
First, after two rounds, the only labels that remain relevant to the subsequent evolution are essentially the first $k$ labels.
Second, in the range $np\gg n^{2/3}$, the expected first-round basin sizes decrease systematically with the label, and this bias is strong enough to make label~$1$ dominant.
Third, in the range $n^{5/8+\eps}\le np\ll n^{2/3}$, the label of the largest basin is mainly driven by fluctuations; the proof then identifies this random largest basin and tracks the amplification of its advantage.
The next three subsections describe these stages in order.

\smallskip
Henceforth, we use the following notation: for $S\in\{\Lvl{A},\Lvl{B},\Lvl{C}\}$, $t\ge 0$, and a label $\ell$, we write~$\Lvltl{S}{t}{\ell}$ for the set of vertices in level~$S$ that carry label~$\ell$
after round~$t$; for instance, $\Lvltl{C}{2}{\ell}$ is the set of vertices in level~$\Lvl{C}$ carrying label~$\ell$ after the second round.
\subsection{Reduction after the first two rounds}
The initial stage of the proof is an essential reduction in the number of relevant labels. This is provided by 
a two-round label reduction lemma (Lemma~\ref{lemma 2}), which establishes that, if $np\ge n^{1/2+\eps}$, then, after the second round, a.a.s.\ every vertex carries a label in~$[k]$ except potentially those whose initial label belongs to
\[
[k+1,K] \qquad\text{where } K=\big\lceil (2\log n)/p\big\rceil.
\]
%
In fact, after the second round, a.a.s.\ each label $\ell\in [k+1,K]$
can only be carried by the vertex~$v_\ell$ to which it was initially assigned: this explains the role of the parameter~$k$. 
The point is not merely that many labels disappear, but that the (at most $K$) exceptional labels surviving after two rounds are confined to a set small enough to be negligible in all later comparisons.

The other key step in this preliminary stage is a bootstrap to consensus lemma (Lemma~\ref{claim 5}). 
It says that, if $np\ge n^{1/2+\eps}$ and, at some round $t$, at least $0.9n$ vertices carry the same label~$\ell$, then, after two additional rounds, a.a.s.\ every vertex carries label~$\ell$. This statement separates the difficult part of the proof from the final cleanup: once the staged exploration produces a label carried by all vertices in level~$\Lvl{C}$, and hence all but $o(n)$ vertices,
the last two rounds are then simple to analyze.

In the densest part of the parameter range, namely when $p\ge \sqrt{15}(n^{-1}\log n)^{1/4}$,
one has $k=1$. Again, the two-round label reduction lemma (Lemma~\ref{lemma 2}) then implies that after the second round all but $O((\log n)/p)=o(n)$ vertices carry label~$1$. The bootstrap to consensus lemma (Lemma~\ref{claim 5}) then promotes this near-consensus state to full consensus. This implies Corollary~\ref{cor:recover}, which recovers the high-density behavior and reduces the remaining work to the case $k\ge 2$.

\subsection{The drift-dominated and critical regimes}
We now focus on the scenario where
$np=\Omega(n^{2/3})$ and $np\le \sqrt{15}n^{3/4}(\log n)^{1/4}$.

This range contains two closely related but qualitatively distinct cases. When
$np\gg n^{2/3}$, the systematic decrease of the expected basin sizes with the
label index dominates the first-round fluctuations, and label~$1$ becomes the
leading basin a.a.s. In the critical window $np=\Theta(n^{2/3})$, this
deterministic bias and the random fluctuations have comparable size. The same
staged amplification mechanism still applies, but the label of the leading
basin is no longer deterministic and must be tracked within a small set of the first few labels. We first describe the cleaner drift-dominated argument and then
explain the modifications needed in the critical window.

In the subrange $np\gg n^{2/3}$, the deterministic bias can be quantified
directly. Indeed, $|\Lvltl{B}{1}{\ell}|\sim \mathrm{Bin}(n-2k,q^{\ell-1}p)$,
where $q=1-p$, so
\[
\mathbb E|\Lvltl{B}{1}{1}|-\mathbb E|\Lvltl{B}{1}{\ell}|
=
(1+o(1))(\ell-1)np^2.
\]
Thus, consecutive expected basin sizes differ by order $np^2$; 
this systematic separation of the basin means is the deterministic drift referred to above. Moreover, the natural fluctuation scale of each basin is~$(np)^{1/2}$. 
Hence, when $np\gg n^{2/3}$, this drift is asymptotically larger than the standard deviation scale.
The key technical result at this point is Lemma~\ref{claim E}, which turns the preceding heuristic comparison of scales into a uniform basin-gap statement: a.a.s.,
\[
\sLvltl{B}{1}{1}-\sLvltl{B}{1}{\ell}
\ge \tfrac{1}{1.4}(\ell-1)np^2
\qquad\text{for every $\ell\in [2,2k]$,}
\]
together with the a.a.s.\ concentration estimate $|\sLvltl{B}{1}{1}-\mathbb E \sLvltl{B}{1}{1}|\le n^{3/4}p^{5/4} \ll np^2$.
This is the point where the deterministic preference for smaller labels is converted into a quantitative first-round advantage for label~$1$.

The second and third rounds are analyzed through the Alternative Label Attribution Procedure (ALAP), which is the main device that makes the staged exploration possible. 
ALAP is coupled to LPA 
in a way that uses the same first-round basins $B_1(1),\ldots,B_1(2k)$ but it reveals and makes use only of the edges needed to determine the relevant label updates at each stage. 
We note that a substantial difference from the original process is that labels larger than $k$ do not appear in the ALAP, but we show that a.a.s.\ this procedure recovers all labels correctly in the third round of the process.

The ALAP is defined as follows.
In round~1, it exposes the information defining the basins $B_1(1),\ldots,B_1(2k)$.
In round~2, it attributes labels to vertices of $\Lvl{B}\cup\Lvl{C}$ based only on their edges to the first-round basins. 
In round~3, it reveals the previously unused edges inside level~$\Lvl{C}$. Thus, ALAP allows us to follow the same early evolution as LPA while preserving fresh randomness inside~$\Lvl{C}$ until the moment when third-round amplification is needed.

Before describing the second-round estimates, 
we identify a quantity that represents a uniform lower bound on the probability that a vertex chooses the larger of two competing basins in the relevant regime here. Specifically, $2\qc-1$ measures the bias toward the dominant basin in one comparison, where
\[
\qc
=
\tfrac12+\tfrac15\min\Big\{1,\tfrac12\sqrt{np^4}\Big\}.
\]

The main second-round estimate in the drift-dominated regime is Lemma~\ref{claim 3}. It shows that the first-round basin gap for label~$1$ 
leads to a quantitative second-round advantage inside level~$\Lvl{C}$, with an explicit geometric profile across the competing labels.
More precisely, after the second round of ALAP, a.a.s.,
\begin{equation}\label{snd round estimate}
\sLvltl{C}{2}{1}\ge \tfrac12(2\qc-1)n,
\quad
\sLvltl{C}{2}{1}-\sLvltl{C}{2}{\ell}
\ge
\frac12
\Big(
1-\Big(\frac{1-\qc}{\qc}\Big)^{\ell-1}
\Big)
\sLvltl{C}{2}{1}
\quad
\text{for all }\ell\in [2,2k],
\end{equation}
which quantifies the advantage of label~$1$.
This is the stage at which the monotonicity lemma for tie-weighted binomial comparisons  (Lemma~\ref{lem:bis}) and the binomial comparison estimates from the preliminaries enter the staged exploration in an essential way. 

Level~$\Lvl{B}$ is treated separately because some edges incident to level~$\Lvl{A}$ have already been exposed and may influence local comparisons. The relevant conclusion is that, after the second round, the contribution from level~$\Lvl{B}$ is too small to counteract the advantage of label 1 in level~$\Lvl{C}$.
In the drift-dominated regime, we show the concrete bound
\begin{equation}\label{eq:estimB}
\sLvltl{B}{2}{\ell}-\sLvltl{B}{2}{1}
\le 44np
\qquad
\text{for every }\ell\in[2,2k],    
\end{equation}
which is weaker than the level~$\Lvl{C}$ estimate from~\eqref{snd round estimate}.

The final part of the proof in this regime uses the still unexposed edges inside level~$\Lvl{C}$.
The second-round estimate~\eqref{snd round estimate} implies that, after round~2, label~$1$ has a surplus over every competing label inside level~$\Lvl{C}$. In particular, this surplus is at least a constant multiple of $(2\qc-1)\sLvltl{C}{2}{1}$.
We reserve a fixed fraction of this surplus:
\[
\mathfrak N=\tfrac{1}{10}(2\qc-1)\sLvltl{C}{2}{1} \ge \tfrac{1}{20}(2\qc-1)^2n.
\]
The fresh edges within level~$\Lvl{C}$, together with this set of reserved vertices, allow us to quantify the advantage of label~$1$ in level~$\Lvl{C}$: namely, a.a.s., for every vertex $u\in\Lvl{C}$,
\[
|N(u)\cap \Lvltl{C}{2}{1}|
\ge
|N(u)\cap \Lvltl{C}{2}{\ell}|+\mathfrak Np
\qquad
\text{for every }\ell\in[2,2k].
\]
The key third-round step is Lemma~\ref{ob:rd3level3}, which shows that, using~\eqref{eq:estimB}, this fresh level~$\Lvl{C}$ surplus forces every vertex in $\Lvl{C}$ to adopt label~$1$ in the third round. 
Lemma~\ref{claim 6} then combines this third-round takeover with the two additional rounds supplied by the bootstrap to consensus lemma (Lemma~\ref{claim 5}), completing the proof of consensus with label~$1$ in the regime $np\gg n^{2/3}$.

The critical window $np=\Theta(n^{2/3})$ lies at the point where the deterministic drift between early basin means and the random fluctuations have the same order. 
The drift-based argument is therefore modified by isolating a small initial set of labels. 
Labels outside this set are shown to have basins 
much smaller than the ones of the leading group with probability close to one, while inside the small set the largest basin remains genuinely random. 
The same second- and third-round amplification mechanism then selects and analyzes the largest basin from this 
small collection. 
This yields the transition-window statement: consensus still occurs, but the smallest label wins with probability bounded away from both zero and one.

\subsection{The fluctuation-dominated regime}
It remains to treat the case where $n^{5/8+\eps}\le np\ll n^{2/3}$.

In this range, the deterministic drift between the expected sizes of consecutive basins 
is smaller than the natural fluctuation scale. 
The label of the largest first-round basin is therefore no longer completely determined by label order. 
The 
analysis of this regime replaces the direct drift argument by an extreme-value analysis of a coupled sequence of independent binomial variables.

For each $\ell\in[2k]$, we consider the second-order approximation $z_\ell p$ of the mean of $\sLvltl{B}{1}{\ell}$, where
\[
z_\ell
=
n-(\ell-1)np+\tfrac12(\ell-1)(\ell-2)np^2.
\]
The novel coupling step is Lemma~\ref{lem approx}: the dependent basin sizes can be coupled to independent random variables
$Z_\ell\sim\mathrm{Bin}(z_\ell,p)$, 
for every $\ell\in[2k]$,
so that a.a.s.,
\[
\max_{\ell\in[2k]}
\left|
Z_\ell-\sLvltl{B}{1}{\ell}
\right|
\le (np)^{2/5}.
\]
This reduces the identification of the largest first-round basin to an order-statistics problem for independent but non-identically distributed binomial variables.

To describe the location and size of the leading basin, we first analyze the maximum of the proxy variables. Let $Z^{(1)}$ and $Z^{(2)}$ denote, respectively, the largest and second-largest values among the random variables $(Z_\ell)_{\ell\in[2k]}$. The maximum satisfies
\[
\frac{Z^{(1)}-np}{\sqrt{np\log(1/(np^3))}}
\xrightarrow{\mathbb P}1.
\]
Moreover, the largest and second-largest values of the sequence are typically both attained at indices in the initial segment $[k_*]$, where  
\[
k_*=\left\lceil\frac12\sqrt{\frac{\log(1/(np^3))}{2np^3}}\right\rceil.
\]
The relevant extreme-value estimates are Lemmas~\ref{lem:max}, \ref{lem:generalization}, and~\ref{lem fst snd max}, together with Corollary~\ref{cor fst snd max}. 
These results identify the scale and location of the 
proxy maximum, prove a gap between the largest and 
second-largest proxy values, and transfer this gap 
back to the actual first-round basins. 
In particular, letting
$\ell_1=\min\{\ell\in[k_*]:Z_\ell=\max_{j\in[k_*]}Z_j\}$,
a.a.s.\ 
\[
Z^{(1)}=\max_{\ell\in[k_*]}Z_\ell,
\qquad
Z^{(2)}=\max_{\ell\in[k_*]\setminus\{\ell_1\}}Z_\ell.
\]
Moreover, for
a suitably chosen $\eps_n\to 0$
and $\gamma_n=(np)^{1/2-\eps_n}$,
we obtain a.a.s.\ that
\[
\sLvltl{B}{1}{\ell_1}=\max_{\ell\in[2k]}\sLvltl{B}{1}{\ell},
\qquad
\sLvltl{B}{1}{\ell_1} \ge \gamma_n + 
\max_{\ell\in[2k]\setminus\{\ell_1\}}\sLvltl{B}{1}{\ell}.
\]
Thus, even though the winning basin is random, it is separated from all competitors by a margin slightly below the natural fluctuation scale $(np)^{1/2}$.

This random basin advantage is then amplified during the second round. For a vertex of level~$\Lvl{C}$, the comparison between label~$\ell_1$ and a competitor~$\ell$ is governed by independent edge counts into the corresponding first-round basins. Since the basin of~$\ell_1$ is larger by at least $\gamma_n$, the expected edge-count advantage in this comparison is at least $\gamma_n p$. 
Moreover, the standard deviation scale of the relevant edge counts is of order $\sqrt{np^2}$, since each first-round basin has size of order $np$. 
Thus, the 
probability bias in favor of label~$\ell_1$ compared to any other label is at least of order
\[
\frac{\gamma_n p}{\sqrt{np^2}}
=
p^{1/2}(np)^{-\eps_n}.
\]
At this point, the two-round reduction from Lemma~\ref{lemma 2} allows us to restrict attention to labels in~$[k]$: all labels in $[k+1,2k]$ are confined to a negligible exceptional set and cannot affect the subsequent comparisons.

The essential amplification statement inside level~$\Lvl{C}$ is Lemma~\ref{claim 3'}. It gives, a.a.s.,
\begin{equation}\label{eq:Cadvantage}
\sLvltl{C}{2}{\ell_1}\ge \frac{n}{4k},   
\qquad
\max_{\ell\in[k]\setminus\{\ell_1\}}\sLvltl{C}{2}{\ell}
\le
\big(
1-c_2p^{1/2}(np)^{-\eps_n}
\big)
\sLvltl{C}{2}{\ell_1},
\end{equation}
for some absolute constant $c_2>0$. This is the fluctuation-regime analog of the level~$\Lvl{C}$ dominance estimate for label~$1$
in the drift-dominated regime.

Controlling the second-round contribution from level~$\Lvl{B}$ is the main decoupling step below~$n^{2/3}$.
The difficulty is that some edges incident to level~$\Lvl{A}$ have already been exposed during the construction of the first-round basins, and these exposed edges may affect second-round comparisons inside level~$\Lvl{B}$. 
For each competitor $\ell\in[k]\setminus\{\ell_1\}$, we therefore separate level~$\Lvl{B}$ into a good part and a small exceptional part. 
The good part consists of the union of those first-round basins whose defining vertices in level~$\Lvl{A}$ are 
adjacent to neither~$v_{\ell_1}$ nor $v_\ell$, and which are not themselves the basins of $\ell_1$ or $\ell$. 
On this good part, the comparison between labels $\ell_1$ and~$\ell$ is still governed by fresh edges into the two relevant basins. Lemma~\ref{lem Level B} establishes the corresponding advantage for label~$\ell_1$ on the good part, while Lemma~\ref{lem snd Level B} shows that the exceptional part is too small to change the comparison.
Together with Lemma~\ref{lem technical}, these results imply that,
for every vertex $v\in\Lvl{C}$ and every competitor $\ell\in[k]\setminus\{\ell_1\}$, a.a.s.,
\[
|N(v)\cap \Lvltl{B}{2}{\ell}|
-
|N(v)\cap \Lvltl{B}{2}{\ell_1}|
\le
400n^{3/2}p^{7/2+\eps/100}.
\]
This bound is precisely what is needed to ensure that the level~$\Lvl{B}$ contribution cannot offset the advantage already present inside level~$\Lvl{C}$ (quantified in~\eqref{eq:Cadvantage}).

Finally, the unrevealed edges inside $\Lvl{C}$ amplify the second-round level~$\Lvl{C}$ advantage during the third round. Conditional on the second-round label classes in $\Lvl{C}$, for each fixed $v\in\Lvl{C}$ and $\ell\in [k]$, the random variables
$|N(v)\cap \Lvltl{C}{2}{\ell}|$
are binomial. The gap supplied by Lemma~\ref{claim 3'} gives label~$\ell_1$ an expected surplus over every competitor that is much larger than the possible loss coming from level~$\Lvl{B}$. 
A union bound over all vertices of $\Lvl{C}$ and all competing labels then shows that a.a.s.\ every vertex in level~$\Lvl{C}$ adopts label~$\ell_1$ in the third round, both in the ALAP and in the original LPA. 
Since level~$\Lvl{C}$ contains $n-o(n)$ vertices, the bootstrap to consensus lemma (Lemma~\ref{claim 5}) completes the passage from this near-consensus state to full consensus.

\section{Concluding remarks}
Beyond determining the true consensus threshold for LPA on $\Gnp$---in particular, closing (on both sides) the theoretical gap towards the empirically suggested scale $np=\Theta(n^{1/5})$---a further research direction is to understand how much of the behavior observed here is specific to~$\Gnp$.

In particular, it is natural to ask whether the transition at average degree $np=\Theta(n^{2/3})$ reflects a structural feature of the update rule or an artifact of the independence structure of~$\Gnp$. It would be interesting to determine whether an analogous transition persists for LPA on random regular graphs, on configuration models with heavy-tailed degree sequences, or under small perturbations of the first-round tie-breaking rule. 
The configuration model case is particularly interesting, since high-degree vertices---absent in $\Gnp$ but ubiquitous in real networks---may introduce new mechanisms for amplifying or suppressing the tie-breaking bias. Understanding which parts of the argument survive these variants would help separate what is intrinsic to the dynamics from what is tied to the binomial random graph.

Finally, in our analysis of the regime $np\gg n^{5/8+\eps}$, the label carried by the vertices in the largest basin typically needed only a few steps to spread throughout the entire graph. 
An exciting question remains: is there a regime where typically consensus is reached but the surviving label is not the one with the largest first-round basin?

\section*{Acknowledgments.} 
The authors thank Ravi Sundaram for calling to their attention the lack of a complete mathematically rigorous understanding of label propagation algorithms, Yoshiharu Kohayakawa for referring us to~\cite{KLPSS19}, and an anonymous referee for carefully inspecting the entire paper and for pointing out an important mistake in a previous version.

The authors acknowledge the use of Gemini Pro and ChatGPT Plus to improve the grammar, style, and clarity of this manuscript, as well as to verify its mathematical correctness.
The authors reviewed and edited all outputs to ensure accuracy and take full responsibility for the content. 

\bibliography{Bib}
\bibliographystyle{plain}

\appendix

\section{Appendix}

\subsection{Probabilistic tools}\label{ap A}
In this section, we collect two standard results and establish a few other probabilistic tools that are used repeatedly throughout the paper.
We begin with two classical concentration inequalities for binomial random variables. We then establish quantitative estimates for differences of binomial random variables with close means. 

We first recall a standard form of Chernoff's inequality that we shall frequently use (see, e.g., Theorem~2.1 in~\cite{JLR00}).
\begin{lemma}\label{lem chernoff}
Let $X \sim\mathrm{Bin}(n,p)$ be a random variable with binomial distribution with parameters $n$ and~$p$
and $\varphi:[-1,\infty)\to\mathbb R$ be such that $\varphi(t)=(t+1)\log(t+1)-t$. Then, for all $t\geq 0$,
\begin{align*}
\Prob(X-\mathbb E X\ge t) &\le \exp\left(-\varphi\left(\frac{t}{\mathbb E X}\right)\cdot\mathbb E X\right) \le \exp\left(-\frac{t^2}{2(\mathbb E X + t/3)}\right), \\
\Prob(X-\mathbb E X\le -t) &\le \exp\left(-\varphi\left(-\frac{t}{\mathbb E X}\right)\cdot\mathbb E X\right) \le \exp\left(-\frac{t^2}{2 \mathbb E X}\right).
\end{align*}
\end{lemma}

The next result complements Chernoff's inequality by providing a lower bound on the upper tail of a binomial random variable. We state the version due to Slud~\cite{Slu77}, which is particularly suited to the range of parameters considered in this paper.
Set
\begin{equation*}
\Phi(t) = \int_{-\infty}^{t} \frac{1}{\sqrt{2\pi}}\exp\left(-\frac{x^2}{2}\right) \mathrm{d}x,
\qquad \text{ for all $t\in\mathbb R$.}
\end{equation*}
To avoid over-cluttering formulas, henceforth we adopt the typical convention of denoting $1-p$ by~$q$.

\begin{lemma}[see Lemma~2.1 in~\cite{Slu77}]\label{lem slud}
Let $X \sim\mathrm{Bin}(n,p)$ be a random variable with binomial distribution with parameters $n$ and
$p \le 1/4$. Then, for every $t\in [0, n - 2np]$,
$$\Prob(X\ge \mathbb E X + t)\ge 1 -  \Phi\left(\frac{t}{\sqrt{np\onemp}}\right).$$
\end{lemma}

The next tool provides quantitative control over the difference between two independent binomial random variables with the same success probability but different means. Such comparisons arise naturally throughout the paper when tracking how small advantages between competing labels evolve across successive rounds of the dynamics. In our applications, the binomial random variables associated with two competing labels may have means whose difference is small relative to the relevant standard deviation scale. Standard concentration inequalities are not sufficiently precise in this setting. The lemma below gives a quantitative lower bound showing that this advantage is nevertheless visible in the comparison probability~$\Prob(X_1-X_2\ge M)$. 
The proof combines Berry--Esseen's normal approximation with elementary properties of Gaussian random variables.
\begin{lemma}\label{lem difference binomials}
Fix $a_1, a_2\in \mathbb N$ and $p\in (0,1/2]$ such that $a_1\ge a_2\ge 1$. Let $X_1\sim\mathrm{Bin}(a_1, p)$ and $X_2\sim\mathrm{Bin}(a_2, p)$ be two independent random variables. Then, there exists an absolute constant $\zeta>0$ such that, for every $M \in \mathbb R$,
\begin{equation}\label{eq 2.3.1 main}
\mathbb P(X_1-X_2\ge M) \ge \Phi\bigg(\frac{(a_1-a_2)p-M}{\sqrt{(a_1+a_2)p\onemp}}\bigg)-\frac{2\zeta}{\sqrt{a_2p}}.
\end{equation}
\end{lemma}
\begin{proof}
By the normal approximation of the binomial distribution (see Berry-Esseen's inequality~\cite{Esseen}), if $X\sim\mathrm{Bin}(a,p)$, then for all $x\in\mathbb R$,
\[
\Big|\mathbb P\bigg(\frac{X-ap}{\sqrt{ap\onemp}}\le x\bigg) - \Phi(x)\Big| \leq \frac{\zeta(p^2+q^2)}{\sqrt{ap\onemp}}
\leq \frac{\zeta}{\sqrt{ap}},
\]
where $\zeta>0$ is an absolute constant.
Let $Z_1$ and $Z_2$ be independent standard normal random variables,
independent also of $X_1$ and $X_2$.
Then,
\[
\mathbb P(X_1-X_2\geq M)
= \mathbb P\bigg(\frac{X_1-a_1p}{\sqrt{a_1p\onemp}}\geq \frac{M+X_2-a_1p}{\sqrt{a_1p\onemp}}\bigg)
\ge -\frac{\zeta}{\sqrt{a_1p}}+\mathbb P\bigg(Z_1\geq \frac{M+X_2-a_1p}{\sqrt{a_1p\onemp}}\bigg).
\]
Set $\Delta=(a_1-a_2)p$ and $\sigma=\sqrt{(a_1+a_2)p\onemp}$.
Since $Z_1\geq \frac{M+X_2-a_1p}{\sqrt{a_1p\onemp}}$ if and only if $Z_1\sqrt{a_1p\onemp}-M+a_1p\ge X_2$, we get
\begin{align*}
\mathbb P(X_1-X_2\geq M)
& \ge
-\frac{\zeta}{\sqrt{a_1p}}+
\mathbb P\bigg(\frac{Z_1\sqrt{a_1p\onemp}-M+\Delta}{\sqrt{a_2p\onemp}}\geq \frac{X_2-a_2p}{\sqrt{a_2p\onemp}}\bigg)
\\ &
\ge -\frac{\zeta}{\sqrt{a_1p}}-\frac{\zeta}{\sqrt{a_2p}} + 
\mathbb P\bigg(\frac{Z_1\sqrt{a_1p\onemp}-M+\Delta}{\sqrt{a_2p\onemp}}\ge Z_2\bigg).
\end{align*}
Using the fact that $a_1\ge a_2$ and that $Z_1\sqrt{a_1p\onemp}-Z_2\sqrt{a_2p\onemp}$ is a normally distributed random variable with mean $0$ and variance $\sigma^2=(a_1+a_2)p\onemp$, we conclude that
\[
\mathbb P(X_1-X_2\ge M) \ge -\frac{2\zeta}{\sqrt{a_2p}}
+ 1-\Phi\bigg(\frac{M-\Delta}{\sigma}\bigg)
= \Phi\bigg(\frac{\Delta-M}{\sigma}\bigg)-\frac{2\zeta}{\sqrt{a_2p}},
\]
and so the claimed tail bound holds.
\end{proof}
The preceding lemma gives a Berry--Esseen lower bound in a general form. In the sequel, we will mainly use it in the asymptotic regime where $p=o(1)$, the quantities $a_2p$ and $(a_1-a_2)p$ both diverge, and the threshold $M$ is fixed. The following lemma records the consequences of the previous estimate in the form needed later. The first bound is deliberately coarser than the one given in Lemma~\ref{lem difference binomials}, but it makes explicit the bias in favor of the binomial variable with larger mean and is easier to apply. 
The second bound shows that, in particular, the probability of a tie is asymptotically smaller than the bias in favor of the larger binomial variable.

\begin{lemma}[Binomial comparison lemma]\label{lem difference binomials tie}
Fix $a_1 = a_1(n), a_2 = a_2(n)\in \mathbb N$ and $p =p(n) = o(1)$ such that $1\ll a_2\le a_1$ and $\min\{a_2, a_1-a_2\}p\to \infty$ as $n\to \infty$. Let $X_1\sim\mathrm{Bin}(a_1, p)$ and $X_2\sim\mathrm{Bin}(a_2, p)$ be two independent random variables. Then, for any fixed constant $M\in\mathbb{R}$ and $n$ large enough,
\begin{equation}\label{eq 2.3.2}
\Prob(X_1 - X_2 \ge M) \ge \frac{1}{2} + \frac{1}{5} \min\left\{1, \frac{(a_1-a_2)p}{\sqrt{(a_1+a_2) p}}\right\}.
\end{equation}
Moreover, for any fixed constant $m\in \mathbb Z$,
\begin{equation*}
\Prob(X_1 - X_2 = m) = o \Big( \Prob(X_1 - X_2 \ge m) - \Prob(X_1 - X_2 < m) \Big).
\end{equation*}
\end{lemma}
\begin{proof}
The proof is based on Lemma~\ref{lem difference binomials} and is divided into two cases depending on how close the value of the $\Phi$-term in~\eqref{eq 2.3.1 main} is to $1/2$. This case distinction gives rise to the minimum in \eqref{eq 2.3.2}.

Set $\Delta=(a_1-a_2)p$ and $\sigma=\sqrt{(a_1+a_2)p\onemp}$.
Let $\xi>0$ be a fixed constant.
Since $p=o(1)$, Lemma~\ref{lem difference binomials} applies for all sufficiently large $n$.
If $\frac{M-\Delta}{\sigma}\le -\xi$, since~$a_2p\to\infty$ as $n\to\infty$ and $\xi$ is a constant, Lemma~\ref{lem difference binomials} gives
\[
\mathbb P(X_1-X_2\geq M)\geq 1-1.01\cdot\Phi(-\xi)
\]
for all sufficiently large $n$.
On the other hand, suppose that $\frac{M-\Delta}{\sigma}> -\xi$. Since $M$ is fixed and~$\Delta\to\infty$ as $n\to\infty$, we have $\Delta-M>0$ for all sufficiently large $n$. Thus, using also that~$\onemp=1-p=1-o(1)$,
\[
1-\Phi\bigg(\frac{M-\Delta}{\sigma}\bigg)\geq \frac12+\frac{e^{-\xi^2/2}}{\sqrt{2\pi}}\frac{\Delta-M}{\sigma}
\ge \frac12+0.99\cdot\frac{e^{-\xi^2/2}}{\sqrt{2\pi}}\frac{(a_1-a_2)p}{\sqrt{(a_1+a_2)p}}.
\]
Observing that
\[
\left.\frac{1}{\sqrt{a_2p}}\right/\frac{(a_1-a_2)p}{\sqrt{(a_1+a_2)p}} = \frac{1}{\sqrt{(a_1-a_2)p}}\sqrt{\frac{2}{(a_1-a_2)p}+\frac{1}{a_2p}} \to 0\quad\text{when $n\to\infty$,}
\]
we obtain, again for all sufficiently large $n$, that
\[
\mathbb P(X_1-X_2\ge M) \ge \frac12+0.98\cdot\frac{e^{-\xi^2/2}}{\sqrt{2\pi}}\frac{(a_1-a_2)p}{\sqrt{(a_1+a_2)p}}.
\]
Summarizing,
\[
\mathbb P(X_1-X_2\geq M) \geq \min\bigg\{1-1.01\cdot\Phi(-\xi),\tfrac12+0.98\cdot\frac{e^{-\xi^2/2}}{\sqrt{2\pi}}\frac{(a_1-a_2)p}{\sqrt{(a_1+a_2)p}}\bigg\}.
\]
Taking $\xi=1$, observing that $0.98\cdot e^{-\xi^2/2}/\sqrt{2\pi}\approx 0.2371>\tfrac{1}{5}$ and verifying in a table of values for the cdf of the standard normal distribution that $\Phi(-1)\approx 0.1587$ (hence, $1-1.01\cdot\Phi(-1)\approx 0.8397>\frac12+\tfrac{1}{5}$) establishes the first bound in the lemma's statement. 

It remains to prove the estimate on point probabilities. This final assertion says that, on the scale at which $X_1$ has a genuine advantage over $X_2$, the probability of an exact equality is negligible. From~\eqref{eq 2.3.2}, we get that
\[
\mathbb P(X_1-X_2\geq m)-\mathbb P(X_1-X_2<m)
= 2\mathbb P(X_1-X_2\geq m)-1
\ge \frac{2}{5}\min\bigg\{1,\frac{(a_1-a_2)p}{\sqrt{(a_1+a_2)p}}\bigg\}.
\]
On the other hand, since the mode $m_*$ of $X_1$ is either $\lfloor (a_1+1)p\rfloor$ or $\lceil(a_1+1)p\rceil-1$, from Stirling's approximation one can deduce that
\[
\mathbb P(X_1-X_2=m)\leq \mathbb P(X_1=m_*)
= (1+o(1))\frac{1}{\sqrt{2\pi a_1p\onemp}}
= O\bigg(\frac{1}{\sqrt{(a_1+a_2)p}}\bigg).
\]
As $(a_1-a_2)p\to\infty$ and $a_2p\to\infty$, the last two displayed inequalities yield the last part of the lemma: if $(a_1-a_2)p\le \sqrt{(a_1+a_2)p}$,
then the ratio of $\mathbb P(X_1-X_2=m)$ to $\mathbb P(X_1-X_2\geq m)-\mathbb P(X_1-X_2<m)$ is
$O(1/((a_1-a_2)p))=o(1)$, whereas otherwise it is $O(1/\sqrt{(a_1+a_2)p})=o(1)$.
\end{proof}

The ideas underlying the proof of the previous result can be adapted to establish the following two refinements, which will be useful later.
\begin{remark}\label{rem binomials'}
If instead $(a_1-a_2)p=O(1)$, while $p=o(1)$ and $a_2p \to \infty$, in Lemma~\ref{lem difference binomials tie}, it is still possible to show that
\begin{equation*}
\Prob(X_1>X_2) + \frac12 \cdot \Prob(X_1=X_2)=\frac12+\Omega\bigg(\frac{(a_1-a_2)p}{\sqrt{a_2p}}\bigg).
\end{equation*}
\end{remark}
\begin{proof}
Let $X'\sim\mathrm{Bin}(a_2,p)$ and $Y\sim\mathrm{Bin}(a_1-a_2,p)$ be random variables independent of each other and of $X_2$.
Since $X'+Y\sim\mathrm{Bin}(a_1,p)$ and $X_1 \sim\mathrm{Bin}(a_1,p)$, we have
\[
\Prob(X_1>X_2) + \frac12 \cdot \Prob(X_1=X_2)
\ge \Prob(X'>X_2)+\frac12 \cdot \Prob(X'=X_2)+\frac12 \cdot \Prob(X'=X_2-1)\Prob(Y\ge 1).
\]
The first two terms on the right-hand side above sum up to $\frac12$.
Moreover, since $(a_1-a_2)p=O(1)$, we have  $\Prob(Y\ge 1)=1-q^{a_1-a_2}=\Theta((a_1-a_2)p)$.
To conclude, we recall that, for every integer $m\in a_2 p\pm \sqrt{a_2 pq}$ one has that $\Prob(X'=m)$ and $\Prob(X_2=m+1)$ are $\Omega(1/\sqrt{a_2 p})$.
Since there are $\Theta(\sqrt{a_2p})$ such integers,
\[
\Prob(X'=X_2-1)\geq \sum_{m\in a_2 p\pm\sqrt{a_2 pq}}
\Prob(X'=m)\Prob(X_2=m+1)
= \Omega\bigg(\frac{1}{\sqrt{a_2p}}\bigg).\qedhere
\]
\end{proof}

\subsection{Proof of monotonicity lemma for tie-weighted binomial comparisons}\label{ap B}
In this subsection, we give the delayed proof of the monotonicity for tie-weighted binomial comparison lemma (Lemma~\ref{lem:bis}). The proof has three main steps. First of all, we restrict our attention to showing the inequality
\begin{equation}\label{eq:targetineq}
\frac{p_X(M,\rho)}{p_Y(M,\rho)}\ge  \frac{p_X(M-1,\rho)}{p_Y(M-1,\rho)}.
\end{equation}
To do this, we expand each of the above terms and factor out powers of $1-p$. 
This allows us to decompose~\eqref{eq:targetineq} into a small number of polynomial inequalities in~$x=p/(1-p)$.
We then verify each of the latter algebraically by simple comparisons of the coefficients.

We work with an integer $M\in [0,n]$ and under the convention that, when a fraction has a positive numerator and denominator $0$, it is equal to infinity.

\paragraph{Step 1: Reduction to adjacent thresholds.}
By rewriting the claimed inequality in the equivalent form
\[\frac{p_X(M,\rho)}{p_X(-1,\rho)}\ge \frac{p_Y(M,\rho)}{p_Y(-1,\rho)},\]
we observe that it is sufficient to show the lemma for $n' = n-1$ (which we do assume): indeed, to deduce the inequality for smaller values of $n'$, it is enough to apply induction.

We are going to show that, for every $M\in [0,n]$,
\begin{equation}\label{eq:1'}
\frac{p_X(M,\rho)}{p_Y(M,\rho)}\ge \frac{p_X(M-1,\rho)}{p_Y(M-1,\rho)}.  
\end{equation}
Note that the conclusion of the lemma is obtained by consecutive applications of~\eqref{eq:1'}.

When $\rho=0$ and $M=0$, \eqref{eq:targetineq} holds with equality, so we may exclude this case below.
For the rest of the proof, fix $\rho\ge 0$ and write
\[
a_M=p_X(M,\rho),
\qquad b_M=p_Y(M,\rho),
\qquad \Delta a_M=a_{M-1}-a_M,
\qquad \Delta b_M=b_{M-1}-b_M.
\]
Thus, using that $a_{M-1}>a_M$ and $b_{M-1}>b_M$ for all integers $M\in [0,n-1]$, the desired monotonicity step \eqref{eq:1'}
is equivalent to
\begin{equation}\label{eq:1}
\frac{a_M}{b_M}\ge \frac{\Delta a_M}{\Delta b_M}.
\end{equation}
The advantage of this formulation is that $a_M=a_n+\sum_{j=M+1}^{n}\Delta a_j,$ and $b_M=b_n+\sum_{j=M+1}^{n}\Delta b_j$.
Thus, to prove that $a_M/b_M$ dominates the current drop ratio $\Delta a_M/\Delta b_M$, it is enough to show that all later drop ratios as well as $a_n/b_n$ dominate $\Delta a_M/\Delta b_M$, namely
\begin{equation}\label{eq:2}
\frac{a_n}{b_n}\ge \frac{\Delta a_{M}}{\Delta b_{M}}  
\end{equation}
and 
\begin{equation}\label{eq:3}
\forall i\in [n-M],\ \frac{\Delta a_{M+i}}{\Delta b_{M+i}} \ge \frac{\Delta a_{M+i-1}}{\Delta b_{M+i-1}}.    
\end{equation}
Indeed, by assuming~\eqref{eq:2} and~\eqref{eq:3} and using the elementary fact that if $r\le \alpha_i/\beta_i$ for all $i$ and all $\beta_i>0$, then $r\le(\sum_i\alpha_i)/(\sum_i\beta_i)$, 
the inequalities above imply~\eqref{eq:1}.

Now, note that~\eqref{eq:2} is clear since $b_{n}=0$ while all other terms are positive real numbers.
Thus, we turn our attention to~\eqref{eq:3}.

\paragraph{Step 2: Polynomial representation of the adjacent drops.}
Set $N = M+i-1$.
Then, the left-hand side of~\eqref{eq:3} rewrites as
\footnotesize 
\[\frac{\frac{1}{\rho+1}\mathbb P(X=N>Y)+\frac{1}{\rho+2}\mathbb P(X=N=Y)+(1-\frac{1}{\rho+1})\mathbb P(X=N+1>Y)+(\frac{1}{2}-\frac{1}{\rho+2})\mathbb P(X=N+1=Y)}{\frac{1}{\rho+1}\mathbb P(Y=N>X)+\frac{1}{\rho+2}\mathbb P(Y=N=X)+(1-\frac{1}{\rho+1})\mathbb P(Y=N+1>X)+(\frac{1}{2}-\frac{1}{\rho+2})\mathbb P(Y=N+1=X)},\]
\normalsize
while the right-hand side is given by
\footnotesize
\[\frac{\frac{1}{\rho+1}\mathbb P(X=N-1>Y)+\frac{1}{\rho+2}\mathbb P(X=N-1=Y)+(1-\frac{1}{\rho+1})\mathbb P(X=N>Y)+(\frac{1}{2}-\frac{1}{\rho+2})\mathbb P(X=N=Y)}{\frac{1}{\rho+1}\mathbb P(Y=N-1>X)+\frac{1}{\rho+2}\mathbb P(Y=N-1=X)+(1-\frac{1}{\rho+1})\mathbb P(Y=N>X)+(\frac{1}{2}-\frac{1}{\rho+2})\mathbb P(Y=N=X)}.\]

\normalsize
We divide each of the numerator and the denominator of the above two displays by $(1-p)^{2n-1}$. 
As a result, each of them can be seen as a polynomial of $x = \tfrac{p}{1-p}$.
For concreteness, we assume that the first display rewrites as $P(x)/Q(x)$ while  the second display rewrites as $R(x)/S(x)$ where
the polynomials $P,Q,R,S$ can be expressed in terms of the three 
basic building blocks, defined for~$m\ge 0$, as follows (with the convention that
empty sums are equal to zero and $\binom{r}{s}=0$ whenever $s\notin\{0,\ldots,r\}$):
\[
U_m(x)=\binom{n}{m}\sum_{i=0}^{m-1}\binom{n-1}{i}x^{m+i},
\;\;
V_m(x)=\binom{n-1}{m}\sum_{i=0}^{m-1}\binom{n}{i}x^{m+i},
\;\;
C_m(x)=\binom{n}{m}\binom{n-1}{m}x^{2m}.
\]
We set $U_m=V_m=C_m=0$ for $m<0$. 
The polynomials $U_m$ and $V_m$ encode the strict-win contributions that appear in the adjacent drop ratios in~\eqref{eq:3}, while $C_m$ encodes the contribution of the tie event $X=Y=m$.

Then,
\[
P=\frac{U_{N}+\rho U_{N+1}}{\rho+1}
+\frac{2C_N+\rho C_{N+1}}{2\rho+4},
\qquad
Q=\frac{V_N+\rho V_{N+1}}{\rho+1}
+\frac{2C_N+\rho C_{N+1}}{2\rho+4},
\]
and
\[
R=\frac{U_{N-1}+\rho U_{N}}{\rho+1}
+\frac{2C_{N-1}+\rho C_N}{2\rho+4},
\qquad
S=\frac{V_{N-1}+\rho V_N}{\rho+1}
+\frac{2C_{N-1}+\rho C_N}{2\rho+4}.
\]
Thus, proving the adjacent drop-ratio inequality is equivalent to proving $PS\ge QR$, which in turn is equivalent to
\begin{equation}\label{eq:polyrho}
(\rho+1)^2(2\rho+4)PS\geq (\rho+1)^2(2\rho+4)QR.
\end{equation}
Expanding both sides of the preceding inequality,  using the definitions of $P,Q,R,S$, the common term
\[
\frac{(\rho+1)^2(2C_N+\rho C_{N+1})(2C_{N-1}+\rho C_N)}{2\rho+4}
\]
appears on both sides and cancels. After this cancellation, the remaining inequality is a comparison between two polynomials in $\rho$. Thus, it is enough to compare the coefficients of $\rho^3,\rho^2,\rho$, and $\rho^0$. These four coefficientwise inequalities are precisely:
\begin{equation*}\label{J1}
2U_{N+1}V_{N}+U_{N+1}C_{N}+V_{N}C_{N+1}\ge 2V_{N+1} U_{N}+V_{N+1}C_{N}+U_{N}C_{N+1},
\tag{$J1$}
\end{equation*}
\vspace{-1ex}
\begin{equation*}\label{J2}\tag{$J2$}
\begin{split}
& 2U_{N}V_{N}+U_{N}C_{N}+2U_{N+1}V_{N-1}+4U_{N+1}V_{N}+2U_{N+1}C_{N-1} 
\\
& \qquad\qquad +U_{N+1}C_{N}+V_{N-1}C_{N+1}+2V_{N}C_{N}+V_{N}C_{N+1} 
\\
& \qquad \geq 2V_{N}U_{N}+V_{N}C_{N}+2V_{N+1}U_{N-1}+4V_{N+1}U_{N}+2V_{N+1}C_{N-1}\\
& \qquad\qquad\qquad +V_{N+1}C_{N}+U_{N-1}C_{N+1}+2U_{N}C_{N}+U_{N}C_{N+1},
\end{split}
\end{equation*}
\begin{equation*}\label{J3}\tag{$J3$}
\begin{split}
& 2U_{N}V_{N-1}+4U_{N}V_{N}+2U_{N}C_{N-1}+U_{N}C_{N}+4U_{N+1}V_{N-1} \\
& \qquad\qquad+2U_{N+1}C_{N-1}+2V_{N-1}C_{N}+V_{N-1}C_{N+1}+2V_{N}C_{N} \\
& \qquad\geq 2V_{N}U_{N-1}+4V_{N}U_{N}+2V_{N}C_{N-1}+V_{N}C_{N}+4V_{N+1}U_{N-1} 
\\ & \qquad\qquad\qquad +2V_{N+1}C_{N-1}+2U_{N-1}C_{N}+U_{N-1}C_{N+1}+2U_{N}C_{N},
\end{split}
\end{equation*}
\begin{equation*}\label{J4}\tag{$J4$}
4U_{N}V_{N-1}+2U_{N}C_{N-1}+2V_{N-1}C_{N}\ge 4V_{N}U_{N-1}+2V_{N}C_{N-1}+2U_{N-1}C_{N}.
\end{equation*}
It now remains to verify each of \eqref{J1}--\eqref{J4} algebraically.
The calculations below are included to make this verification completely explicit. 
The recurring pattern is that every coefficient splits into a symmetric part and an asymmetric part; the symmetric part of every coefficients turns out to be positive and to dominate the (potentially negative) contribution of the asymmetric part.

\paragraph{Step 3: Difference inequalities.}
We first prove the following auxiliary inequality:
\begin{equation*}\label{eqn I}
U_mV_{m-1}\ge V_mU_{m-1}
\tag{$I(m)$}
\end{equation*}
and use it to derive inequalities~\eqref{J1}--\eqref{J4}. 

For $m=0,1$ the claim is immediate, so assume $m\ge 2$. We show that all coefficients of the polynomial $U_mV_{m-1}-V_mU_{m-1}$ are non-negative.
First, we consider the terms $x^{2m-1+k}$ with~$k\in [0,m-2]$. For any such $k$, the coefficient in front of $x^{2m-1+k}$ rewrites
\begin{align}\label{eq:pos1}
\sum_{i,j\ge 0,\, i+j=k} \bigg(\binom{n}{m} \binom{n-1}{i} \binom{n-1}{m-1} \binom{n}{j} - \binom{n-1}{m} \binom{n}{i} \binom{n}{m-1} \binom{n-1}{j}\bigg).
\end{align}
Then, by the inequality $\binom{n}{m}\binom{n-1}{m-1}\geq \binom{n-1}{m}\binom{n}{m-1}$ and the symmetry of the sum with respect to~$i,j$, the positivity of the last expression follows.
In order to prepare the terrain for the case $k\in [m-1,2m-3]$, we further transform~\eqref{eq:pos1}.
Using that, for every $\ell\ge 0$, $\tbinom{n-1}{\ell} = \tfrac{n-\ell}{n}\tbinom{n}{\ell}$, we obtain that the above coefficient rewrites
\begin{align*}
\sum_{i,j\ge 0,\, i+j=k} \binom{n}{m} \binom{n}{i} \binom{n}{m-1} \binom{n}{j} \frac{(n-i)(n-m+1)-(n-j)(n-m)}{n^2},
\end{align*}
which, thanks to the inherent symmetry between $i$ and $j$ above, rewrites
\begin{align*}
\frac{1}{2} \sum_{i,j\ge 0,\, i+j=k} \binom{n}{m} \binom{n}{i} \binom{n}{m-1} \binom{n}{j} \frac{(n-i+n-j)((n-m+1)-(n-m))}{n^2},
\end{align*}
or equivalently
\begin{align}\label{eq:pos2}
\frac{1}{2} \sum_{i,j\ge 0,\, i+j=k} \binom{n}{m} \binom{n}{i} \binom{n}{m-1} \binom{n}{j} \frac{2n-k}{n^2} > 0.
\end{align}
Next, we turn to the coefficients in front of the terms $x^{2m-1+k}$ for $k\in [m-1,2m-3]$ where $i$ could be equal to $m-1$ but $j$ cannot, hence the loss of the symmetry used above.
Nevertheless, using that the term corresponding to $i=m-1, j=k-(m-1)$ is the only term without symmetric counterpart, we obtain that the coefficient in front of $x^{2m-1+k}$ rewrites as a sum of two terms. Following~\eqref{eq:pos2}, the first term contains the `symmetric part' of the sum, which is equal to
\begin{align*}
\frac{1}{2} \sum_{i,j\le m-2,\, i+j=k} \binom{n}{m} \binom{n}{i} \binom{n}{m-1} \binom{n}{j} \frac{2n-k}{n^2},
\end{align*}
and further rewrites as
\begin{align}\label{eq:I11}
\binom{n}{m}\binom{n}{m-1}\binom{n}{m-1}\binom{n}{k-m+1} \sum_{i,j\le m-2,\, i+j=k} \frac{\tbinom{n}{i}\tbinom{n}{j}}{\tbinom{n}{m-1}\tbinom{n}{k-m+1}} \frac{n-k/2}{n^2}.
\end{align}
The second term is
\begin{align*}
\binom{n}{m}\binom{n-1}{m-1}^2\binom{n}{k-m+1} - \binom{n-1}{m}\binom{n}{m-1}^2\binom{n-1}{k-m+1},
\end{align*}
which further rewrites as
\begin{align}
& \binom{n}{m}\binom{n}{m-1}^2\binom{n}{k-m+1} \frac{(n-m+1)^2-(n-m)(n+m-k-1)}{n^2}\nonumber\\
& \qquad =
\binom{n}{m}\binom{n}{m-1}^2\binom{n}{k-m+1} \frac{1-(n-m)(2m-k-3)}{n^2}.\label{eq:I12}
\end{align}
Finally, using that minimizing the expression $\tbinom{n}{i}\tbinom{n}{j}$ under the constraints $i,j\le m-2$ and $i+j=k$ requires $|i-j|$ to be largest possible (that is, $i=m-2$ and $j=k-m+2$ or vice versa), we derive that $\tbinom{n}{i}\tbinom{n}{j}\ge \tbinom{n}{m-2}\tbinom{n}{k-m+2}\ge \tbinom{n}{m-1}\tbinom{n}{k-m+1}$. 
This yields the inequality
\[\sum_{i,j\le m-2,\, i+j=k} \frac{\tbinom{n}{i}\tbinom{n}{j}}{\tbinom{n}{m-1}\tbinom{n}{k-m+1}} \frac{n-k/2}{n^2}\ge \frac{(2m-k-3)(n-k/2)}{n^2}\ge \frac{(2m-k-3)(n-m)}{n^2},\]
which shows that the sum of~\eqref{eq:I11} and~\eqref{eq:I12} is positive and finishes the proof of~\eqref{eqn I}.

\vspace{1em}

Before we continue with the proof of \eqref{J1}--\eqref{J4}, we simplify the expressions $U_{m}-V_{m}$. Using that, for every $\ell\ge 0$, $\binom{n-1}{\ell} = \tfrac{n-\ell}{n}\binom{n}{\ell}$, we obtain that
\begin{equation}\label{prlm:P1minusQ1}
U_{m}-V_{m} = \sum_{i=0}^{m-1} \bigg(\binom{n}{m}\binom{n-1}{i} 
- \binom{n-1}{m}\binom{n}{i}\bigg) x^{m+i} = \sum_{i=0}^{m-1} \binom{n}{m}\binom{n}{i}\frac{m-i}{n}x^{m+i}.
\end{equation}

\paragraph{Showing \eqref{J1}.}
By~\eqref{eqn I} with $m=N+1$, we already know that $U_{N+1}V_{N}\geq V_{N+1}U_{N}$, so it remains to show that 
\begin{align}\label{eqn:j2prime1st}
(U_{N+1}-V_{N+1})C_{N}-(U_{N}-V_{N})C_{N+1} \geq 0.
\end{align}
By~\eqref{prlm:P1minusQ1} with $m=N$ and $m=N+1$, the desired inequality is equivalent to
\begin{align*}
& \sum_{i=0}^{N}\binom{n}{N+1}\binom{n}{i}\binom{n}{N}\binom{n-1}{N}\frac{N+1-i}{n}
x^{3N+i+1} \\
& \qquad - \sum_{i=0}^{N-1}\binom{n}{N}\binom{n}{i}\binom{n}{N+1}\binom{n-1}{N+1}\frac{N-i}{n}
x^{3N+i+2}\ge 0.
\end{align*}
By ignoring the summand for $i=0$ in the first sum above and reindexing the second sum by replacing its index with $i-1$, so that the new index satisfies $i\in[1,N]$, the last display is implied by
\[
\sum_{i=1}^{N}\bigg(\binom{n}{N+1}\binom{n}{i}\binom{n}{N}\binom{n-1}{N}-\binom{n}{N}\binom{n}{i-1}\binom{n}{N+1}\binom{n-1}{N+1}\bigg)\frac{N+1-i}{n}
x^{3N+i+1}
\geq 0,
\]
which is implied by the fact that, for every $i\in [N]$,
\[
\frac{\binom{n}{i}}{\binom{n}{i-1}}\cdot
\frac{\binom{n-1}{N}}{\binom{n-1}{N+1}}
= \frac{n-i+1}{i}\cdot \frac{N+1}{n-1-N}
\geq \frac{n-N+1}{N}\cdot\frac{N+1}{n-1-N} > 1.
\]

\paragraph{Showing \eqref{J2}.}
By~\eqref{eqn I} with $m=N+1$ and $m=N$, we get $U_{N+1}V_N\ge V_{N+1}U_N$
and~$U_{N}V_{N-1}\ge V_{N}U_{N-1}$.
By multiplying these two inequalities, and using the zero conventions at the boundary, we also have $U_{N+1}V_{N-1}\ge V_{N+1}U_{N-1}$.
Thus, in order to prove \eqref{J2}, it suffices to show that
\[
(U_{N}-V_{N})C_{N}+(U_{N+1}-V_{N+1})(2C_{N-1}+C_{N})-(U_{N-1}-V_{N-1})C_{N+1}-(U_{N}-V_{N})(2C_{N}+C_{N+1}) \geq 0.
\]
Using~\eqref{eqn:j2prime1st}, it is enough to show that
\begin{align}\label{eqn:j2prime2nd} 
(U_{N}-V_{N})(C_{N}-2C_{N})+2(U_{N+1}-V_{N+1})C_{N-1}-(U_{N-1}-V_{N-1})C_{N+1} \geq 0.
\end{align}
To prove~\eqref{eqn:j2prime2nd}, observe that~\eqref{prlm:P1minusQ1} and a change of variable ensures that
\begin{align}
(U_{N}\!-V_{N})(C_{N}-2C_{N}) = -(U_{N}\!-V_{N})C_{N}
&= -\sum_{i=0}^{N-1}\binom{n}{N}\binom{n}{i}\frac{N-i}{n}\binom{n}{N}\binom{n-1}{N}x^{3N+i}\nonumber\\
&= -\sum_{j=1}^{N}\binom{n}{N}\binom{n}{j-1}\frac{N+1-j}{n}\binom{n}{N}\binom{n-1}{N}x^{3N+j-1}.\label{eq:simA1B2}
\end{align}
Furthermore,~\eqref{prlm:P1minusQ1} with $m=N+1$ and $m=N-1$  ensure respectively that
\begin{align}
2(U_{N+1}-V_{N+1})C_{N-1}  & = 2\sum_{i=0}^{N} \binom{n}{N+1}\binom{n}{i}\frac{N+1-i}{n}\binom{n}{N-1}\binom{n-1}{N-1}x^{3N+i-1},\label{eq:simA1B2+}\\
(U_{N-1}-V_{N-1})C_{N+1} & = \sum_{i=0}^{N-2} \binom{n}{N-1} \binom{n}{i} \frac{N-1-i}{n}\binom{n}{N+1}\binom{n-1}{N+1}x^{3N+i+1}\nonumber\\
& = \sum_{j=2}^{N} \binom{n}{N-1} \binom{n}{j-2} \frac{N+1-j}{n}\binom{n}{N+1}\binom{n-1}{N+1}x^{3N+j-1}.\label{eq:simA1B2++}
\end{align}
So, using our convention that $\binom{b}{a}=0$ if $a<0$, by combining~\eqref{eq:simA1B2},~\eqref{eq:simA1B2+} and~\eqref{eq:simA1B2++}, the coefficient in front of $x^{3N+i-1}$ in~\eqref{eqn:j2prime2nd} is non-negative if and only if
\begin{align}\label{eq:finalJ2}
& 2\binom{n}{N+1}\binom{n}{i}\binom{n}{N-1}\binom{n-1}{N-1} -\binom{n}{N-1}\binom{n}{i-2}\binom{n}{N+1}\binom{n-1}{N+1}
\\ 
& \qquad\qquad
- \binom{n}{N}\binom{n}{i-1}\binom{n}{N}\binom{n-1}{N}\ge 0. \nonumber
\end{align}
For $i=0$, the inequality is obvious.
For $i\in [N]$, observe that
\[
\frac{\binom{n}{i}}{\binom{n}{i-1}}
\cdot
\frac{\binom{n}{N-1}}{\binom{n}{N}}
\cdot \frac{\binom{n}{N+1}\binom{n-1}{N-1}}{\binom{n}{N}\binom{n-1}{N}}
= \frac{n-i+1}{i}\cdot\frac{N}{n-N+1}\cdot \frac{N}{N+1}
\geq \frac{N}{N+1},
\]
and
\[
\frac{\binom{n}{N-1}}{\binom{n}{N}}
\cdot\frac{\binom{n}{i-2}}{\binom{n}{i-1}}
\cdot\frac{\binom{n}{N+1}}{\binom{n}{N}}
\cdot\frac{\binom{n-1}{N+1}}{\binom{n-1}{N}}
= \frac{N}{n-N+1}\cdot\frac{i-1}{n-i+2}\cdot\frac{n-N}{N+1}\cdot\frac{n-N-1}{N+1}
\leq \frac{i-1}{N+1}\leq \frac{N-1}{N+1}.
\]
By combining the last two inequalities, it follows that the left-hand side of~\eqref{eq:finalJ2} dominates
\[\bigg(\frac{2N}{N+1}-\frac{N-1}{N+1}-1\bigg)\binom{n}{N}\binom{n}{i-1}\binom{n}{N}\binom{n-1}{N} = 0,\]
which finishes the proof of \eqref{J2}.

\paragraph{Showing \eqref{J3}.}
As observed in the proof of \eqref{J2}, 
we have $U_{N}V_{N-1}\ge V_{N}U_{N-1}$ and $U_{N+1}V_{N-1}\ge V_{N+1}U_{N-1}$.
Thus, in order to prove \eqref{J3}, it suffices to show that
\[
(U_{N}-V_{N})(2C_{N-1}+C_{N})+2(U_{N+1}-V_{N+1})C_{N-1}-(U_{N-1}-V_{N-1})(2C_{N}+C_{N+1})-2(U_{N}-V_{N})C_{N} \geq 0.
\]
Using~\eqref{eqn:j2prime2nd}, it is enough to show that
\begin{equation}\label{eq:finalsym}
2(U_N-V_N)C_{N-1}-2(U_{N-1}-V_{N-1})C_N \ge 0.
\end{equation}
For $N=0$,~\eqref{eq:finalsym} is immediate from the zero conventions. For $N\ge 1$, after dividing by $2$, it is exactly~\eqref{eqn:j2prime1st} with $N$ replaced by $N-1$.

\paragraph{Showing \eqref{J4}.}
This inequality follows by combining~\eqref{eqn I} with $m=N$ and~\eqref{eq:finalsym}.

\section{Staged exploration and proof of the main theorem}\label{ap C}

The proof of Theorem~\ref{thm 1} has three stages. First, we show that after two rounds almost all labels have disappeared, and only smaller labels remain relevant. Second, we analyze the basins created by the first round and identify a largest basin. Third, we show that the advantage of this basin is amplified until consensus is reached.

We now introduce the decomposition used to carry out this program. Fix $\eps\in(0,1/24)$ and assume that $n^{5/8+\eps}\le np\ll n$. Recall from the proof outline that we partition the vertex set~$V$ into \emph{levels}. In more detail, level~$\Lvl{A}$ consists of the vertices~$v_1,\ldots,v_{2k}$, where $v_\ell$ is the vertex whose initial label is~$\ell$ and
\begin{equation}\label{eq:defk}
k = k(n) = \Big\lceil 15 p^{-2}(n^{-1} \log n)^{1/2} \Big\rceil,
\end{equation}
level~$\Lvl{B}$ consists of the neighbors of vertices in $A$ (that is, $\Lvl{B} = N(\Lvl{A}) \setminus \Lvl{A}$), and level~$C$ consists of all remaining vertices (that is, $\Lvl{C} = V \setminus (\Lvl{A} \cup \Lvl{B})$). 

The parameter $k$ is chosen large enough so that all but very few vertices have a label in $[k]$ after two rounds.
At the same time, $k$ is small enough so that the union of the neighborhoods of the vertices in $A$ has predictable size (of order $knp$) and overlaps among these neighborhoods remain negligible at the scale relevant to the argument.

The constant $15$ in the definition of $k$ here is not optimized; it is chosen specifically so that its square provides a sufficiently large exponent in the Chernoff bounds to easily survive a union bound over $O(n^2)$ pairs later in the proof of Lemma~\ref{lemma 2}. For the sake of readability, throughout the manuscript, we adopt this same style of fixing specific constants in order to avoid introducing additional parameters and quantifiers.

Next, we recall one of our important notational conventions.
Whenever considering a set of vertices~$S\subseteq V$, the subset of its elements that have label $\ell$ after round $t$ will be denoted by $S_{t}(\ell)$. 
For a set of labels $W\subseteq [n]$, we let $S_t(W)$ be the subset of vertices of~$S$ that after round $t$ have a label that belongs to $W$, that is, $S_t(W)=\bigcup_{\ell\in W} S_t(\ell)$. In particular, for $W\subseteq [n]$, we will use $V(W)=V_0(W)$ to denote the set of vertices that initially have a label from $W$. 
For ease of reference, the following table summarizes the notation introduced.
\begin{center}
\renewcommand{\arraystretch}{1.2}
\begin{tabular}{ll}
\hline
Notation & Meaning \\
\hline
$\Lvl{A}$ & $\{v_1,\ldots,v_{2k}\}$ \\

$\Lvl{B}$ & $N(\Lvl{A})\setminus \Lvl{A}$ \\

$\Lvl{C}$ & $V\setminus(\Lvl{A}\cup \Lvl{B})$ \\

$S_t(\ell)$ & vertices of $S\subseteq V$ carrying label $\ell$ after round $t$ \\

$S_t(W)$ & vertices of $S$ carrying a label in $W$ after round $t$  \\

$V(W)$ & vertices of $V$ whose initial label belongs to $W$ \\
\hline
\end{tabular}
\end{center}

\subsection{The first two rounds: reduction of the set of relevant labels}\label{sec first two rounds}

In this section, we study what happens during the first two rounds of LPA. 
Our goal is to show that after round~2, a.a.s.\ almost all vertices carry a label in $[k]$.

To begin with, note that, after the first round, every vertex keeps its own label or switches to a smaller label. Hence, all vertices in level $\Lvl{A}$ carry labels in $[2k]$ after the first round. More importantly, every vertex in level $\Lvl{B}$ also carries a label in $[2k]$ after the first round: indeed, while initially every vertex $v\in B$ has a label in $[2k+1,n]$, it has a neighbor in level $\Lvl{A}$ with label in $[2k]$.

For each $\ell\in[2k]$, recall that $B_1(\ell)$ is the set of vertices in level $\Lvl{B}$ that receive label $\ell$ after the first round. We refer to $B_1(\ell)$ as the \emph{basin} of $v_\ell$. A vertex $v\in B$ belongs to this basin if and only if $v$ is adjacent to $v_\ell$ and is not adjacent to any of $v_1,\ldots,v_{\ell-1}$. Thus,
\begin{equation}\label{def:basin}
B_1(\ell)=N(v_\ell)\setminus\Big(A\cup \bigcup_{i=1}^{\ell-1}N(v_i)\Big).
\end{equation}
After the first round, the labels in $[2k]$ are the only labels that can appear on vertices in $\Lvl{A}\cup\Lvl{B}$, and the sets~$\Lvltl{B}{1}{1}, \Lvltl{B}{1}{2},\ldots, \Lvltl{B}{1}{2k}$, describe the basins created inside $\Lvl{B}$. The next lemma shows that, after one further round, the set of labels that can still play a significant role is much smaller: apart from a narrow exceptional interval of labels, all vertices carry labels in $[k]$.

\begin{lemma}[Two-round label reduction]\label{lemma 2}
Fix $\eps > 0$ and suppose that $np\ge n^{1/2+\eps}$. Fix $K = K(n) = \lceil 2(\log n)/p\rceil$.
Then, a.a.s., after the second round, each vertex $v_j$ with $j\in [n]\setminus [k+1,K-1]$ carries a label in~$[k]$. Moreover, a.a.s.\ every label $j$ between $k+1$ and $K-1$ may be carried only by the vertex~$v_j$ after the second round. 
\end{lemma}
The threshold $K$ in the statement above marks the point beyond which a vertex is extremely unlikely to retain its own label after two rounds. Indeed, for $j\ge K$, a vertex carrying label $j$ must avoid too many smaller labels in order for $j$ to remain competitive.

\begin{proof}
We first show that no vertex $u\neq v_{\ell}$ can receive a label $\ell>k$ after the second round.

Note that all vertices in $V([k]) \subseteq A$, as well as their neighbors, get a label from $[k]$ after the first round. 
In particular, the desired property holds for all vertices in $V([k])$. It is important to notice that we do not need to expose \emph{any} edges of $\Gnp$ to conclude this.

Next, we show that, for every two vertices $u, v_\ell \in V \setminus V([k])$ (so, in particular,~$\ell > k$), with probability $1 - o(n^{-2})$, the number of common neighbors of $v_1$ and $u$ is larger than the number of vertices in $N[u]$ carrying label~$\ell$ after the first round.
The exclusion of $\Lvltl{B}{1}{[k]}$ reflects the fact that vertices already assigned labels in $[k]$ after the first round cannot help label $\ell$ compete against the first $k$ labels during the second round. 
Therefore, whenever the above comparison holds, vertex $u\neq v_\ell$ cannot obtain label~$\ell$ after the second round. It follows that $u$ obtains label~$\ell$ after the second round with probability~$o(n^{-2})$.

On the one hand, the number of common neighbors of~$u$ and $v_1$ is a random variable $Y_1$ with distribution $\mathrm{Bin}(n-2, p^2)$. Moreover, the number of vertices in $N[u]$ carrying a label $\ell$ after round~1 
is dominated by~$Y_\ell+2$, where $Y_\ell$ is a binomial random variable with distribution $\mathrm{Bin}(n-2, \onemp^kp^2)$: indeed, except for vertices~$u$ and $v_\ell$ that could carry label $\ell$ after the first round, the remaining relevant vertices must be neighbors of $u$ (to influence its label), neighbors of $v_\ell$ (to obtain label $\ell$ after the first round), but not adjacent to any of $v_1,\ldots,v_k$, which happens with probability~$q^kp^2$.

Let $r$ be the expected advantage between the number of common neighbors of~$u$ and $v_1$
and $Y_\ell$, that is, 
\[
r = \mathbb E[Y_1-Y_\ell] = p^2 (n-2) (1-\onemp^k) = (1-o(1))np^3k,
\]
where in the last equality we used that 
$1-\onemp^k = 1-(1-p)^k = (1-o(1))pk$.
Clearly, 
\begin{equation}\label{eq:ineq Y}
\Prob(Y_1\le Y_\ell+2)
\le \;
\Prob(Y_1-\mathbb E Y_1\le -r/2+1) + \Prob(Y_\ell-\mathbb E Y_\ell\ge r/2-1).    
\end{equation}
Thus, by Chernoff's bound, the probabilities on the right-hand side above are bounded from above by
$\exp\left(-\frac{(r-2)^2/4}{2(\mathbb E Y_1+(r-2)/6)}\right)$ and 
$\exp\left(-\frac{(r-2)^2/4}{2(\mathbb EY_\ell+(r-2)/6)}\right)$, respectively. 
As a result,
\[
\Prob(Y_1\le Y_\ell+2)\le
2 \exp\Big(-\frac{(r-2)^2/4}{2(\mathbb E Y_1+(r-2)/6)}\Big)
=
2\exp\Big(-(1+o(1))\frac{n^2p^6k^2}{8np^2}\Big). 
\]
Using that $np^4k^2\ge 225\log n$, we conclude by a union bound that, for all pairs $(u, v_\ell)$ with $\ell > k$ and $u\neq v_\ell$, vertex $u$ does not get label $\ell$ after round 2.
Here we use the choice of $k$ which makes the expected gap $r$ large enough for the failure probability to be $o(n^{-2})$, so that a union bound over all pairs~$(u,v_\ell)$ is possible.

We now consider the only remaining way in which a label $\ell\ge K$ could survive: it might still be carried by its original vertex $v_\ell$. For this to happen, label $\ell$ must compete at $v_\ell$ with label $1$ during the second round.
Observe that $Z_\ell=|N(v_1)\cap N(v_\ell)|$ has distribution $\mathrm{Bin}(n-2, p^2)$, while $|N[v_\ell]\setminus N[V([\ell-1])]|$ is dominated by a random variable $W_\ell+1$ where $W_\ell$ is distributed according to $\mathrm{Bin}(n-\ell, pq^{\ell-1})$. Hence, using that for $\ell\ge K$,
\[\mathbb E Z_\ell = (n-2)p^2\gg npq^{\ell-1}+1 \geq \mathbb E[W_\ell + 1],\]
an argument similar to~\eqref{eq:ineq Y} with $\tfrac{1}{3} \mathbb E Z_\ell$ instead of $r$ shows that
\[\Prob(Z_\ell\le W_\ell+1)\le 2 \exp\Big(-\frac{(\tfrac{1}{3} \mathbb E Z_\ell)^2}{2(\mathbb E Z_\ell+\tfrac{1}{9}\mathbb E Z_\ell)}\Big) = \exp(-\Omega(np^2)).\]
Thus, a union bound implies that, for every $\ell\in [K, n]$, vertex $v_\ell$ does not carry label $\ell$ after the second round, as desired.
\end{proof}

At this point, we are able to recover fast convergence of the LPA when $k=1$. 
Once a label occupies 90\% of the vertices, the remaining graph structure forces that label to take over in at most two more rounds. The next lemma formalizes this final amplification step
(as usual, the numerical constants in the next lemma are not optimized but rather fixed for convenience).
\begin{lemma}[Bootstrap to consensus]\label{claim 5}
Fix $\eps > 0$ and suppose that $np\ge n^{1/2+\eps}$. Also, suppose that at some round $t$ of LPA, at least $0.9 n$ of the vertices in $\Gnp$ have the same label $\ell$. Then, after two more rounds, a.a.s.\ every vertex carries label $\ell$.
\end{lemma}
\begin{proof}
If more than $300p^{-1}$ vertices fail to adopt label $\ell$ at round $t+1$, then some set $S$ of size $s=|S|=300p^{-1}$ consists entirely of vertices that avoid label $\ell$.
Let $R\subseteq V$ of size $r=|R| = 0.9n$ and not necessarily disjoint from $S$ (one can think of $R$ as vertices having label $\ell$ at round $t$, and of~$S$ as vertices ``trying to avoid'' label $\ell$ at round $t+1$). 

We claim that, with probability $1-o(\tbinom{n}{s}^{-1}\tbinom{n}{r}^{-1})$, there is at least one vertex in $u\in S$ such that $|N(u)\cap R|>|N(u)\cap (V\setminus R)|+1$.
Such a vertex must adopt label~$\ell$ in round $t+1$; hence, the event that every vertex of $S$ avoids label~$\ell$ is contained in the complement of the event in the claim.

Observe that the number of potential edges between $S$ and $V\setminus R$ is
at most $|S|\cdot|V\setminus R|\leq 0.1sn$, and that the number of potential edges between $S$ and $R\setminus S$ is at least $s(r-s)\geq 0.8sn$ for large enough~$n$. Thus, the number of edges in $\Gnp$ between $S$ and $V\setminus R$ is dominated by~$X\sim\mathrm{Bin}(0.1sn,p)$ and the number of edges between $S$ and $R\setminus S$ dominates $Y\sim\mathrm{Bin}(0.8sn,p)$.

Now, suppose that every vertex of $S$ avoids label~$\ell$ at round $t+1$. Then, for each 
$u\in S$, the number of vertices in $N[u]$ carrying labels different from $\ell$ at round $t$ must be at least the number of vertices in $N[u]$ carrying label~$\ell$. Since every vertex of $R$ carries label~$\ell$ at round $t$, this implies $|N(u)\cap (V\setminus R)|+1 \ge |N(u)\cap R|$.
Summing this inequality over all $u\in S$ gives
\[
\sum_{u\in S}|N(u)\cap R|
\le
\sum_{u\in S}|N(u)\cap (V\setminus R)|+s.
\]
The left-hand side is at least the number of edges between $S$ and $R\setminus S$, while the sum in the right-hand side is at most twice the number of edges between $S$ and $V\setminus R$, because an edge with endvertices inside $S\cap (V\setminus R)$ may be counted twice in the sum over $u\in S$. 
Therefore, for the fixed pair $(S,R)$, the event that no vertex of $S$ has strictly more neighbors in $R$ than in $V\setminus R$ has probability at most $\mathbb{P}(2X+s\ge Y)$.
To finish the proof of the claim just note that, by Chernoff's bound, 
\begin{align*}
\mathbb{P}(2X+s\geq Y)
&\leq \big(\mathbb{P}(X\geq 0.2snp)+\mathbb{P}(Y\leq 0.5snp)\big)
\leq 2\exp\Big(-\frac{snp}{100}\Big) = o\bigg(\binom{n}{s}^{-1}\binom{n}{r}^{-1}\bigg).
\end{align*}

A union bound over the at most $\binom{n}{s}\binom{n}{r}$ pairs of sets $S,R$ with $|S|=s$ and $|R|=r$ implies that, a.a.s., for every such pair, at least one vertex in $S$ obtains label~$\ell$ at the next round. In particular, a.a.s.\ all but at most $300p^{-1}$ vertices carry label~$\ell$ at round $t+1$.

One more round forces all vertices to have label $\ell$.
Indeed, a.a.s., every vertex has degree $\Omega(np)\gg p^{-1}$, and most of the neighbors of every vertex carry label $\ell$ after round $t+1$.
Therefore, a.a.s.\ all vertices get label $\ell$ at round $t+2$.
\end{proof}
We now apply Lemma~\ref{lemma 2} in the range where $k=1$, that is when $p\ge \sqrt{15}\left(n^{-1}\log n\right)^{1/4}$.
In this regime, Lemma~\ref{lemma 2} implies that after the second round all but at most $O((\log n)/p)$ vertices carry label~$1$. Moreover,
$(\log n)/p=o(np)$,
and a.a.s.\ every vertex has degree $\Omega(np)$. Hence, after the second round, every vertex has strictly more neighbors carrying label~$1$ than vertices carrying any other label. Therefore, all vertices receive label~$1$ in the third round. 
\begin{corollary}\label{cor:recover}
Suppose that $np\ge \sqrt{15}(\log n)^{1/4}n^{3/4}$. Then, a.a.s.\ LPA attributes label $1$ to all vertices after three rounds.
\end{corollary}

The previous corollary proves Theorem~\ref{thm 1} throughout the range
$np\ge \sqrt{15}(\log n)^{1/4}n^{3/4}$.
It remains to treat the complementary range, where $k\geq 2$ and the competition among the first $k$ labels must be analyzed in more detail. Therefore, in the sequel we assume that $np\le \sqrt{15}(\log n)^{1/4}n^{3/4}$.

\subsection{\texorpdfstring{The regime $\Omega(n^{2/3}) = np\le \sqrt{15}(\log n)^{1/4}n^{3/4}$}{}}

From Section~\ref{sec first two rounds}, we know that after round~2, a.a.s.\ all but $O((\log n)/p)$ vertices receive a label in $[k]$. In this section, we prove that after round five, a.a.s.\ every vertex has label~$1$. Most of the work goes into showing that after round~3, a.a.s.\ every vertex in level~$C$ has label~$1$. 
Since levels~$A$ and~$B$ together contain only $o(n)$ vertices a.a.s., Lemma~\ref{claim 5} implies that only two additional rounds are needed for label~$1$ to spread to the whole graph.

In the statements below, we specify the required range of $p$ separately for each result, since these ranges sometimes differ. Although several estimates hold in greater generality, we state them only in the form needed in the sequel.

\subsubsection{First-round basin comparison}\label{sec:3.2.1}

This subsection collects the first-round estimates on basin sizes that will be used later. The main point is that the tie-breaking rule creates a deterministic drift in the expected basin sizes: 
smaller labels have larger basins in expectation.
We will show three basin estimates. The first (Lemma~\ref{claim E}) applies in the regime $np\gg n^{2/3}$ and identifies label~$1$ as having a robust first-round basin advantage. 
The second (Remark~\ref{rem E:IR}) adapts this estimate to the critical window $np=\Theta(n^{2/3})$, where only the first bounded number of labels remain unresolved. 
The third (Lemma~\ref{claim E'}) gives a stronger simultaneous concentration estimate for all basins. 

Since we focus our discussion on basin sizes, we recall that, by
definition (see~\eqref{def:basin}), a vertex belongs to $B_1(\ell)$ if it is
one of the $n-2k$ vertices not in level~$\Lvl{A}$ which has no neighbor among $v_1,\ldots,v_{\ell-1}$ but is adjacent to $v_\ell$. The latter event has probability $q^{\ell-1}p$, which implies that $\sLvltl{B}{1}{\ell}\sim\mathrm{Bin}(n-2k, \onemp^{\ell-1}p)$. Thus, for every $\ell\in [2k]$, since $\ell p \le 2kp = o(1)$ and $k = o(n)$, we have 
\begin{equation*}
\mathbb E \sLvltl{B}{1}{\ell}=(n-2k)\onemp^{\ell-1}p \sim np. 
\end{equation*}
Now, define
\begin{equation*}
\omega=\omega(n)=((np)^{1/2}\cdot np^2)^{1/2}=n^{3/4}p^{5/4}.
\end{equation*}
In particular, this choice guarantees that $(np)^{1/2} \ll \omega \ll np^2$ as long as $np\gg n^{2/3}$.

The estimates in this subsection are based on the latter asymptotic comparison. For consecutive labels, the deterministic gap between the mean of the basin sizes is of order $np^2$, while the natural fluctuation scale of each basin size is of order $(np)^{1/2}$. Thus, the regime $np\gg n^{2/3}$ is precisely the range in which the one-step drift created by tie-breaking dominates typical fluctuations. The parameter $\omega$ is chosen as an intermediate scale between these two quantities.

Next, we formally state and prove our first basin-size estimate.
\begin{lemma}\label{claim E}
Suppose that $n^{2/3} \ll np \le \sqrt{15} n^{3/4} (\log n)^{1/4}$. Then, the following basins-gap event holds a.a.s.:
\begin{equation*}
\evnt{E}{gp}
= \Big\{\text{$\forall\ell\in [2,2k], \sLvltl{B}{1}{1} - \sLvltl{B}{1}{\ell}\ge \tfrac{1}{1.4} (\ell-1) n p^2$}\}
\cap \{|\sLvltl{B}{1}{1} - \mathbb E\sLvltl{B}{1}{1}|\le \omega\Big\}.
\end{equation*}
\end{lemma}
The subscript $\mathrm{gp}$ emphasizes that this event records a global gap between the basin of label~$1$ and all other first-round basins, together with a concentration estimate for the first basin. 
\begin{proof}
Recall that $\sLvltl{B}{1}{\ell}$, for $\ell\in [2k]$, is a binomial random variable with mean $(n-2k)pq^{\ell-1}$. Hence, since $\ell p\le 2kp = o(1)$ and $k = o(n)$,
\begin{equation*}\mathbb E \sLvltl{B}{1}{1} - \mathbb E \sLvltl{B}{1}{\ell} = (n-2k)p(1-(1-p)^{\ell-1}) = (1+o(1)) (\ell-1) np^2.\end{equation*}

Since $1-\tfrac{1}{1.4}>1/5$, for $n$ large, the event
  $\sLvltl{B}{1}{1}-\sLvltl{B}{1}{\ell}\le \frac1{1.4}(\ell-1)np^2$
can occur only if either $\sLvltl{B}{1}{1}$ is below its mean by at least~$\frac1{10}(\ell-1)np^2$, or $\sLvltl{B}{1}{\ell}$ is above its mean by at least~$\frac1{10}(\ell-1)np^2$. 
Therefore, by Chernoff's bound,
\begin{align*}
&\Prob(\sLvltl{B}{1}{1} - \sLvltl{B}{1}{\ell}\le \tfrac{1}{1.4} (\ell-1) n p^2)\\
& \qquad \le 
\Prob(\sLvltl{B}{1}{1} - \mathbb E \sLvltl{B}{1}{1} \le - 
\tfrac{1}{10} (\ell-1) n p^2) + \Prob(\sLvltl{B}{1}{\ell} - \mathbb E \sLvltl{B}{1}{\ell}\ge \tfrac{1}{10} (\ell-1) n p^2)\\
& \qquad \le
\exp\left(-\frac{((\ell-1) n p^2/10)^2}{2\mathbb E \sLvltl{B}{1}{1}}\right) + \exp\left(-\frac{((\ell-1) n p^2/10)^2}{2(\mathbb E \sLvltl{B}{1}{\ell} + (\ell-1) n p^2/30)}\right)\\
& \qquad \le 
2\exp\left(-\frac{((\ell-1) n p^2)^2}{300np}\right) \\
& \qquad =
2\exp\left(-\frac{(\ell-1)^2}{300} np^3\right).
\end{align*}
Since $np^3\gg 1$, a union bound over the range $\ell\in [2,2k]$ shows that the first event in the intersection that determines $\evnt{E}{gp}$ holds a.a.s. For the second event therein, using Chernoff's bound and the fact that $\omega^2 = n^{3/2} p^{5/2}\gg \mathbb E \sLvltl{B}{1}{1}$ shows that it also holds a.a.s.\ finishing the proof.
\end{proof}

Lemma~\ref{claim E} is enough in the regime $np\gg n^{2/3}$, because the expected gap between consecutive basin sizes already dominates the fluctuation scale. In the critical window $np=\Theta(n^{2/3})$, this gap is only of the same order as the fluctuations. The same Chernoff argument still controls all labels except for a bounded initial segment, but the first few basin sizes require a separate comparison.

\begin{remark}\label{rem E:IR}
The proof of Lemma~\ref{claim E} also implies the following result. Suppose that $np = c n^{2/3}$ for some constant $c > 0$ (in particular, $np^3=c^3$ and $\omega^2=\Theta(\mathbb E \sLvltl{B}{1}{1})$). Then, for every $\eps\in (0,1)$ and every sufficiently large positive integer $L = L(\eps, c)$, the following basins-tail-gap event holds with probability at least $1-\eps$:
\begin{align*}
\evnt{E}{tl}_L = 
&\Big\{\text{$\forall\ell\in [L+1,2k],\sLvltl{B}{1}{1} - \sLvltl{B}{1}{\ell}\ge \tfrac{1}{1.4} (\ell-1) n p^2$}\}
\cap \{|\sLvltl{B}{1}{1} - \mathbb E\sLvltl{B}{1}{1}|\le L n^{1/3}\Big\}.  
\end{align*}
This is the critical-window tail-gap event: it controls all labels beyond the first $L$ labels, leaving only a comparison of a bounded number of leading basins. 

To prove the claim, just note that, for $np=cn^{2/3}$, the estimate in the proof of Lemma~\ref{claim E} becomes
\begin{equation*}
\Prob\big(
\sLvltl{B}{1}{1}-\sLvltl{B}{1}{\ell}
\le \tfrac1{1.4}(\ell-1)np^2
\big)
\le
2\exp\Big(-\frac{c^3}{300}(\ell-1)^2\Big).
\end{equation*}
Thus, the sum over $\ell\ge L+1$ can be made smaller than $\eps/2$ by choosing $L=L(\eps,c)$ sufficiently large. The concentration bound for $\sLvltl{B}{1}{1}$ follows similarly, since its standard deviation is of order~$n^{1/3}$.
\end{remark}

Remark~\ref{rem E:IR} reduces the critical-window difficulty to a finite-dimensional problem: after discarding all labels beyond a sufficiently large constant $L$, only the first $L$ basins may remain comparable. 
Lemma~\ref{lem:separate} resolves this remaining issue by showing that the differences in the sizes of these finitely many leading basins remain of order $\Omega(np^2)$
with probability arbitrarily close to one. The proof uses the following technical lemma, which gives the limiting joint distribution of finitely many color class sizes (encoding the small labels) within the critical window.

\begin{lemma}\label{lem:bin thinning}
  Fix $L\in \mathbb N$, a set of $L+1$ colors, and suppose that $\widehat p = \widehat p(n)$ is a real number satisfying
  $n\widehat p = (c+o(1))n^{2/3}$
  for some constant $c \in (0,\infty)$. Color the elements in $[n]$ independently so that every $j\in [n]$ obtains color $i\in [L]$ with probability $\widehat p$, and color $L+1$ with probability $1 - L \widehat p$. Denote by $X_i$ the number of vertices in color $i$, and set $Y_i = \tfrac{X_i - \widehat p n}{\sqrt{\widehat p (1-\widehat p) n}}$. Then,
\begin{equation*}
(Y_1,\ldots,Y_L)\xrightarrow[n\to \infty]{d}(N_1,\ldots,N_L),
\end{equation*}
where $(N_i)_{i=1}^L$ are i.i.d.\ normal random variables with mean $0$ and variance $1$.
\end{lemma}
\begin{proof}
The idea is to compare the multinomial coloring with a collection of independent binomial color classes. We first construct independent sets $\widehat{\cX}_1,\ldots,\widehat{\cX}_L$ with marginal probability $\widehat q$ chosen so that the probability of belonging to at least one of them is $L\widehat p$. We then break ties uniformly among the sets containing a given element. This produces the desired multinomial coloring. Since triple overlaps are negligible and pairwise overlaps contribute only a deterministic shift of order $n^{1/3}$, the centered and normalized color counts have the same limiting distribution as independent binomials.

Formally, define $(\widehat \cX_i)_{i=1}^L$ as $L$ independent subsets of $[n]$ where every number $j\in [n]$ belongs to~$\widehat \cX_i$ with probability 
\begin{equation*}\widehat q = 1 - (1-L\widehat p)^{1/L} = \widehat p + \tfrac12(L-1) \widehat p^2 + O(\widehat p^3),\end{equation*}
where we used that for every $\alpha > 0$, it holds that $(1+x)^{\alpha} = 1 + \alpha x + \tfrac12\alpha (\alpha-1) x^2 + O(x^3)$ as~$x\to 0$. 

Set $\widehat X_i = |\widehat \cX_i|$ and $\widehat Y_i = \tfrac{\widehat X_i - \widehat q n}{\sqrt{\widehat q (1-\widehat q) n}}$. Then, by the central limit theorem for independent binomial random variables,
\begin{equation}\label{eq:conv Ys}
(\widehat Y_1,\ldots,\widehat Y_L)\xrightarrow[n\to \infty]{d}(\widehat N_1,\ldots,\widehat N_L),
\end{equation}
where $(\widehat N_i)_{i=1}^L$ are i.i.d.\ normal random variables with mean 0 and variance 1. 

We construct the random variables $(X_i)_{i=1}^L$ from the sets $(\widehat \cX_i)_{i=1}^L$. For every integer~$j$ belonging to at least one of the sets $(\widehat \cX_i)_{i=1}^L$, associate a random variable $U_j$ that is uniformly distributed over $\{i: \text{ $j\in \widehat \cX_i$}\}$ so that $(U_j)_{j=1}^n$ are independent.  Then, for every $i\in [L]$, denote $\cX_i=\{j\in [n] \colon \text{$U_j$ exists and $U_j=i$}\}$ and $X_i = |\cX_i|$. Note that the probability to belong to $\cX_i$ is the same for all $i\in [L]$, and it is exactly $\tfrac{1}{L}\big(1 - (1-\widehat q)^L\big) = \widehat p$, so $(X_i)_{i=1}^L$ has the desired distribution.

The only difference between $\widehat X_i$ and $X_i$ comes from elements that belong to more than one of the sets $\widehat{\cX}_1,\ldots,\widehat{\cX}_L$. Elements belonging to at least three such sets have probability $O(\widehat q^3)=\widehat{p}^3+O(\widehat{p}^4)=O(n^{-1})$, and hence, by Markov's inequality, their total number is at most $n^{1/6}=o(n^{1/3})$ a.a.s.

It remains to account for pairwise overlaps. For every pair $i\neq j$, a.a.s., the number of elements belonging to $\widehat{\cX}_i\cap \widehat{\cX}_j$ and to no other $\widehat{\cX}_r$ is
\begin{equation*}
(1+o(1))\widehat q^2 n
= (1+o(1))\big(\widehat{p}^2+O(\widehat{p}^3)\big)n
=(1+o(1))c^2n^{1/3}.
\end{equation*}
Among these elements, half are assigned to color $i$ and half to color $j$, up to an $o(n^{1/3})$ error. Therefore, for each fixed $i$,
\begin{equation*}
\widehat X_i-X_i
=
\tfrac12(L-1)\widehat q^2n+o(n^{1/3})
=
\tfrac12(L-1)c^2n^{1/3}+o(n^{1/3}).
\end{equation*}
On the other hand,
\begin{equation*}
(\widehat q-\widehat p)n
=
\tfrac12(L-1)\widehat p^2n+O(\widehat p^3n)
=
\tfrac12(L-1)c^2n^{1/3}+o(n^{1/3}).
\end{equation*}
The last two displays show that $\widehat X_i-\widehat qn
=X_i-\widehat pn+o(n^{1/3})$ uniformly over $i\in[L]$,
Combining this with~\eqref{eq:conv Ys} and the fact that $\sqrt{\widehat q(1-\widehat q)n} = (1+o(1))\sqrt{\widehat p(1-\widehat p)n}=\Theta(n^{1/3})$ finishes the proof.
\end{proof}

\begin{lemma}\label{lem:separate}
Suppose that $np = c n^{2/3}$ for some constant $c > 0$. For every $\eps\in (0,1)$ and every positive integer $L$, there is $\delta = \delta(\eps, L, c) > 0$ such that the following 
basins-separation event holds with probability at least $1-\eps$: 
\begin{equation*}\evnt{E}{sp}_{L, \eps} = \Big\{\min_{i,j\in [L]\colon i\neq j} |\sLvltl{B}{1}{i} - \sLvltl{B}{1}{j}|\ge \delta np^2\Big\}.\end{equation*}
\end{lemma}
\begin{proof}
For every set $S\subseteq [L]$, denote by $X_S$ the number of vertices in $V\setminus \Lvl{A}$ that connect to all vertices in~$S$ and do not connect to the vertices in $[L]\setminus S$. If $S = \{i\}$ or $S = \{i,j\}$ for some $i,j\in [L]$, for simplicity of notation we denote $X_i$ and $X_{i,j}$ instead of $X_{\{i\}}$ and $X_{\{i,j\}}$, respectively.

Similarly to the proof of Lemma~\ref{lem:bin thinning}, we have that a.a.s.
\begin{equation}\label{eq:ge 3}
\max_{S\subseteq [L]: |S|\ge 3} X_S = O(n^{1/6}),   
\end{equation} 
and
\begin{equation}\label{eq:Xij}
X_{i,j} = (1+o(1))\mathbb E X_{i,j} = (1+o(1))c^2n^{1/3}.
\end{equation}
Moreover, for every $i\in [L]$, denote 
\begin{equation*}Y_i = \frac{X_i - (n-2k)pq^{L-1}}{\sqrt{(n-2k)pq^{L-1}(1-pq^{L-1})}}.\end{equation*}

Thus, by applying Lemma~\ref{lem:bin thinning} with $\widehat p = pq^{L-1}$ and $n-2k$ instead of $n$ (where color $L+1$ is given to the vertices connected to 0 or at least 2 vertices among $v_1,\ldots,v_\ell$)
\begin{equation}\label{eq:convY}
(Y_1,\ldots,Y_L)\xrightarrow[n\to \infty]{d}(N_1,\ldots,N_L),
\end{equation}
where $(N_i)_{i=1}^L$ are i.i.d.\ normal random variables with mean $0$ and variance 1.

Now, it remains to notice that  $\sLvltl{B}{1}{i} = \sum_{S\subseteq [L]: \min S = i} X_S$, for every $i\in [L]$.
The deterministic shift in the limiting vector comes from vertices adjacent to both $v_i$ and some later vertex $v_j$ with $j>i$: such vertices are assigned to the basin of $v_i$ because $i=\min\{i,j\}$. 
Moreover, using~\eqref{eq:ge 3} and~\eqref{eq:Xij}, the number of such vertices is
\begin{equation*}
\sum_{S\subseteq [L]: \min S=i,\ |S|\ge 2}X_S
=
\sum_{j=i+1}^L X_{i,j}+o(n^{1/3})
=
(L-i)c^2n^{1/3}+o(n^{1/3}).
\end{equation*}
Since
\begin{equation*}
\sqrt{(n-2k)pq^{L-1}(1-pq^{L-1})}
=(1+o(1))\sqrt c\,n^{1/3},
\end{equation*}
this contributes a deterministic shift of $(L-i)c^{3/2}$ in the limit.

For every $i\in [L]$, denote 
\begin{equation*}Z_i = \frac{\sLvltl{B}{1}{i} - (n-2k)pq^{L-1}}{\sqrt{(n-2k)pq^{L-1}(1-pq^{L-1})}} = Y_i + \sum_{S\subseteq [L]: \min S = i, |S|\ge 2} \frac{X_S}{\sqrt{(n-2k)pq^{L-1}(1-pq^{L-1})}}.\end{equation*}
Then, combining~\eqref{eq:ge 3},~\eqref{eq:Xij},~\eqref{eq:convY} and the fact that $\sqrt{(n-2k)pq^{L-1}(1-pq^{L-1})} = (1+o(1)) \sqrt{c}n^{1/3}$ implies that 
\begin{equation}\label{eq:converge Z}
(Z_1,\ldots,Z_L)\xrightarrow[n\to \infty]{d}(N_i+(L-i)c^{3/2})_{i=1}^L.   
\end{equation}

In particular, there is a $\delta=\delta(\eps,L,c) > 0$ such that 
\begin{equation*}\min_{i,j\in [L]: i\neq j} |N_i-N_j+(j-i) c^{3/2}|\ge 2\delta c^{3/2}\end{equation*}
with probability at least $1-\tfrac{\eps}{2}$, which combined with~\eqref{eq:converge Z}  implies that for all sufficiently large $n$, 
\begin{equation*}\min_{i,j\in [L]: i\neq j} |Z_i-Z_j|\ge \delta c^{3/2}\end{equation*}
holds  with probability at least $1-\eps$. (Note that the factor $2$ disappeared to take into account the error coming from the in distribution.) Coming back to $(\sLvltl{B}{1}{i})_{i=1}^L$ finishes the proof.
\end{proof}

\begin{remark}\label{rem:limit}
Suppose that $np=cn^{2/3}$ for some constant $c>0$.
Let $w$ be the smallest index maximizing $|B_1(i)|$
over $i\in[2k]$.
For every fixed $L\ge1$, let $w_L$ be the smallest
index maximizing $|B_1(i)|$ over $i\in[L]$.

Let $(N_i)_{i=1}^L$ and $(Z_i)_{i=1}^L$ be as in the proof of Lemma~\ref{lem:separate}.
Since the normalization defining $Z_i$ is the same
for all $i\in[L]$, the index $w_L$ is also the
smallest index maximizing $Z_1,\ldots,Z_L$.
It follows from~\eqref{eq:converge Z} that, for
every $i\in[L]$,
\[
\lim_{n\to\infty}\mathbb P(w_L=i)
=
\mathbb P\Big(
N_i-ic^{3/2}>
\max_{j\in[L]\setminus\{i\}}(N_j-jc^{3/2})
\Big),
\]
where the right-hand side is understood to be $1$
when $L=1$.
Moreover, for every fixed $i\ge1$, these limiting
probabilities are bounded away from zero uniformly
over $L\ge i$. 
In particular, the index of the maximal basin takes each of the values in $[L]$ with positive probability. 

\end{remark}

The next estimate is used in the upper part of the present density range. Unlike Lemma~\ref{claim E}, which only separates the first basin from the others, it gives simultaneous concentration for every basin around its own mean. This stronger uniform control will be needed when the second-round analysis compares the numbers of edges that vertices in level~$C$ send to several different basins.

\begin{lemma}\label{claim E'}
Suppose that $np\in [n^{3/4} (\log n)^{-1/2}, \sqrt{15} n^{3/4} (\log n)^{1/4}]$. Then, the following uniform basins-concentration event holds a.a.s.: 
\begin{equation*}
\evnt{E}{cn}_* = \{\text{for each $\ell\in [2k]$, $\sLvltl{B}{1}{\ell}\in \mathbb E \sLvltl{B}{1}{\ell}\pm \ell\omega$}\}.
\end{equation*} 
\end{lemma} 
\begin{proof}
Since $p\gg n^{-1/3}$,
\begin{equation*}
(np^{-1})^{1/4} k^{-1}  
= \Omega((n^3p^7)^{1/4}(\log n)^{-1/2})
= \Omega(n^{1/6} (\log n)^{-1/2}) \gg 1.
\end{equation*}
Hence, $\mathbb E \sLvltl{B}{1}{\ell} \sim np = (np^{-1})^{1/4} \omega \gg 2k\omega\ge \ell\omega$. 
Thus $\ell\omega=o(\mathbb E\sLvltl{B}{1}{\ell})$ uniformly in $\ell\in[2k]$.
Recalling that $\omega/(np)^{1/2}=(np^3)^{1/4}$ and applying Chernoff's bound we obtain, for $n$ large enough and uniformly in~$\ell \in [2k]$,
\begin{equation*}
\Prob(\sLvltl{B}{1}{\ell}\notin \mathbb E \sLvltl{B}{1}{\ell}\pm\ell\omega)\le 2\exp\Big(-\frac{(\ell \omega)^2}{2(\mathbb E \sLvltl{B}{1}{\ell} + \ell \omega)}\Big) \le \exp\Big(-\frac13 \ell^2 (np^3)^{1/2}\Big).
\end{equation*}
In particular, a union bound yields
\begin{align*}
\Prob(\exists \ell\in [2k], \sLvltl{B}{1}{\ell}\not\in \mathbb E \sLvltl{B}{1}{\ell}\pm \ell\omega)
\le \sum_{\ell=1}^{2k} \exp\Big(-\frac13 \ell^2 (np^3)^{1/2}\Big)
\le \exp\Big(-\frac14(np^3)^{1/2}\Big) =o(1),
\end{align*}
where in the last equality we used that $(np^3)^{1/2}\to\infty$ throughout the stated range.
This proves the lemma.
\end{proof}

\begin{remark}\label{rem E'} 
When $np\in [n^{3/4} (\log n)^{-1/2}, \sqrt{15} n^{3/4} (\log n)^{1/4}]$, on the basins-concentration event~$\evnt{E}{cn}_*$ we have, for~$n$ large and uniformly in $\ell \in [2k-1]$,
\begin{align*}
\sLvltl{B}{1}{\ell} - \sLvltl{B}{1}{\ell+1}
& \ge\; (\onemp^{\ell-1}p(n-2k)-\omega\ell) - (\onemp^{\ell}p(n-2k)+\omega(\ell+1))\\
& =\; \onemp^{\ell-1} p^2 (n-2k) - \omega(2\ell+1)\ge \tfrac{1}{1.4} np^2.
\end{align*}
\end{remark}

The remaining estimate controls the total size of level~$B$, rather than the individual basins. 
\begin{lemma}\label{claim F}
Suppose that $(n \log n)^{1/2} \ll np \le \sqrt{15} n^{3/4} (\log n)^{1/4}$. 
The following mid-level size event holds a.a.s.:
$$
\evnt{B}{sz} =\left\{\text{$\frac 43 \, knp \le \sLvl{B} \le \frac 83 \, knp$}\right\}.
$$ 
\end{lemma}
\begin{proof}
Since $(\log n/n)^{1/2}\ll p \le \sqrt{15} (\log n/n)^{1/4}$, we get that $kp = \Theta( p^{-1}(\log n/n)^{1/2} ) = o(1)$. 
Therefore, $1-q^{2k}=1-(1-p)^{2k}=(2+o(1))kp$.

As a result, since $\sLvl{B} \sim\mathrm{Bin}(n-2k, 1-\onemp^{2k})$ and $k = o(n)$, 
we have $\mathbb E\sLvl{B}=(n-2k)(1-q^{2k})=(2-o(1))knp$ and the lemma follows directly from Chernoff's bound.
\end{proof}

Note that Lemma~\ref{claim F} holds for a wider range of $np$ than needed in this section, and it will be used in the proof of both the first and the second point of Theorem~\ref{thm 1}.

\smallskip
The events introduced in this subsection will be used as first-round input for the analysis of the second round. Their role is to guarantee that the dominant basin is already separated from its competitors by a margin large enough to be amplified when vertices in level~$C$ compare their numbers of neighbors in the different basins. 
All these events depend only on the edges revealed in the first round, namely the edges that determine the basins inside level~$B$. 
This fact is important later, because the edges from level~$C$ to level~$B$ remain available for the second-round analysis.

For ease of future reference, the following table summarizes the events defined throughout this section.
\begin{center}
\renewcommand{\arraystretch}{1.4}
\begin{tabular}{p{0.16\textwidth}p{0.057\textwidth}p{0.71\textwidth}}
\hline
Name & Event & Meaning \\
\hline
Basins-gap & $\evnt{E}{gp}$ & 
$\displaystyle\forall\ell\in [2,2k], \sLvltl{B}{1}{1}{-}\sLvltl{B}{1}{\ell}\ge \tfrac{(\ell{-}1)n p^2}{1.4}$,
$|\sLvltl{B}{1}{1} - \mathbb E\sLvltl{B}{1}{1}|\le \omega$
\\
Basins-tail-gap & $\evnt{E}{tl}_L$ 
& $\forall\ell\in [L{+}1,2k], \sLvltl{B}{1}{1}{-}\sLvltl{B}{1}{\ell}\!\ge\!\tfrac{(\ell{-}1)n p^2}{1.4}$, 
$|\sLvltl{B}{1}{1} - \mathbb E\sLvltl{B}{1}{1}|\le Ln^{1/3}$
\\
Basins-separation & $\evnt{E}{sp}_{L,\eps}$ & $\displaystyle\min_{i,j\in [L], i\neq j}|\sLvltl{B}{1}{i} - \sLvltl{B}{1}{j}|\ge \delta np^2$
\\
Basins-concentration & $\evnt{E}{cn}_*$ & 
$\forall\ell\in [2k]$, $\sLvltl{B}{1}{\ell}\in \mathbb E \sLvltl{B}{1}{\ell}\pm \ell\omega$
\\
Mid-level size & $\evnt{B}{sz}$ & $\frac43knp\le \sLvl{B}\le \frac83 knp$
\\
\hline
\end{tabular}
\end{center}

\subsubsection{Consequences of the basin sizes: the second round}

The goal of this subsection is to convert the first-round basin gap established in Section~\ref{sec:3.2.1} into a more substantial advantage in favor of the most dominant label after the second round. We do this in three steps. 
First, we introduce ALAP, a coupled version of LPA that determines the first-round basins and the second-round labels of vertices in $B\cup C$ while exposing only the edges needed for these decisions. 
Second, we show that after round~2, the label carried by the largest basin is typically the most represented label inside level~$C$. Third, we prove a weaker but sufficient analogue for level~$B$, which will be needed when vertices in level~$C$ update in the third round.

We will mostly concentrate on the regime $n^{2/3} \ll np \le \sqrt{15} n^{3/4} (\log n)^{1/4}$. Modifications for the regime $np = \Theta(n^{2/3})$ are minor and essentially consist in the fact that the event $\evnt{E}{tl}_L$ (defined as in Remark~\ref{rem E:IR}) does not concern the basins of $v_2, \ldots, v_L$. We discuss these modifications in remarks after the corresponding lemmas for the first regime.

Next, we show that conditionally on $\evnt{E}{gp}\cap\evnt{B}{sz}$, a.a.s.\ after the second round, the number of vertices of label~$\ell$ in level~$C$ decreases as a function of~$\ell$. Here, Lemma~\ref{lemma 2} is crucial since it establishes that the label of almost all vertices~$v$ in level~$C$ can a.a.s.\ be attributed based only on the edges between $v$ and vertices in level~$B$. In fact, the vertices for which the above does not hold are so few that it will be convenient to think of them as very rare exceptions that, as we shall see, do not have any significant influence on LPA when $n$ is large.
In order to formalize this, we now define an alternative label propagation procedure.

\paragraph{The Alternative Label Attribution Procedure (ALAP).} 
A key component of the staged exploration framework is the Alternative Label Attribution Procedure (ALAP). 
Like LPA, ALAP assigns labels to the vertices. The difference lies in how these labels are determined. 
Whereas LPA updates the labels of all vertices
simultaneously using the fully revealed graph,
ALAP processes the vertices sequentially during
the first two rounds. When a vertex is considered, all incident edges satisfying the prescribed exposure rules are revealed simultaneously. 
The order in which vertices are processed, together with the choice of edges revealed, depends only on the partition of the vertices into the levels $A$, $B$, and $C$, and on the initial labels defining the set $A$. This selective exploration preserves enough unrevealed
randomness for the analysis of subsequent rounds.
In particular, the edges within level~$C$ remain
unexposed during the first two rounds, allowing
their use in concentration arguments for the third round.

The point of ALAP is not to define a different process for its own sake, but to expose only those edges that are needed to determine the first two rounds.

The procedure is formally defined as follows.
\begin{itemize}
\item In round 1, for each $i\in[2k]$, expose the edges from $v_i$ to the vertices outside $A$ that have not yet been claimed by any earlier basin $\Lvltl{B}{1}{1},\ldots,\Lvltl{B}{1}{i-1}$, namely to
$V\setminus\big(N(V([i-1]))\cup A\big)$.
Also expose all edges with both endpoints in $\Lvl{A}$. On the basis of these edges, attribute labels only to the vertices in~$\Lvl{A}\cup\Lvl{B}$, which coincide with the labels these vertices receive after round~$1$ in LPA.

\item In round~2, given a vertex $u\in \Lvl{B}\cup\Lvl{C}$, expose the edges from $u$ to $A\cup B$ which remained unrevealed after the previous round  and denote by $U\subseteq[2k]$ the set of labels such that $u$ has the same number of neighbors in $A_1(i)\cup \Lvltl{B}{1}{i}$ for every $i\in U$, and strictly fewer neighbors in $A_1(i)\cup \Lvltl{B}{1}{i}$ for every $i\in[2k]\setminus U$. Then, choose one label from $U$ uniformly at random and attribute it to $u$.

\item Finally, in round~3, expose the edges between vertices in $\Lvl{C}$ and let ALAP evolve according to the same rules as LPA on the set of labels present after round~2.
\end{itemize}

The usefulness of ALAP rests on two key properties.

First, by Lemma~\ref{lemma 2}, LPA and ALAP on $\Gnp$ can be coupled so that, a.a.s., only the vertices in $V([K]\setminus[2k])$ may receive different labels after the second round. This coupling is quite natural: to decide the label of a vertex in the second round of ALAP, run LPA for two rounds and attribute labels according to the same majority and tie-breaking rules but restricted to the neighbors with labels in $[2k]$ only.
We refer to such discrepancies (from the point of view of LPA) as \emph{mistakes} of ALAP. Later, we show that these mistakes do not affect the subsequent amplification of the dominant labels. 

Second, the edges exposed during the first round of ALAP determine the basins $(\Lvltl{B}{1}{i})_{i=1}^{2k}$, which coincide with the corresponding basins of LPA. Consequently, every event introduced in Section~\ref{sec:3.2.1} is measurable with respect to the information revealed during the first round of ALAP, and all results established there remain valid. From this point onward, we therefore use the notation and terminology introduced for LPA when referring to ALAP.

The next claim isolates an analytic comparison needed in the proof of Lemma~\ref{claim 3}.
Its role is to turn the binomial comparison obtained from Lemma~\ref{lem:bis} and Lemma~\ref{lem difference binomials tie} into a clean geometric decay bound for the probabilities of the competing labels. The proof is purely analytic.
The lemma's statement makes reference to the following expression
that will be used henceforth:
\begin{equation*}\label{eq q}
\qc = \qc(n, p) = \frac{1}{2} + \frac{1}{5}\min\bigg\{1, \frac{\sqrt{np^4}}{2}\bigg\}.
\end{equation*}
The quantity $\qc$ should be interpreted as a uniform lower bound on the probability that a vertex prefers the larger of two competing basins when their sizes differ by the amount guaranteed in Section~\ref{sec:3.2.1}. 
Thus, $2\qc-1$ measures the bias toward the dominant basin. 

\begin{claim}\label{claim:normal-ratio}
Fix any constant $\zeta>0$ and
suppose that $np^4=o((\log n)^{-1})$ and $np\gg n^{2/3}$.
Then, uniformly over $\ell\in[2,2k]$,
\[
\frac{1-\Phi\Big(\tfrac{(\ell-1)\sqrt{np^4}}{2\sqrt{1+p}}\Big)+4\zeta/\sqrt{np^2}}
{\Phi\Big(\tfrac{(\ell-1)\sqrt{np^4}}{2\sqrt{1+p}}\Big)-4\zeta/\sqrt{np^2}}
\le
\left(\frac{1-\qc}{\qc}\right)^{\ell-1}.
\]
\end{claim}
\begin{proof}[Proof of Claim~\ref{claim:normal-ratio}]
Note that 
\begin{equation}
\begin{split}\label{eq lambda}
\left(\frac{1-\qc}{\qc}\right)^{\ell-1} &=\; \left(\frac{1-\sqrt{np^4}/5}{1+\sqrt{np^4}/5}\right)^{\ell-1} = \bigg(1-\frac{2}{5}\sqrt{np^4} + O(np^4)\bigg)^{\ell-1}\\
&=\; \exp\bigg(-\frac25(\ell-1)\sqrt{np^4}+O((\ell-1)np^4)\bigg),    
\end{split}
\end{equation}
and if $(\ell-1)\sqrt{np^4} \le \eps$ for some sufficiently small $\eps > 0$, then using that $(1+p)^{-1/2} = 1 + O(p)$, we get
\[\Phi\left(\tfrac{(\ell-1)\sqrt{np^4}}{2\sqrt{1+p}}\right) - \frac{4\zeta}{\sqrt{np^2}} = \frac{1}{2} +  \frac{(\ell-1)\sqrt{np^4}}{2\sqrt{2\pi}} + O\bigg((\ell-1)^2np^4+(\ell-1)\sqrt{np^4}\cdot p+\frac{1}{\sqrt{np^2}}\bigg).\]
Hence, using that $(\ell-1)\sqrt{np^4}\cdot p = o((\ell-1)^2 np^4)$,
\begin{equation}\label{eq Phi}
\frac{1-\Phi\Big(\tfrac{(\ell-1)\sqrt{np^4}}{2\sqrt{1+p}}\Big) + \tfrac{4\zeta}{\sqrt{np^2}}}{\Phi\Big(\tfrac{(\ell-1)\sqrt{np^4}}{2\sqrt{1+p}}\Big) - \tfrac{4\zeta}{\sqrt{np^2}}} = 1 - \frac{2(\ell-1)\sqrt{np^4}}{\sqrt{2\pi}} + O\bigg((\ell-1)^2np^4+\frac{1}{\sqrt{np^2}}\bigg),
\end{equation}
and using that $\tfrac{2}{5}< \tfrac{2}{\sqrt{2\pi}}$ shows the desired inequality when $\eps$ is sufficiently small.
On the other hand, when $\eps^{-1}\ge (\ell-1)\sqrt{np^4}\ge \eps$, the inequalities $np^2 \gg 1$ and $\ell n p^4\le 2k n p^4 = O(\sqrt{np^4 \log n}) = o(1)$ ensure that it is sufficient to prove that for every $x > 0$, $\tfrac{1-\Phi(x)}{\Phi(x)} < \exp(-\tfrac{4}{5}x)$, or equivalently 
\[\Phi(x)(1+\exp(-\tfrac{4}{5}x)) - 1 > 0,\]
and then use this inequality for $x = \frac{1}{2}(\ell-1)\sqrt{np^4}$. The latter could be checked via tedious analysis; we provide a link\footnote{\tiny{\url{https://www.wolframalpha.com/input?key=&i=\%281\%2F2+\%2B+erf\%28x\%2Fsqrt\%282\%29\%29\%2F2\%29*\%28exp\%28-4x\%2F5\%29\%2B1\%29-1}}} with a numerical justification instead (using that $\Phi(x) = \tfrac{1}{2} + \tfrac{1}{2}\mathrm{erf}(\tfrac{x}{\sqrt{2}})$, where $\mathrm{erf}$ is the error function). Finally, if $(\ell-1)\sqrt{np^4}\ge \eps^{-1}$ for some sufficiently small $\eps$, using that~$\frac{x}{1-x}\le 2x$ when $x$ is small together with~\eqref{eq lambda}, the left-hand side of~\eqref{eq Phi} is at most
\[2\bigg(\bigg(1-\Phi\bigg(\frac{(\ell-1)\sqrt{np^4}}{2\sqrt{1+p}}\bigg)\bigg) + \frac{4\zeta}{\sqrt{np^2}}\bigg)\le 2\exp\bigg(-\frac{(\ell-1)^2 np^4}{8(1+p)}\bigg) + \frac{8\zeta}{\sqrt{np^2}} \ll \bigg(\frac{1-\qc}{\qc}\bigg)^{\ell-1},\]
where to show that $1/\sqrt{np^2}\ll \left(\tfrac{1-\qc}{\qc}\right)^{\ell-1}$, we used that 
\[(\ell-1)\sqrt{np^4}\le 2k\sqrt{np^4}\le 30\sqrt{\log n}+O(\sqrt{np^4})\ll \log(np^2).\qedhere\]
\end{proof}

We now use this estimate to prove the second-round dominance of label~$1$ inside level~$\Lvl{C}$.

\begin{lemma}\label{claim 3}
Suppose that $n^{2/3} \ll np < \sqrt{15} n^{3/4} (\log n)^{1/4}$. The event ``after the second round of ALAP, there are at least $\tfrac12(2\qc-1)n$ vertices in level $\Lvl{C}$ with label $1$, and the number of vertices with label $\ell\in [2,2k]$ is smaller by at least $\frac12 \sLvltl{C}{2}{1} \big(1-\big(\tfrac{1-\qc}{\qc}\big)^{\ell-1}\big)$ than the number of vertices with label~$1$'', that is, 
\[
\Big\{\sLvltl{C}{2}{1}\geq \tfrac12(2\qc-1)n \text{ and for every } \ell\in [2,2k], \sLvltl{C}{2}{1} - \sLvltl{C}{2}{\ell}\geq \tfrac12\big(1-\big(\tfrac{1-\qc}{\qc}\big)^{\ell-1}\big)\sLvltl{C}{2}{1}\Big\},
\]
holds a.a.s.
\end{lemma}
\begin{proof}
Fix $t_n = np^2/\sqrt{2}$. 
Condition on $\evnt{E}{gp}$ (as in Lemma~\ref{claim E}), on $\evnt{B}{sz}$ (as in Lemma~\ref{claim F}), and the edges (and non-edges) incident to all vertices in level~$A$. Moreover, if 
\[np\in [n^{3/4} (\log n)^{-1/2}, \sqrt{15} n^{3/4} (\log n)^{1/4}],\] condition on $\evnt{E}{cn}_*$ as well (as in Lemma~\ref{claim E'}). Note that, by definition of ALAP, $(\sLvltl{B}{1}{\ell})_{\ell=1}^{2k}$ are all measurable in terms of the edges between levels~$A$ and~$B$. Since we are conditioning on $\evnt{E}{gp}$ and $\evnt{B}{sz}$ which hold a.a.s.~(by Lemma~\ref{claim E} and Lemma~\ref{claim F}), if $np\le n^{3/4} (\log n)^{-1/2}$, it is sufficient to show the conclusion of the lemma conditionally on $\evnt{E}{gp}\cap\evnt{B}{sz}$, while if $np\in [n^{3/4} (\log n)^{-1/2}, \sqrt{15} n^{3/4} (\log n)^{1/4}]$, we condition on $\evnt{E}{gp}\cap\evnt{E}{cn}_*\cap\evnt{B}{sz}$ instead (we may do so since $\evnt{E}{cn}_*$ holds a.a.s.).

\medskip
The rest of the proof has two parts. First, we compare the probabilities with which a fixed vertex $u\in C$ receives labels $1$ and $\ell$ in the second round of ALAP. This will show that the probabilities of receiving labels $\ell\ge 2$ decay geometrically relative to the probability of receiving label~$1$. Second, we use concentration over the vertices of level~$C$ to turn this pointwise probability comparison into the asserted inequalities for the random counts $\sLvltl{C}{2}{\ell}$.

\paragraph{Pointwise comparison of label probabilities.}
Fix a vertex $u\in\Lvl{C}$ and, for every $\ell\in [2k]$, let~$p_\ell$ be the probability that $u$ receives label $\ell$ after the second round of ALAP, conditionally on the first-round exposure.
Our goal will be to compare $p_1$ and~$p_\ell$. First, recall that, conditionally under the event $\evnt{E}{gp} \cap \evnt{B}{sz}$ (or $\evnt{E}{gp} \cap \evnt{E}{cn}_* \cap \evnt{B}{sz}$, respectively), we have $\sLvltl{B}{1}{1} \ge \sLvltl{B}{1}{\ell} + (\ell-1)t_n+1/p$ (indeed, the event $\evnt{E}{gp}$ guarantees that 
$\sLvltl{B}{1}{1}-\sLvltl{B}{1}{\ell}
\ge (\ell-1)np^2/1.4$,
and the claimed inequality follows from
$t_n=np^2/\sqrt{2}$
and $np^3\to\infty$).
Then, we apply Lemma~\ref{lem:bis} with
$n=\sLvltl{B}{1}{1}$,
$n'=\sLvltl{B}{1}{\ell}$,
the comparison threshold $M$ equal to the maximum
number of neighbors of $u$ in
$\Lvltl{B}{1}{i}$ over
$i\in[2k]\setminus\{1,\ell\}$,
and $\rho$ being the number of times that this
maximum is attained.
Since $\evnt{E}{gp}$ implies that
$\sLvltl{B}{1}{1}>\sLvltl{B}{1}{i}$
for every $i\in[2,2k]$,
we have $0\le M<\sLvltl{B}{1}{1}$,
as required by Lemma~\ref{lem:bis}.
This shows that
\begin{equation}\label{eq:appl1:bis}
\frac{p_1}{p_{\ell}} 
\ge \frac{\Prob(\mathrm{Bin}(\sLvltl{B}{1}{1}, p) > \mathrm{Bin}(\sLvltl{B}{1}{\ell}, p))+\tfrac{1}{2}\Prob(\mathrm{Bin}(\sLvltl{B}{1}{1}, p)= \mathrm{Bin}(\sLvltl{B}{1}{\ell}, p))}{\Prob(\mathrm{Bin}(\sLvltl{B}{1}{\ell}, p) > \mathrm{Bin}(\sLvltl{B}{1}{1}, p))+\tfrac{1}{2}\Prob(\mathrm{Bin}(\sLvltl{B}{1}{\ell}, p)= \mathrm{Bin}(\sLvltl{B}{1}{1}, p))}.    
\end{equation}
Using the inequality $\sLvltl{B}{1}{1} \ge \sLvltl{B}{1}{\ell} + (\ell-1)t_n+1/p$, the latter implies that
\begin{equation}\label{eq:appl2:bis}
\begin{split}
\frac{p_1}{p_1 + p_{\ell}} 
&\ge \Prob(\mathrm{Bin}(\sLvltl{B}{1}{1}, p) > \mathrm{Bin}(\sLvltl{B}{1}{\ell}, p))+\frac{1}{2}\Prob(\mathrm{Bin}(\sLvltl{B}{1}{1}, p) = \mathrm{Bin}(\sLvltl{B}{1}{\ell}, p))\\
&\ge \Prob(\mathrm{Bin}(\sLvltl{B}{1}{1}, p)\ge \mathrm{Bin}(\sLvltl{B}{1}{1}-(\ell-1)t_n-1/p, p)+1).
\end{split}
\end{equation}
 
Further, we apply Lemma~\ref{lem difference binomials} with $a_1 = \sLvltl{B}{1}{1}$, $a_2 = \sLvltl{B}{1}{1}-(\ell-1)t_n-1/p = (1-o(1))np$ (where we use that the difference of $(\ell-1)np^2/1.4$ guaranteed between the basins by the event $\evnt{E}{gp}$ dominates $(\ell-1)t_n+1/p$), $M = 1$ and $p$; note that its assumptions are satisfied as, under the event~$\evnt{E}{gp}$, we have that $1\ll a_2\le a_1$ and $\min\{a_2, a_1-a_2\}p \ge t_np = np^3/\sqrt{2}\gg 1$. Since 
\[\frac{(a_1-a_2)p-M}{\sqrt{(a_1+a_2)pq}}\ge \frac{(\ell-1)t_np}{\sqrt{2pa_1}}=\frac{(\ell-1)np^3}{2\sqrt{pa_1}}\]
and $\tfrac{np}{4}\le a_2\le a_1\leq (n-2k)p + \omega \le np(1+p)$, we deduce that
\begin{equation}\label{eq p1pell}
\frac{p_1}{p_1 + p_{\ell}} \ge \Phi\bigg(\frac{(\ell-1)np^3}{2\sqrt{np^2(1+p)}}\bigg) - \frac{4\zeta}{\sqrt{np^2}},
\end{equation}
which leads to
\begin{equation*}
p_{\ell}\leq \frac{1-\Phi\Big(\frac{(\ell-1)np^3}{2\sqrt{np^2(1+p)}}\Big) + \frac{4\zeta}{\sqrt{np^2}}}{ \Phi\Big(\frac{(\ell-1)np^3}{2\sqrt{np^2(1+p)}}\Big) - \frac{4\zeta}{\sqrt{np^2}}}p_1 = \frac{1-\Phi\Big(\tfrac{(\ell-1)\sqrt{np^4}}{2\sqrt{1+p}}\Big) + \tfrac{4\zeta}{\sqrt{np^2}}}{\Phi\Big(\tfrac{(\ell-1)\sqrt{np^4}}{2\sqrt{1+p}}\Big) - \tfrac{4\zeta}{\sqrt{np^2}}}p_1.  
\end{equation*}
We claim that
\begin{equation}\label{assertion}
p_\ell\le \left(\frac{1-\qc}{\qc}\right)^{\ell-1}p_1.
\end{equation}
By Claim~\ref{claim:normal-ratio}, the assertion holds if $np^4=o((\log n)^{-1})$. 
Now, suppose that $(\log n)^{-2}\le np^4\le 225 \log n$ or equivalently 
\[np\in [n^{3/4} (\log n)^{-1/2}, \sqrt{15} n^{3/4} (\log n)^{1/4}].\] 
Then, on the event $\evnt{E}{cn}_*$, Remark~\ref{rem E'} ensures that $\sLvltl{B}{1}{\ell} - \sLvltl{B}{1}{\ell+1} \ge t_n = np^2 / \sqrt{2}$ for every $\ell\in [2k-1]$. By an application of Lemma~\ref{lem:bis} similar to~\eqref{eq:appl1:bis} and~\eqref{eq:appl2:bis}, we obtain that 
\begin{align*}
\frac{p_{\ell}}{p_{\ell} + p_{\ell+1}} 
&\ge \Prob(\mathrm{Bin}(\sLvltl{B}{1}{\ell}, p)\! > \mathrm{Bin}(\sLvltl{B}{1}{\ell+1}, p))+\frac{1}{2}\Prob(\mathrm{Bin}(\sLvltl{B}{1}{\ell}, p) = \mathrm{Bin}(\sLvltl{B}{1}{\ell+1}, p))\\
&\ge \Prob(\mathrm{Bin}(\sLvltl{B}{1}{\ell}, p)\! \ge \mathrm{Bin}(\sLvltl{B}{1}{\ell+1}, p)+1),
\end{align*}
implying that
\[
p_{\ell+1} \leq \frac{1-\mathbb{P}(\mathrm{Bin}(\sLvltl{B}{1}{\ell},p)-\mathrm{Bin}(\sLvltl{B}{1}{\ell+1},p)\geq 1)}{\mathbb{P}(\mathrm{Bin}(\sLvltl{B}{1}{\ell},p)-\mathrm{Bin}(\sLvltl{B}{1}{\ell+1},p)\geq 1)}\,p_\ell.
\]
Applying Lemma~\ref{lem difference binomials tie} for $a_1 = \sLvltl{B}{1}{\ell}$, $a_2 = \sLvltl{B}{1}{\ell+1}$, $M = 1$, $p$, and using Remark~\ref{rem E'}, which gives $a_1-a_2\ge np^2/1.4$, while uniformly in $\ell\in[2k-1]$, $a_1+a_2=(2+o(1))np$, we obtain
  \[
  \frac{(a_1-a_2)p}{\sqrt{(a_1+a_2)p}}
\ge
\left(\frac{1}{1.4\sqrt{2}}+o(1)\right)\sqrt{np^4}
\ge \frac{\sqrt{np^4}}{2},
\]
for all sufficiently large $n$. Therefore, by the definition of $\qc$, we get
$\Prob\big(\mathrm{Bin}(a_1,p)-\mathrm{Bin}(a_2,p)\ge1\big)\ge \qc$,
so $p_{\ell+1}\leq \frac{1-\qc}{\qc}p_{\ell}$.
An immediate induction then yields assertion~\eqref{assertion}.

Thus, recalling that $(p_{\ell})_{\ell=1}^{2k}$ adds up to $1$,
\begin{align}\label{eq:sum p}
1=\sum_{\ell=1}^{2k}p_{\ell}\leq p_1\sum_{\ell=0}^{\infty}\Big(\frac{1-\qc}{\qc}\Big)^{\ell}
= {p_1\frac{\qc}{2\qc-1}}.
\end{align}

\paragraph{Concentration over level~$C$.}
For all $\ell\in [2k]$, recall that $\sLvltl{C}{2}{\ell}$ equals the number of vertices in level~$C$ that get label $\ell$ at the second round.
Conditionally on the information revealed in the first round of ALAP, the sets $B$ and $C$, the basins $(\Lvltl{B}{1}{i})_{i=1}^{2k}$, and hence the probabilities $(p_i)_{i=1}^{2k}$ are fixed. The labels assigned in the second round to distinct vertices of level~$C$ are independent and have the same distribution, with probabilities $(p_i)_{i=1}^{2k}$. In particular, conditionally on the first-round exposure,
$\sLvltl{C}{2}{\ell}\sim\mathrm{Bin}(\sLvl{C},p_\ell)$
for every $\ell\in[2k]$.

If $np^4\ge 4$, then $\qc = 1/2 + 1/5=7/10$ and $k\le \lceil 15 (\log n)^{1/2}/2\rceil$. Combining this with~\eqref{eq:sum p} shows that $\mathbb E \sLvltl{C}{2}{1}= p_1\sLvl{C}\ge \tfrac{n}{2}\gg \log n$. Hence, by Chernoff's bound,
\begin{equation}
\begin{split}\label{eqn concentration Ci}
&\mathbb P\big(\sLvltl{C}{2}{1}\in p_1\sLvl{C}\pm 2(p_1\sLvl{C}\log n)^{1/2}\big)=\; 1-o(n^{-1}),\\
&\mathbb P\big(\sLvltl{C}{2}{\ell}\le\; p_{\ell}\sLvl{C} + \max\{2(p_{\ell}\sLvl{C}\log n)^{1/2}, (\log n)^2\}\big)=1-o(n^{-1}).
\end{split}
\end{equation}

In particular, with probability $1-o(n^{-1})$, we get that $\sLvltl{C}{2}{1}\ge \tfrac{3}{4} p_1\sLvl{C} \ge \tfrac{2}{3} p_1 n \ge \tfrac{1}{2} (2\qc-1) n$, which proves the first part of the lemma. 
On the other hand,~\eqref{eqn concentration Ci} implies that, for every $\ell\in [2,2k]$, with probability $1-o(n^{-1})$,
\begin{align*}
\sLvltl{C}{2}{1} - \sLvltl{C}{2}{\ell} \ge\;
& \big(p_1\sLvl{C} - 2(p_1\sLvl{C}\log n)^{1/2}\big) - \big(p_\ell\sLvl{C}+ \max\{2(p_{\ell}\sLvl{C}\log n)^{1/2}, (\log n)^2\}\big) \\
\ge\; 
& (p_1-p_{\ell})\sLvl{C} - 4(p_1\sLvl{C}\log n)^{1/2}\\
\ge\; 
&\frac{2}{3}\Big(1 - \Big(\frac{1-\qc}{\qc}\Big)^{\ell-1}\Big) p_1\sLvl{C}\\
\ge\; 
&\frac{1}{2}\Big(1-\Big(\frac{1-\qc}{\qc}\Big)^{\ell-1}\Big) \sLvltl{C}{2}{1},
\end{align*}
and by a union bound the statement of the lemma follows for the case $np^4 \ge 4$.

Now, consider the case $np^4 < 4$. 
By Chernoff's bound we have that, for every $\eps > 0$,
\[
\Prob(|\sLvltl{C}{2}{1} - \mathbb E \sLvltl{C}{2}{1}| \ge (np_1)^{1/2+\eps}) = O(n^{-2}).
\]
Recalling that $1\leq p_1\frac{\qc}{2\qc-1}$, we get $\mathbb E\sLvltl{C}{2}{1}=(1-o(1))np_1\geq (2\qc-1)n = \Omega(\sqrt{n^3p^4}) \ge (\log n)^2$. Hence, with probability $1-O(n^{-2})$, we have
$\sLvltl{C}{2}{1}\ge \frac12(2\qc-1)n$, which proves the first part of the lemma.
On the other hand,
\begin{align*}
& \Prob\Big(\sLvltl{C}{2}{1} - \sLvltl{C}{2}{\ell}\le \Big(1-\Big(\tfrac{1-\qc}{\qc}\Big)^{\ell-1}\Big)\tfrac{\sLvltl{C}{2}{1}}{2}\Big) \\
& \qquad \le 
\Prob\Big(\sLvltl{C}{2}{1} - \sLvltl{C}{2}{\ell}\le \Big(1-\Big(\tfrac{1-\qc}{\qc}\Big)^{\ell-1}\Big) \tfrac{2\sLvl{C}p_1}{3}\Big)
+ \Prob\Big(\sLvltl{C}{2}{1}\ge \tfrac{4\sLvl{C}p_1}{3}\Big).
\end{align*}
Since $\mathbb E\sLvltl{C}{2}{1}=\sLvl{C}p_1$, by Chernoff's bound, the 
second term on the right-hand side above is $O(n^{-2})$, while the first is bounded from above by the sum of
\[ 
\Prob\bigg(\sLvltl{C}{2}{1} \le p_1\sLvl{C} - \bigg(1-\bigg(\frac{1-\qc}{\qc}\bigg)^{\ell-1}\bigg) \frac{p_1\sLvl{C}}{6}\bigg)\]
and
\[ 
\Prob\bigg(\sLvltl{C}{2}{\ell}\ge \bigg(\frac{1-\qc}{\qc}\bigg)^{\ell-1} p_1\sLvl{C} + \bigg(1-\bigg(\frac{1-\qc}{\qc}\bigg)^{\ell-1}\bigg) \frac{p_1\sLvl{C}}{6}\bigg).
\]
Using Chernoff's bound again (and recalling that 
$\mathbb E\sLvltl{C}{2}{\ell}=\sLvl{C}p_\ell\leq (\frac{1-\qc}{\qc})^{\ell-1}\sLvl{C}p_1$), both probabilities above can be bounded from above by
\begin{align*}
\exp\bigg(-\bigg(1-\bigg(\frac{1-\qc}{\qc}\bigg)^{\ell-1}\bigg)^2 \frac{p_1\sLvl{C}}{100}\bigg) 
&\leq 
\exp\bigg(-\bigg(\frac{2\qc-1}{\qc}\bigg)^2 \frac{p_1\sLvl{C}}{100}\bigg)\\ 
&\leq 
\exp\bigg(-\frac{(2\qc-1)^3n}{100}\bigg)
=O(n^{-2}),
\end{align*}
where in the last inequality we used that 
$\frac{2\qc-1}{\qc}\leq p_1$, $|C|=(1-o(1))n$, and $\qc=\frac{1}{2}+\frac{\sqrt{np^4}}{10}\le \frac{7}{10}$, and for the equality we used that $(2\qc-1)^3=\Omega((np^4)^{3/2})=\Omega(n^{-1/2})$.
Summarizing, uniformly in $\ell\in [2,2k]$,
\[
\Prob\bigg(\sLvltl{C}{2}{1} - \sLvltl{C}{2}{\ell}\le \frac12\bigg(1-\bigg(\frac{1-\qc}{\qc}\bigg)^{\ell-1}\bigg)\sLvltl{C}{2}{1}\bigg) = O(n^{-2}).
\]
The lemma follows by a union bound over all $\ell\in [2,2k]$.
\end{proof}

\begin{remark}[Critical-window modification -- Part I]\label{rem claim 3}
Suppose that $np=cn^{2/3}$ for some constant $c>0$. Fix
$\eps\in(0,1)$ and choose $L=L(\eps,c)$ sufficiently large, as in
Remark~\ref{rem E:IR}. Conditionally on $\evnt{E}{tl}_L$, a.a.s.\ the
following two events hold simultaneously:
\[
\Big\{\max_{i\in [L]} \sLvltl{C}{2}{i}\geq \tfrac12(2\qc-1)n\Big\}
\]
and
\[
\Big\{\text{for every }\ell\in [L+1,2k], \max_{i\in [L]} \{\sLvltl{C}{2}{i} - \sLvltl{C}{2}{\ell}\}\geq \tfrac12\big(1-\big(\tfrac{1-\qc}{\qc}\big)^{\ell-1}\big)
\max_{i\in [L]} \sLvltl{C}{2}{i}\Big\}.
\]

The modifications of the proof of Lemma~\ref{claim 3} necessary to establish the claim are as follows. First, at the beginning of the proof, we replace~$\evnt{E}{gp}$ by~$\evnt{E}{tl}_L$.
Since $np^3=c^3$, choosing $L$ sufficiently large ensures
that the separation between the first basin and the basins
indexed by $\ell\ge L+1$ dominates the
normal-approximation error in the binomial comparison.
The critical-window adaptation of the comparison leading to~\eqref{eq p1pell} and of Claim~\ref{claim:normal-ratio} 
then gives, uniformly for $\ell\in[L+1,2k]$,
\[
p_\ell\le\Big(\frac{1-\qc}{\qc}\Big)^{\ell-1}p_1
\le \Big(\frac{1-\qc}{\qc}\Big)^{\ell-1}p_*
\qquad\text{where $p_*=\max_{\ell\in[L]}p_\ell$.}
\]
Summing these bounds over $\ell\ge L+1$ and using
$p_i\le p_*$ for $i\in[L]$, we obtain the
following analogue of~\eqref{eq:sum p}:
\[
p_* \Big(L+\frac{\qc}{2\qc-1}\Big)\ge 1.
\]
This normalization bound gives
$p_*\sLvl{C}=\Omega(n^{5/6})$.
Moreover, writing $r=(1-\qc)/\qc$, we have
$(1-r^L)^2p_*\sLvl{C}=\Omega(n^{1/2})\gg\log n$,
so Chernoff's inequality and a union bound allow us to
transfer the probability comparison between the $p_\ell$'s to the
numbers of vertices carrying the corresponding labels.
This establishes both asserted events.
\end{remark}

\begin{remark}[Critical-window modification -- Part II]\label{rem claim 3'}
Suppose that $np=cn^{2/3}$ for some constant $c>0$.
Fix $\eps\in(0,1)$ and let $L=L(\eps,c)$ be as in
Remark~\ref{rem claim 3}. Define
\begin{equation}\label{eq:lhat}
\widehat \ell_1 =
\min\Big\{\ell\in[L]:
\sLvltl{C}{2}{\ell}
=\max_{i\in[L]}\sLvltl{C}{2}{i}\Big\}.
\end{equation}
hen, conditionally on
$\evnt{E}{tl}_L\cap\evnt{E}{sp}_{L,\eps}$,
a.a.s.\ for every $\ell\in[L]\setminus\{\widehat\ell_1\}$,
\begin{equation*}
\sLvltl{C}{2}{\widehat\ell_1}
-\sLvltl{C}{2}{\ell}
=\Omega\big((2\qc-1)\sLvltl{C}{2}{\widehat\ell_1}\big).
\end{equation*}
The corresponding adaptation of the proof of
Lemma~\ref{claim 3} uses Lemma~\ref{lem:separate}
and Remark~\ref{rem binomials'} in place of
Lemma~\ref{claim E} and
Lemma~\ref{lem difference binomials tie},
respectively, and $p_*$ from Remark~\ref{rem claim 3}
in place of $p_1$.
%
%
%
\end{remark}
We now return to the regime $n^{2/3}\ll np\le \sqrt{15}n^{3/4}(\log n)^{1/4}$ and continue the second-round analysis in level~$\Lvl{B}$.

Our next goal is to bound from above the number of vertices
in level~$\Lvl{B}$ carrying the most frequent label in this
level after the second round.
The following observation will be useful in establishing this bound.

\begin{claim}\label{ob 3.5}
Suppose that $np=\Omega(n^{5/8})$ and
$np\le \sqrt{15}n^{3/4}(\log n)^{1/4}$.
Then, a.a.s.\ every vertex in level~$\Lvl{A}$ and level~$\Lvl{B}$ is connected to at most $10$ vertices in level~$\Lvl{A}$. Moreover, a.a.s.\ $k-o(k)$ of the vertices $(v_i)_{i=k+1}^{2k}$ have no neighbor among $(v_i)_{i=1}^{k}$.
\end{claim}
\begin{proof}
For any $j\in [n]$, the number of neighbors of $v_j$ in level~$\Lvl{A}$ is stochastically dominated by $\mathrm{Bin}(2k,p)$. Since $kp=p^{-1}n^{-1/2+o(1)} \le n^{-1/9}$, by a union bound over all vertices we conclude that 
\begin{equation*}
\Prob(\exists j\in [n], |N(v_j)\cap\Lvl{A}|\ge 10)\le n \binom{2k}{10} p^{10} = O(n (kp)^{10}) = o(1),
\end{equation*}
which shows the first statement. The second statement follows from the fact that every vertex in $A$ has $(2k-1)p=o(1)$ expected neighbors in $A$ and Markov's inequality.
\end{proof}

The next result is an analog of Lemma~\ref{claim 3} but concerning vertices at level~$\Lvl{B}$. However, unlike in Lemma~\ref{claim 3} where $\sLvltl{C}{2}{1}$ was approximated by a binomial random variable, the lower bound on $\sLvltl{B}{2}{1} - \sLvltl{B}{2}{\ell}$ in Lemma~\ref{claim 4} is given in terms of $\mathbb E\sLvltl{B}{2}{1}$ and not of $\sLvltl{B}{2}{1}$ itself due to the lack of an appropriate upper bound on $\sLvltl{B}{2}{1}$.

\begin{lemma}\label{claim 4}
Suppose that $n^{2/3} \ll np \le \sqrt{15} n^{3/4} (\log n)^{1/4}$. Then, a.a.s.\ for every $\ell\in [2,2k]$, the difference between $\sLvltl{B}{2}{\ell}$ and $\sLvltl{B}{2}{1}$ is bounded from above by $44np$.
\end{lemma}
\begin{proof}
The proof has two steps. First, we show that, conditionally on the edges incident to level $A$, 
the expected deficit of label 1 in level $B$ after the second round of ALAP compared to the most represented label is only $\sqrt{n}=o(np)$.
Second, we show that the differences $|B_2(1)|-|B_2(\ell)|$ for $\ell\in [2,2k]$ are sufficiently concentrated around their (conditional) expectations to ensure that a.a.s.\ no competing label is overrepresented in $B$ after two rounds of ALAP.

\smallskip\noindent
\textbf{Step 1: Quantifying the expected advantage.}
First, we reveal all edges incident to level $A$, thus determining the levels $B,C$ as well as the first-round basins. 
By Lemmas~\ref{claim E} and~\ref{claim F}, Claim~\ref{ob 3.5} and Chernoff's
inequality, we may assume each of the following a.a.s.\ events: for all $\ell\in [2,2k]$,
\begin{equation}\label{eq:claim4-good-exposure}
\begin{split}
|B_1(1)|\ge |B_1(\ell)|+1;\;
 np/2\le |B_1(\ell)|\le |N(v_\ell)|\le 2np;\; \max_{u\in A\cup B}|N(u)\cap A|\le 10;\; |B|\le 3knp.
\end{split}
\end{equation}
Fix a first-round exposure configuration satisfying these inequalities.
In particular, all probabilities computed below are over this conditional space.

Now, for every $u\in B$ and $\ell\in [2k]$, write
\[
 a_\ell(u)=|N(u)\cap A_1(\ell)|+\mathds{1}_{\{u\in B_1(\ell)\}},\qquad
 b_\ell(u)=|N(u)\cap B_1(\ell)|.
\]
Here, $a_\ell(u)$ is fixed under the conditioning, bounded by 11 due to~\eqref{eq:claim4-good-exposure} and typically much smaller than the total score $a_\ell(u)+b_\ell(u)$. Also, for every $u\in B$, the variables
$b_\ell(u)$ are independent and
\[
 b_\ell(u)\sim\operatorname{Bin}
 \bigl(|B_1(\ell)|-\mathds{1}_{\{u\in B_1(\ell)\}},p\bigr).
\]
In particular, the inequality $|B_1(1)|\ge |B_1(\ell)|+1$ implies that $b_1(u)$
stochastically dominates $b_\ell(u)$ for all $u$ and $\ell\ne 1$.

Fix $\ell\in[2,2k]$ and set
\[
 D_u=\mathds{1}_{u\in B_2(1)}-\mathds{1}_{u\in B_2(\ell)},
 \qquad
 S_\ell=\sum_{u\in B}D_u=|B_2(1)|-|B_2(\ell)|.
\]
Note that, if $a_1(u)\ge a_\ell(u)$, stochastic domination and uniform
tie-breaking imply that
$\mathbb E[D_u]\ge 0$.

We now identify the vertices for which the above stochastic domination does not apply.
Set
\[
 E_\ell=\{u\in B:a_\ell(u)>a_1(u)\}.
\]
Since a vertex in $A$ can carry label $\ell$ after round~1 only if it is $v_\ell$ or adjacent to $v_\ell$, by \eqref{eq:claim4-good-exposure}, we have $|A_1(\ell)|\le 11$.
Consequently,
\begin{equation}\label{eq:boundEu}
E_\ell\subseteq B_1(\ell)\cup\bigl(N(A_1(\ell))\cap B\bigr),\quad \text{and thus} \quad |E_\ell|\le 2np+11\cdot 2np=24np.    
\end{equation}
Now, for each such vertex in $A_\ell$, consider an auxiliary label attribution in which the deterministic contribution $a_\ell(u)$ is replaced by $a_1(u)$, leaving every other neighborhood size unchanged. 
In this auxiliary comparison, stochastic domination implies that the expected signed advantage of label 1 over $\ell$ is non-negative.
Now, restoring $a_\ell(u)$ changes the score of label $\ell$ by at most 11. 
Such a modification can affect the selected label only when $b_\ell(u)$ lies within a bounded distance of the maximum of the other scores. 
Conditioning on those other scores and using that a binomial random variable with parameters $\Theta(np)$ and $p$ lies in any interval of $a_\ell(u)\le 12$ consecutive values with probability $O((np^2)^{-1/2})$ shows that there is $C>0$ such that, for every $u\in E_\ell$, $\mathbb E[D_u]\ge -C/\sqrt{np^2}$, which combined with~\eqref{eq:boundEu} yields
 \[\mathbb E[S_\ell]
 \ge -24Cnp/\sqrt{np^2}=-24C\sqrt n.
\]

\smallskip\noindent
\textbf{Step 2. Concentration of the label-class sizes.}
For distinct $u,v\in B$, the pair $uv$ is the only (potential) edge influencing the label of each of $u$ and $v$. 
Thus, conditionally on the status of $uv$ and the edges incident to $A$, the choice of labels  of $u$ and $v$ takes place independently.

Now, the status of $uv$ influences (by 1) only the number of neighbors of $u$ with the same label as $v$ after round 1 and vice-versa.
By writing $G$ for the host graph and using the argument for estimating $\mathbb E[D_u]$ from the end of step 1,
\begin{equation}\label{eq:discDu}
 \big|\mathbb E[D_u\mid uv\in E(G)]
       -\mathbb E[D_u\mid uv\notin E(G)]\big|
 = O((np^2)^{-1/2}).
\end{equation}
Now, given $\ell\in [2k]$, to understand the covariance of $D_u,D_v$ (and to eventually bound $S_\ell$ from above), we split the computation according to the state of $uv$:
\begin{align}\label{eq snd moment D}
\mathbb E[D_uD_v] = \onemp \mathbb E[D_uD_v\mid uv\notin G] + p \mathbb E[D_uD_v\mid uv\in G].
\end{align}
The random variables $D_u$ and $D_v$ are independent in each of the two conditional spaces introduced above; indeed, in both cases, the two variables are measurable in terms of disjoint sets of (potential) edges. 
Hence, 
\begin{align}\label{eq 2nd moment 1}
\mathbb E[D_uD_v\mid uv\notin G] & = \mathbb E[D_u\mid uv\notin G]\cdot \mathbb E[D_v\mid uv\notin G], \\
\label{eq 2nd moment 2}
\mathbb E[D_uD_v\mid uv\in G] & = \mathbb E[D_u\mid uv\in G]\cdot \mathbb E[D_v\mid uv\in G].
\end{align}
Combining~\eqref{eq:discDu}--\eqref{eq 2nd moment 2} shows that the covariance $\mathbb E[D_uD_v]-\mathbb E[D_u]\mathbb E[D_v]$ of $D_u$ and $D_v$ is given by
\begin{align}\label{eq:covs}
&q\mathbb E[D_u\mid uv\notin G]\cdot \mathbb E[D_v\mid uv\notin G] + p \mathbb E[D_u\mid uv\in G]\cdot \mathbb E[D_v\mid uv\in G]\nonumber\\
&\hspace{5em} - (p\mathbb E[D_u\mid uv\in G]+q\mathbb E[D_u\mid uv\notin G])(p\mathbb E[D_v\mid uv\in G]+q\mathbb E[D_v\mid uv\notin G])\nonumber\\
&\qquad = pq\big(\mathbb E[D_u\mid uv\in G]-\mathbb E[D_u\mid uv\notin G]\big)\big(\mathbb E[D_v\mid uv\in G]-\mathbb E[D_v\mid uv\notin G]\big) \nonumber \\
& \qquad = O(1/np).
\end{align}

We now turn to estimating $S_\ell$. Since $|B|\le 3knp$ by \eqref{eq:claim4-good-exposure} and $|D_u|\le 1$ for all $u\in B$,
\begin{align*}
 \operatorname{Var}(S_\ell)=\sum_{u\in B}\operatorname{Var}(D_u)
   +\sum_{\substack{u,v\in B\\u\ne v}}
      \operatorname{Cov}(D_u,D_v) = O\bigg(|B|+\frac{|B|^2}{np}\bigg)=O(k^2np).
 \label{eq:claim4-offsets-var}
\end{align*}
Recall also that, $\sqrt{n}=o(np)$ and uniformly over all first-round configurations satisfying~\eqref{eq:claim4-good-exposure}, $\mathbb E[S_\ell]\ge-C\sqrt n$.
As a result, for all sufficiently large $n$, Chebyshev's inequality implies that

\begin{equation*}\label{eq:claim4-chebyshev}
 \mathbb P(S_\ell<- 44np)\le \mathbb P(S_\ell-\mathbb E[S_\ell]<- 43np)\le \operatorname{Var}(S_\ell)/(43np)^2=O(k^2/np),
\end{equation*}
and taking a union bound over the $2k-1$ competing labels yields
\[
 \mathbb P\bigl(\exists\ell\in [2,2k],
 |B_2(\ell)|-|B_2(1)|>44np\bigr)
 = O(k^3/np) = O((\log n)^{3/2}/n^{5/2}p^7) = o(1), 
\]
as desired. 
\end{proof}

The following remark adapts Lemma~\ref{claim 4} to the critical regime $np=cn^{2/3}$.

\begin{remark}\label{rem claim 4}
Suppose that $np=cn^{2/3}$ for a fixed $c>0$.
We explain how the proof of Lemma~\ref{claim 4} changes when label~$1$
need not have the largest first-round basin. 
Denote by $w$
the smallest index of a largest first-round basin, namely 
\[
 w=\text{argmin}\Big\{i\in[2k]:
          |B_1(i)|=\max_{j\in[2k]}|B_1(j)|\Big\}.
\]
The same reasoning (with minor modifications) shows that a.a.s., for every fixed $\eta>0$,
\[
 |B_2(\ell)|-|B_2(w)|\le \eta np
 \qquad\text{for every $\ell\in[2k]\setminus\{w\}$.}
\]
The necessary modifications are as follows.

\smallskip\noindent
\textbf{Replacement of the first-round gap event.}
Fix $\eps\in(0,1/4)$ and choose $L=L(\eps,c)$
sufficiently large as in Remark~\ref{rem E:IR}. Let
\[
 \mathcal H_{L,\eps}
   =\evnt{E}{tl}_L\cap\evnt{E}{sp}_{L,\eps}.
\]
By Remark~\ref{rem E:IR} and Lemma~\ref{lem:separate},
$\mathbb P(\mathcal H_{L,\eps})\ge1-2\eps-o(1)$.
By the event $\evnt{E}{tl}_L$, the largest basin is among the first $L$, and by $\evnt{E}{sp}_{L,\eps}$ this basin is unique: indeed, since $np^2=\Theta(n^{1/3})$, for all sufficiently large~$n$ and $i\in [2k]\setminus \{w\}$, $|B_1(w)|\ge |B_1(i)|+1$.

\smallskip\noindent
\textbf{The first-rounds events in~\eqref{eq:claim4-good-exposure}.}
Each of the inequalities
$np/2\le|B_1(i)|\le|N(v_i)|\le2np$,
$\max_{u\in A\cup B}|N(u)\cap A|\le 10$
and $|B|\le3knp$ holds a.a.s.\ in this regime as well.
Because $\mathbb P(\mathcal H_{L,\eps})$ is bounded
away from zero, these bounds still have conditional
probability $1-o(1)$ given $\mathcal H_{L,\eps}$.
Fixing the complete first-round exposure then leaves
all edges inside~$B$ independent, exactly as before.

\smallskip\noindent
\textbf{Using $w$ instead of $1$ in the expectation comparison.}
For each $i\ne w$, replace $D_u$ and $S_\ell$ by
\[
 D_u^w=\mathds{1}_{\{u\in B_2(w)\}}
             -\mathds{1}_{\{u\in B_2(i)\}},
 \qquad S_i^w=\sum_{u\in B}D_u^{(w,i)}.
\]
For the updated exceptional sets $E_{w,i}=\{u\in B:a_i(u)>a_w(u)\}$,
exactly as in~\eqref{eq:boundEu},
we have $|E_{w,i}|\le24np$ and the same binomial
anti-concentration estimate gives
$\mathbb E[S_i^w]\ge -C\sqrt n$,
uniformly in $i$ and in the exposed configuration. Note that conditionally on
$\mathcal H_{L,\eps}$,  $w=\widehat{\ell}_1$ with $\widehat{\ell}_1$ defined as in Remark~\ref{rem claim 3'}.

\smallskip\noindent
\textbf{Changes in the constants of the second-moment}.
By the same reasoning as in the proof of Lemma~\ref{claim 4}, for every $\eta>0$ and $\ell\in [2k]\setminus \{w\}$, we have $\E[S_\ell]\ge-24C\sqrt{n}\ge -\eta np/2$, where we used that $\sqrt{n}=o(np)$, and $\operatorname{Var}(S_\ell)=O(k^2np)$. Then, Chebyshev's inequality gives
\[
  \mathbb P\big(S_\ell<-\eta np\big)
  \le\mathbb P\big(S_\ell-\E[S_\ell]<-\eta np/2\big)
  \le\frac{ \operatorname{Var}(S_\ell)}{(\eta np/2)^2}
  =\frac{4}{\eta^2}\cdot\frac{O(k^2np)}{(np)^2}
  =O\!\bigg(\frac{k^2}{\eta^2\,np}\bigg).
\]
Taking a union bound over the $2k-1$ competing labels yields
\[
\mathbb P\big(\exists\,\ell\in[2k] \setminus \{w\},\ |B_2(\ell)|-|B_2(w)|>\eta np\big)
  =O\!\bigg(\frac{k^3}{\eta^2\,np}\bigg)
  =O\!\bigg(\frac{(\log n)^{3/2}}{\eta^2\,n^{5/2}p^7}\bigg)
  =o(1).
\]
Hence, a.a.s., for every fixed $\eta>0$ and every $\ell\in [2k]\setminus \{w\}$, $|B_2(\ell)|-|B_2(w)|\le \eta np$.
Thus, the desired bound holds with conditional
probability $1-o(1)$ given $\mathcal H_{L,\eps}$,
and hence unconditionally with probability at least
$1-2\eps-o(1)$. Since $\eps>0$ can be taken
arbitrarily small (with $L$ chosen afterwards), the
claim follows a.a.s.
\end{remark}

The preceding results control the numbers of vertices
carrying each label in levels~$B$ and~$C$ after the
second round. To analyze the third round, however, we need a more local statement: for most vertices in level~$C$, no competing label has many more neighbors in level~$B$ than label~$1$. The next lemma provides this estimate.


\begin{lemma}\label{ob 4.5}
Suppose that $n^{2/3} \ll np \le \sqrt{15} n^{3/4} (\log n)^{1/4}$. Then, after the second round of ALAP, a.a.s.\ $n-o(n)$ of the vertices in level $\Lvl{C}$ have the following property: for every $\ell\in [2,2k]$, 
their number of neighbors in level~$B$ with label
$\ell$ is at most $90np^2$ larger than their number
of neighbors in level~$B$ with label~$1$.
\end{lemma}
\begin{proof}
Condition on the states of all edges incident to level~$A$ and of the edges with both endpoints in~$B$,
as well as on the a.a.s.\ event $\evnt{B}{sz}$ from
Lemma~\ref{claim F} and the conclusion of Lemma~\ref{claim 4}.
This determines the sets $B$, $C$ and $B_2(i)$ for all $i\in[2k]$, while the edges between $B$ and $C$
remain unrevealed.
For every $\ell\in [2,2k]$, denote by $B_2'(\ell)$ an arbitrary subset of $B_2(\ell)$ of size $\lceil 44np\rceil$, if $\sLvltl{B}{2}{\ell}\ge 44np$, and $B_2'(\ell) = B_2(\ell)$ otherwise. Then, Chernoff's inequality and a union bound show that a.a.s., for every $\ell\in [2,2k]$ and every vertex $u\in C$, there are at most~$88np^2$ neighbors of $u$ in~$B_2'(\ell)$.

Now, fix a vertex $u\in \Lvl{C}$. Also, denote by $Y_1$ the number of neighbors of $u$ in $\Lvltl{B}{2}{1}$ and, for every $\ell\in [2,2k]$, denote by $Y_\ell$ the number of neighbors of $u$ in $\Lvltl{B}{2}{\ell}\setminus B_2'(\ell)$.
In particular, $Y_1\sim\mathrm{Bin}(\sLvltl{B}{2}{1}, p)$ and $Y_\ell\sim\mathrm{Bin}(\max\{\sLvltl{B}{2}{\ell}-\lceil 44np\rceil,0\}, p)$ for $\ell\in [2,2k]$. 
Then, Chernoff's bound implies
\begin{equation}\label{eq:Chern_split}
\begin{split}
\Prob(Y_1+2np^2\le Y_\ell)
&\le \Prob\big(Y_1\le \tfrac12\mathbb E[Y_1+Y_\ell]-np^2\big) + \Prob\big(Y_\ell\ge \tfrac12\mathbb E[Y_1+Y_\ell]+np^2\big)\\
&\le \Prob\big(Y_1 - \mathbb E Y_1\le -np^2\big) + \Prob\big(Y_\ell - \mathbb E Y_\ell\ge np^2\big)\\
&\le 2\exp\Big(-\tfrac{n^2p^4}{2\mathbb E[Y_1]+2np^2}\Big) 
\\
& \le 2\exp\Big(-\tfrac{n^2p^4}{8np^2k}\Big),
\end{split}
\end{equation}
where the last inequality follows from the fact that  the event $\evnt{B}{sz}$ holds, so $p\sLvltl{B}{2}{1}\le 3np^2k$. 
Since $knp^2=O(\sqrt{n\log n})$ and $n^2p^4\gg n^{2/3}$, the exponent in the
last term is of order $\omega(n^{1/6}/\sqrt{\log n})=\omega(\log n)$.
Hence, a union bound over all $u\in C$ and
$\ell\in[2,2k]$ shows that a.a.s.\
$Y_\ell<Y_1+2np^2$ simultaneously for all such
$u$ and $\ell$. Combining this with the bound
$|N(u)\cap B_2'(\ell)|\le88np^2$ proves the lemma.
\end{proof}

\begin{remark}\label{rem 4.5}
Suppose that $np=cn^{2/3}$ for some constant $c>0$.
Fix $\eps\in(0,1)$ and choose $L=L(\eps,c)$ sufficiently
large, as in Remark~\ref{rem E:IR} and
Lemma~\ref{lem:separate}.
As in Remark~\ref{rem claim 4}, let $w$ be the smallest index $j\in[2k]$ maximizing $|B_1(j)|$.
On the event $\evnt{E}{tl}_L\cap\evnt{E}{sp}_{L,\eps}$,
we have $w\in[L]$, and this maximum is attained uniquely.
We may now replace
Lemma~\ref{claim 4} with Remark~\ref{rem claim 4}
in the proof of Lemma~\ref{ob 4.5} to conclude that,
conditionally on
$\evnt{E}{tl}_L\cap\evnt{E}{sp}_{L,\eps}$, for every $\eta>0$, 
a.a.s.\ all but $o(n)$ vertices $u\in\Lvl{C}$ satisfy
\[
|N(u)\cap\Lvltl{B}{2}{\ell}|
\le |N(u)\cap\Lvltl{B}{2}{w}|
+\eta np^2
\qquad\text{for every }\ell\in [L] \setminus \{w\}.
\]
We note that the adaptation of Lemma~\ref{ob 4.5} requires choosing the sets $B_2'(\ell)$ of size $\eta np/3$, say, and replacing the $2np^2$ slack in Chernoff's inequality~\eqref{eq:Chern_split} by another $\eta np^2/3$.
\end{remark}

The preceding lemma shows that, for almost all vertices
in level~$\Lvl{C}$, the contribution from level~$\Lvl{B}$ can favor
a competing label over label~$1$ by at most $90np^2$.
We now show that the still-unexposed edges within
level~$C$ provide label~$1$ with a substantially
larger advantage over every competing label.
This will be crucial in establishing that almost all
vertices adopt label~$1$ in the third round.

\begin{lemma}\label{ob:rd3level3}
Set $\mathfrak{N}
= \tfrac{1}{10}(2\qc-1)\sLvltl{C}{2}{1}$.
Suppose $n^{2/3}\ll np \le \sqrt{15} n^{3/4} (\log n)^{1/4}$. Then, a.a.s.\ $\mathfrak{N} \ge \tfrac{1}{20}(2\qc-1)^2 n$ and every vertex~$u$ in level $\Lvl{C}$ satisfies $|N(u)\cap C_2(1)|\ge |N(u)\cap C_2(\ell)|+\mathfrak{N} p$ for all $\ell\in [2,2k]$.
\end{lemma}
\begin{proof}
We condition on the event that in ALAP,
\begin{equation}\label{eq:somelabel}
\Big\{
\sLvltl{C}{2}{1}\ge\tfrac12(2\qc-1)n
\text{ and }
\sLvltl{C}{2}{1}-\sLvltl{C}{2}{\ell}
\ge\tfrac12\Big(1-
\Big(\tfrac{1-\qc}{\qc}\Big)^{\ell-1}\Big)
\sLvltl{C}{2}{1}
\text{ for every }\ell\in[2,2k]
\Big\},
\end{equation}
which happens a.a.s.\ by Lemma~\ref{claim 3}. In particular, $\mathfrak{N}=\tfrac{1}{10}(2\qc-1)\sLvltl{C}{2}{1}
\ge\tfrac{1}{20}(2\qc-1)^2n$,
proving the first assertion.

Fix a vertex $u\in C$. 
Since $\qc\le1$,~\eqref{eq:somelabel} implies that, for every $\ell\in[2,2k]$,
$\sLvltl{C}{2}{1}-\sLvltl{C}{2}{\ell}
\ge5\mathfrak{N}$.
Fix $\ell\in[2,2k]$. Define $X_i = |N(u)\cap \Lvltl{C}{2}{i}|$. Notice that 
$X_1\sim\mathrm{Bin}(\sLvltl{C}{2}{1}-\indic{\Lvltl{C}{2}{1}}{u},p)$ and $X_\ell\sim\mathrm{Bin}(\sLvltl{C}{2}{\ell}-\indic{\Lvltl{C}{2}{\ell}}{u},p)$. Moreover,
\[
\mathbb EX_1-\mathbb E X_\ell\ge \big((\sLvltl{C}{2}{1}-1)-\sLvltl{C}{2}{\ell}\big)p
\ge \frac12(2\qc-1)\sLvltl{C}{2}{1}p-p
= (5\mathfrak{N}-1)p.
\]
Thus, using Chernoff's bound, since $\mathfrak{N}\to\infty$, we get that 
\begin{align*}
\Prob(X_1-X_\ell\le 3\mathfrak{N} p)
&\le \Prob\Big(X_1 - \mathbb E X_1\le -\frac12\mathfrak{N} p\Big) + \Prob\Big(X_\ell - \mathbb E X_\ell\ge \frac12\mathfrak{N} p\Big)\\
&\le \exp\left(-\frac{\mathfrak{N}^2 p^2}{8\mathbb E X_1} \right) + \exp\Big(-\frac{\mathfrak{N}^2 p^2}{8 (\mathbb E X_\ell + \mathfrak{N} p/6)}\Big).
\end{align*}
Here, all probabilities and
expectations are conditional on the information
fixed above, and the bounds are uniform over
$u$ and $\ell$.
Since each of the right-hand side terms in the last display are bounded by
$\exp(-\Omega(\mathfrak{N}^2p/\sLvltl{C}{2}{1}))
= \exp(-\Omega((2\qc-1)^3np)) = o(n^{-2})$.
Indeed, since
$2\qc-1=\Theta(\min\{1,\sqrt{np^4}\})$
and $np\gg n^{2/3}$, we have
$(2\qc-1)^3np\gg n^{1/6}\gg\log n$.
A union bound over all $\ell\in [2,2k]$ implies
that $X_1>X_\ell+3\mathfrak{N}p$ holds a.a.s.
Finally, a union bound over all vertices in $\Lvl{C}$
finishes the proof.
\end{proof}

\begin{remark}\label{rem:rd3level3}
For every $\eps\in(0,1)$ and $c>0$, there
exist constants $\gamma=\gamma(\eps,c)>0$
and $\delta=\delta(\eps,c)>0$ such that
the following holds.
Suppose $np=cn^{2/3}$, and let
$\widehat\ell_1$ be defined by~\eqref{eq:lhat}.
Set
$\widehat{\mathfrak{N}}
=\gamma(2\qc-1)\sLvltl{C}{2}{\widehat\ell_1}$.
Then, with probability at least $1-\eps$,
$\widehat{\mathfrak{N}}
\ge\delta(2\qc-1)^2n$
and every vertex $u\in\Lvl{C}$ satisfies
$|N(u)\cap C_2(\widehat\ell_1)|
\ge |N(u)\cap C_2(\ell)|
+\widehat{\mathfrak{N}} p$
for every
$\ell\in[2k]\setminus\{\widehat\ell_1\}$.

The proof follows that of
Lemma~\ref{ob:rd3level3} (by replacing the bounds for $\ell \in [2,2k]$ by bounds for $\ell \in [2k]\setminus \{w\}$), using
Remarks~\ref{rem claim 3} and
\ref{rem claim 3'} to obtain the
corresponding second-round label
separation in the critical window.
\end{remark}

\begin{lemma}\label{claim 6}
Suppose that $n^{2/3} \ll np \le \sqrt{15} n^{3/4} (\log n)^{1/4}$. Then, after round five of LPA, a.a.s.\ all vertices have label $1$. Also, if $np=cn^{2/3}$ for some $c > 0$, after five rounds of LPA, a.a.s.\ all vertices have label $w$.
\end{lemma}
\begin{proof}
We start with the first statement. Recall that $K=\lceil 2(\log n)/p\rceil$ and
that $K>2k$ for all sufficiently large $n$.
By Chernoff's bound and a union bound, a.a.s.\
every vertex $u\in V$ satisfies
$|N(u)\cap V([K])|=O(Kp)$.

Recall also that $\mathfrak{N}$ is defined in
the statement of Lemma~\ref{ob:rd3level3}.
Combining that lemma with Lemma~\ref{ob 4.5},
we conclude that a.a.s.\ there are $n-o(n)$
vertices $u\in\Lvl{C}$ such that, for every
$\ell\in[2,2k]$,
\[
|N(u)\cap\Lvltl{C}{2}{1}|
 +|N(u)\cap\Lvltl{B}{2}{1}|
-
 |N(u)\cap\Lvltl{C}{2}{\ell}|
 -|N(u)\cap\Lvltl{B}{2}{\ell}|
 \ge \mathfrak{N}p-90np^2.
\]
Since 
$\mathfrak{N}p=\Omega((2\qc-1)^2np)\gg np^2$,
the right-hand side is
$(1-o(1))\mathfrak{N}p$.

Moreover, by Lemma~\ref{lemma 2} and the
ALAP--LPA coupling, a.a.s.\ discrepancies
between the second-round labels assigned by
the two processes are confined to
$V([K]\setminus[2k])$.
As $\mathfrak{N}p\gg Kp$, these discrepancies
cannot eliminate the advantage of label~$1$
established above.
The possible contribution of the vertex itself
to its closed neighborhood is also negligible.

Consequently, a.a.s.\ $n-o(n)$ vertices carry
label~$1$ after the third round of LPA.
The conclusion of the first statement now follows from
Lemma~\ref{claim 5}.

For the second part, we use the definition of $\mathfrak{N}p$ from Remark~\ref{rem:rd3level3} instead of Lemma~\ref{ob:rd3level3} and observe that 
$\mathfrak{N}p\ge \delta np^2$ for some $\delta > 0$. However, replacing Lemma~\ref{ob 4.5} by Remark~\ref{rem 4.5} and choosing $\eta$ in Remark~\ref{rem 4.5} sufficiently small, we see that $\mathfrak{N}p-\eta np^2 = \delta'np^2$ for some $\delta' > 0$, and noting that $Kp=o(np^2)$, conditionally on the event $\evnt{E}{tl}_L\cap \evnt{E}{sp}_{L,\eps}$,  a.a.s.\ only one label survives. Since these events hold jointly with probability at least $1-3\eps$, and since $\eps$ could be chosen arbitrarily small (independently of the constants $L, \delta, \eta$), only one label survives a.a.s.\ in the unconditional space as well. 
\end{proof}



Now, the first part of Lemma~\ref{claim 6} proves Theorem~\ref{thm 1} (the regime $n^{2/3}\ll np\ll n$). Also, by combining the second statement of Lemma~\ref{claim 6} with Remark~\ref{rem:limit},  Proposition~\ref{prop:criticalRegime} (the case $np=cn^{2/3}$) holds as well.

\subsection{\texorpdfstring{The regime $n^{5/8+\eps}\le np\ll n^{2/3}$}{}}\label{sec 3.2}

We now fix $\eps\in(0,1/24)$ and turn to the fluctuation-dominated regime $n^{5/8+\eps}\le np\ll n^{2/3}$.
All results in this subsection are proved only in this regime, and we therefore omit this assumption from the statements of the lemmas.
In this range, the deterministic drift between the expected first-round basin sizes is no longer strong enough to force the smallest label to have the largest basin. 
Instead, the largest first-round basin is produced by random fluctuations and is carried by a random label, denoted below by $\ell_1$.

The proof proceeds in five stages. First, we approximate the dependent first-round basin sizes by independent binomial random variables. Second, we use this approximation to identify the largest first-round basin and show that it is separated from the second largest by a gap of order~$(np)^{1/2-o(1)}$. Third, we show that this first-round gap is amplified after the second round inside level~$C$. Fourth, we control the contribution of level~$B$, including the exceptional vertices for which some relevant edges have already been exposed. Finally, we use the fresh edges inside level~$C$ to prove that label~$\ell_1$ takes over in the third round.

\paragraph{Approximating the first-round basins.}
Recall that $|A|=2k$, with $k$ defined in~\eqref{eq:defk}. Our first task is to analyze the maximum of the first-round basin sizes $(\sLvltl{B}{1}{\ell})_{\ell=1}^{2k}$, which we denote by~$\sLvl{B^{(1)}}$. 
In the present regime, this maximum is governed by random fluctuations rather than by the deterministic preference for smaller labels. We therefore begin by replacing these dependent basin sizes with independent binomial random variables. For this purpose, for every $\ell\in[2k]$, define
\[
z_\ell
=
n-(\ell-1)np+\frac12(\ell-1)(\ell-2)np^2.
\]
The parameter $z_\ell$ is chosen so that $z_\ell p$ captures the relevant second-order approximation to the expected size of the $\ell$-th first-round basin.
\begin{lemma}\label{lem approx}
There exists a coupling of $(\sLvltl{B}{1}{\ell})_{\ell=1}^{2k}$ with a sequence of independent random variables $(Z_\ell)_{\ell=1}^{2k}$ such that $Z_\ell\sim\mathrm{Bin}(z_\ell,p)$ for every $\ell\in[2k]$ and, a.a.s.,
\[
\max_{\ell\in[2k]} |Z_\ell-\sLvltl{B}{1}{\ell}|\le (np)^{2/5}.
\]
\end{lemma}
\begin{proof} 
  The proof has three steps. We first approximate the mean of each basin size by $np-(\ell-1)np^2$. We then construct independent binomial variables $Z_\ell$, treating $\ell=1$ separately and, for $\ell\ge2$,
  replacing the random set of previously claimed vertices by a deterministic-size set. Finally, we show that both the auxiliary error for $\ell=1$ and the error introduced by this replacement for $\ell\ge2$ are at most $(np)^{2/5}$.

\smallskip\noindent
\textbf{Mean approximation.}
The first-round basin $\sLvltl{B}{1}{\ell}$ has expected size close to $np-(\ell-1)np^2$; we first quantify the error in this approximation uniformly over $\ell$.
For every $\ell\in [2k]$, set $x_\ell = \ell^2 np^3$. First, we show that, for every $\ell\in [2k]$,
\begin{equation*}
|\mathbb E \sLvltl{B}{1}{\ell}-np+(\ell-1)np^2| \le x_\ell.
\end{equation*}
Recall that $\sLvltl{B}{1}{\ell}\sim\mathrm{Bin}(n-2k,q^{\ell-1}p)$, so $\mathbb{E}\sLvltl{B}{1}{\ell}=(n-2k)q^{\ell-1}p$. Since $q^{\ell-1}\leq 1-(\ell-1)p+\ell^2p^2$, we get
\[
\mathbb{E}\sLvltl{B}{1}{\ell} -np+(\ell-1)np^2 
\leq nq^{\ell-1}p-np+(\ell-1)np^2
\leq \ell^2np^3=x_{\ell}.
\]
Moreover, since $q^{\ell-1}\geq 1-(\ell-1)p$, we have
\[
\mathbb{E}\sLvltl{B}{1}{\ell}
-np+(\ell-1)np^2 
= np(q^{\ell-1}-1+(\ell-1)p)-2kq^{\ell-1}p
\geq -2kp.
\]
It remains to show that $\mathbb E \sLvltl{B}{1}{\ell}-np+(\ell-1)np^2 \ge -x_\ell$.
For this, observe that for~$n$ large enough,
\[2kp\leq 32\sqrt{\frac{\log n}{np^2}}\leq 32n^{1-4\epsilon}p^3\sqrt{\log n}\leq x_\ell.\]
Now, define the event
$$
\cA_\ell = \Big\{ \Big|\sum_{j=1}^{\ell-1}\big(\sLvltl{B}{1}{j}-np+(j-1)np^2\big)\Big| \le \sum_{j=1}^{\ell-1}x_j+2\Big(2\log n\sum_{j=1}^{\ell-1}\mathbb E\sLvltl{B}{1}{j}\Big)^{1/2} \Big\}.
$$
Then, the preceding mean bound and the triangle inequality imply that 
\begin{equation*}
    \overline{\cA_\ell}\subseteq \Big\{\Big|\sum_{j=1}^{\ell-1}(\sLvltl{B}{1}{j}-\mathbb E \sLvltl{B}{1}{j})\Big|\ge 2\Big(2\log n\sum_{j=1}^{\ell-1}\mathbb E\sLvltl{B}{1}{j}\Big)^{1/2}\Big\}.
\end{equation*}
Note that $\overline{\cA_1}=\emptyset$. 
For every $\ell\in\{2,\ldots,2k\}$, since
$\sum_{j=1}^{\ell-1}\sLvltl{B}{1}{j}\sim\mathrm{Bin}(n-2k, 1-\onemp^{\ell-1})$, Chernoff's bound gives
\[
\Prob(\overline{\cA_\ell})
\le \Prob\bigg(\big|\sum_{j=1}^{\ell-1}(\sLvltl{B}{1}{j}-\mathbb E \sLvltl{B}{1}{j})\big|\ge 2\Big(2 \log n\sum_{j=1}^{\ell-1}\mathbb  E \sLvltl{B}{1}{j}\Big)^{1/2} \bigg)
\le 2e^{-(2+o(1)) \log n}\le \frac{1}{n}.
\]

\smallskip\noindent
\textbf{Coupling construction.}
We construct the independent proxies $(Z_\ell)_{\ell=1}^{2k}$ by treating the first index separately and then using a common coupling for all $\ell\ge2$. For $\ell=1$, since
$\sLvltl{B}{1}{1}\sim\mathrm{Bin}(n-2k,p)$ and $z_1=n$, let
$X_1\sim\mathrm{Bin}(2k,p)$ be independent of the graph and of all other auxiliary randomness, and set $Z_1=\sLvltl{B}{1}{1}+X_1$.
Then, $Z_1\sim\mathrm{Bin}(z_1,p)$. For every $\ell\in\{2,\ldots,2k\}$, we instead replace the random set of vertices
$\Lvltl{B}{1}{[\ell-1]}$
already claimed by earlier basins with a set $U_\ell$ of deterministic size $n-z_\ell-2k$. The set $U_\ell$ will be chosen using only information revealed before exposing the edges from $v_\ell$ to $V\setminus\Lvl{A}$.

For every $\ell\in\{2,\ldots,2k\}$, observe that $\sLvltl{B}{1}{\ell}$ equals the number of edges from $v_\ell$ to $V \setminus (\Lvl{A} \cup \Lvltl{B}{1}{[\ell-1]})$. We define the random variable~$Z_\ell$ as the number of edges from $v_\ell$ to $V \setminus (\Lvl{A} \cup U_\ell)$, where $U_\ell$ is a set of vertices of size $n-z_\ell-2k$ defined as follows:
\begin{itemize}
\item if $\sLvltl{B}{1}{[\ell-1]} \ge n-z_\ell-2k$, then $U_\ell$ consists of $n-z_\ell-2k$ vertices of $\Lvltl{B}{1}{[\ell-1]}$, chosen using only the information revealed before exposing the edges from $v_\ell$ to $V\setminus\Lvl{A}$,

\item otherwise, construct $U_\ell$ by adding
  to $\Lvltl{B}{1}{[\ell-1]}$ any $n - z_\ell -2k - \sLvltl{B}{1}{[\ell-1]}$ vertices of $V \setminus (\Lvl{A}\cup\Lvltl{B}{1}{[\ell-1]})$, again chosen using only the information revealed before exposing the edges from $v_\ell$ to $V\setminus\Lvl{A}$.
\end{itemize}
Therefore, conditional on this previously revealed information, the edges from $v_\ell$ to $V\setminus(\Lvl{A}\cup U_\ell)$ are independent Bernoulli random variables with parameter $p$. Moreover, the edge variables used for different values of $\ell$ are disjoint and are independent of $X_1$. It follows that the resulting random variables $(Z_\ell)_{\ell=1}^{2k}$ are independent and satisfy $Z_\ell\sim\mathrm{Bin}(z_\ell,p)$ for every $\ell\in[2k]$.

\smallskip\noindent
\textbf{Error bound.}
The final step is to transfer the coupling above back to the actual basin sizes. For $\ell\ge2$, we control the error introduced by replacing the vertices already claimed by earlier basins with $U_\ell$, while for $\ell=1$ we control the auxiliary term $X_1$ introduced in the coupling. More precisely, we prove that a.a.s.~the resulting discrepancy is at most $(np)^{2/5}$ simultaneously for all $\ell\in[2k]$.
For every $\ell\in\{2,\ldots,2k\}$, the triangle inequality yields
\[|n-z_{\ell}-2k-\sLvltl{B}{1}{[\ell-1]}| \le 2k+\bigg|\sum_{j=1}^{\ell-1} (np - (j-1)np^2-\sLvltl{B}{1}{j})\bigg|.\]
Hence, conditionally on $\cA_\ell$, $|\sLvltl{B}{1}{\ell}-Z_{\ell}|$ is stochastically dominated by 
\[Y_\ell \sim \mathrm{Bin}\Big(2k+\sum_{j=1}^{\ell-1}x_j+ 2\Big(2 \log n\sum_{j=1}^{\ell-1}\mathbb  E \sLvltl{B}{1}{j}\Big)^{1/2}, p\Big).
\]
For $\ell=1$, $|Z_1-\sLvltl{B}{1}{1}|=X_1\sim\mathrm{Bin}(2k,p)$,
so the same domination holds, since the parameter in the preceding binomial distribution reduces to $2k$.
Using that $\mathbb E Y_\ell \le p \big(2k+\sum_{j=1}^{\ell-1} x_j + 10\sqrt{\ell np \log n}\big) = O(kp+k^3 n p^4 + \sqrt{knp^3\log n}) = o((np)^{2/5})$,
Chernoff's bound implies that
$$
\Prob\big(Y_\ell\ge (np)^{2/5} \big) \le\; \frac{1}{n}.
$$
Hence, for every $\ell\in[2k]$, conditionally on the event $\cA_\ell$, the probability that
$|\sLvltl{B}{1}{\ell} - Z_\ell|\ge (np)^{2/5}$ is $O(n^{-1})$. Recalling that
$\Prob(\overline{\cA_\ell})\le 1/n$, the corresponding unconditional probability is also
$O(n^{-1})$. Since $k=o(n)$, the  desired result now follows by a union bound over all $\ell\in[2k]$.
\end{proof}

\begin{remark}
The preceding proof can be extended to the regime
$np\in[n^{1/2+\varepsilon},n^{2/3}]$, for any
$\varepsilon\in(0,1/6]$, by retaining further terms in the expansion
of $q^{\ell-1}$ in the approximation of
$\mathbb E\sLvltl{B}{1}{\ell}$. However, since this lemma is not the
bottleneck in our argument, we do not pursue this extension here.
\end{remark}

Our next aim is to give a lower bound on the gap between the maximum and the second maximum of $(\sLvltl{B}{1}{\ell})_{\ell=1}^{2k}$. We do this by estimating this gap for the sequence $(Z_\ell)_{\ell=1}^{2k}$ instead, and transfer our conclusion to $(\sLvltl{B}{1}{\ell})_{\ell=1}^{2k}$ using Lemma~\ref{lem approx}.
Define
\[
Z^{(1)}=\max_{\ell\in[2k]} Z_\ell,
\qquad
\ell^{(1)}=\min\{\ell\in[2k]:Z_\ell=Z^{(1)}\},
\]
and then define
\[
Z^{(2)}=\max_{\ell\in[2k]\setminus\{\ell^{(1)}\}} Z_\ell.
\]
To begin with, we determine the asymptotic scale of $Z^{(1)}$.

\begin{lemma}\label{lem:max}
It holds that
$\displaystyle
\frac{Z^{(1)} - np}{\sqrt{np \log(1/(np^3))}} \xrightarrow{\mathbb P} 1 \text{ as } n\to \infty$.
\end{lemma}
\begin{proof}
For $\zeta>0$ we show that the deterministic threshold $T_n = T_n(\zeta) = np+\sqrt{\zeta np\log(1/(np^3))}$ is a.a.s.\ an upper bound for $Z^{(1)}$ whenever $\zeta>1$, and a.a.s.\ a lower bound whenever $\zeta<1$.  Taking $\zeta>1$ and $\zeta<1$ arbitrarily close to $1$ then gives the lemma.

\smallskip\noindent
\textbf{Tail estimates for individual proxies.}
We first obtain matching upper and lower bounds for the probability that an individual proxy $Z_\ell$ exceeds $T_n$. These estimates will allow us to determine whether $T_n$ is an upper bound for $Z^{(1)}$ by summing the individual exceedance probabilities.

By independence of $(Z_\ell)_{\ell=1}^{2k}$, we have
\[
\Prob\big(Z^{(1)}\le T_n\big)=\Prob\Big(\bigcap_{\ell=1}^{2k} \{Z_\ell \le T_n\}\Big)
=
\prod_{\ell=1}^{2k} \Prob(Z_\ell \le T_n).
\]
We now derive matching upper and lower bounds for $\Prob\big(Z_\ell \ge T_n+1\big)$. On the one hand, the fact that $Z_\ell\sim\mathrm{Bin}(z_\ell,p)$ implies that
$\mathbb E Z_{\ell} = z_{\ell}p = (1+o(1))np\gg T_n + 1 - z_{\ell}p$
(the $o(1)$ terms here and below may be chosen independently of $\ell$).
Combining this with Chernoff's bound yields
\begin{equation}\label{eq:tail-upper-Zell}
\Prob\big(Z_\ell\ge T_n+1\big)
=
\Prob\big(Z_\ell-\mathbb E Z_{\ell}\ge T_n+1-z_{\ell}p\big)
\le
\exp\Big(-(1+o(1))\frac{(T_n+1-z_\ell p)^2}{2np}\Big).
\end{equation}

On the other hand, recalling that $q=1-p$, Slud's inequality (Lemma~\ref{lem slud}) yields
\[
\Prob\big(Z_\ell\ge T_n+1\big)
\ge
1-\Phi\Big(\frac{T_n+1-z_\ell p}{\sqrt{z_\ell p\onemp}}\Big),
\]
where we recall that $\Phi$ is the cumulative density function of a standard normal random variable.
Using that, as $x\to\infty$,
\[
\int_x^\infty \mathrm{e}^{-u^2/2}\,\mathrm{d}u
\ge
\int_x^{x+1/x^2} \mathrm{e}^{-u^2/2}\,\mathrm{d}u
\ge
\frac{1}{x^2}\mathrm{e}^{-x^2/2+o_x(1)}
=
\mathrm{e}^{-(1+o_x(1))x^2/2},
\]
we obtain
\[
\Prob\big(Z_\ell\ge T_n+1\big)
\ge
\exp\Big(-(1+o(1))\frac{(T_n+1-z_\ell p)^2}{2np}\Big).
\]
Combining this lower bound with~\eqref{eq:tail-upper-Zell}, uniformly in $\ell\in[2k]$,
\begin{equation}\label{eq:tail-Zell}
\Prob\big(Z_\ell\ge T_n+1\big)
=
\exp\Big(-(1+o(1))\frac{(T_n+1-z_\ell p)^2}{2np}\Big).
\end{equation}

By our choice of $T_n$, expanding the square and canceling the factor $np$, we have
\begin{align}
\frac{1}{2np}(T_n+1-z_\ell p)^2
&=
\frac{1}{2np}\Big(\sqrt{\zeta np\log(1/(np^3))}
 +(1+o(1))(\ell-1)np^2\Big)^2
\label{eq:1+o(1) term}\\
&=
\frac12\zeta\log(1/(np^3))\nonumber\\
&\qquad
+(1+o(1))\Big(
(\ell-1)\sqrt{\zeta np^3\log(1/(np^3))}
+\frac12(\ell-1)^2np^3
\Big).
\nonumber
\end{align}
Hence, by~\eqref{eq:tail-Zell}, since $1-x=\mathrm{e}^{-(1+o(1))x}$ as $x\to0$ and
$\Prob(Z_\ell\le T_n)=1-\Prob(Z_\ell\ge T_n+1)=1-o(1)$,
\begin{align}
  &\Prob\big(Z^{(1)}\le T_n\big) \nonumber  \\
& \qquad 
= \prod_{\ell=1}^{2k}(1-\Prob(Z_\ell\geq T_n+1)) = \exp\Big(-(1+o(1))\sum_{\ell=1}^{2k}\Prob(Z_\ell\geq T_n+1)\Big)\nonumber\\
& \qquad =
\exp\Big(-(np^3)^{\zeta/2+o(1)}\sum_{\ell=1}^{2k} \exp\Big(-(1+o(1))\Big((\ell-1)\sqrt{\zeta np^3\log(1/(np^3))} \nonumber
\\
&\hspace{30em}+\frac12(\ell-1)^2np^3\Big)\Big)\Big). \label{eq:that I need}
\end{align}
\smallskip\noindent
\textbf{Summing the exceedance probabilities.}
Equation~\eqref{eq:that I need} shows that determining whether $T_n$ is a.a.s.\ an upper bound for $Z^{(1)}$ reduces to estimating the sum in its exponent. We now determine the order of this sum by matching upper and lower bounds.

For a lower bound on the sum, note that summing up to $\ell_* = (np^3 \log(1/(np^3)))^{-1/2}$ shows that it is bounded from below by
\begin{equation}\label{eq i_*}
\left(\frac{1}{\sqrt{np^3 \log(1/(np^3))}}\right)^{1+o(1)} = \frac{1}{(np^3)^{1/2+o(1)}}.
\end{equation}

\noindent
On the other hand, for an upper bound, note that
\begin{equation}\label{eq:UB Z_ell}
\begin{split}
&\sum_{\ell=1}^{2k} \exp\Big(-(1+o(1))\Big((\ell-1)\sqrt{\zeta np^3\log(1/(np^3))}+\frac12(\ell-1)^2np^3\Big)\Big)\\
& \qquad\le
\sum_{\ell=1}^{2k} \exp\Big(-(1+o(1))(\ell-1)\sqrt{\zeta np^3\log(1/(np^3))}\Big)\\
& \qquad \le
\frac{1}{1-\exp\Big(-(1+o(1))\sqrt{\zeta np^3\log(1/(np^3))}\Big)} \\
& \qquad\le
\frac{1}{(np^3)^{1/2+o(1)}}.
\end{split}
\end{equation}
\noindent
\smallskip\noindent
\textbf{Locating the threshold.}
Combining~\eqref{eq:that I need} with the lower bound~\eqref{eq i_*} and upper bound~\eqref{eq:UB Z_ell} shows that $\Prob(Z^{(1)}\le T_n)$ undergoes a transition according to whether $\zeta$ is larger or smaller than $1$. This identifies the asymptotic scale of $Z^{(1)}$.
Indeed, if $\zeta > 1$, then
$\Prob(Z^{(1)}\le T_n)$ converges to 1 as $n \to \infty$, whereas if $\zeta < 1$, it converges to 0. This proves the lemma. 
\end{proof}
We next identify the range of indices in which the two largest
values among the proxy variables $Z_\ell$ can occur.
Random fluctuations can raise individual variables $Z_\ell$ above $np$ by the scale identified in Lemma~\ref{lem:max}, while the means $\mathbb E Z_\ell=z_\ell p$ decrease with $\ell$. Beyond a suitable initial range of indices, this decrease in the means of the variables $Z_\ell$ is large enough that variables with sufficiently large index are unlikely to compete with the leading values. To make this comparison precise, we restrict our attention to the initial segment $[k_*]$, where
\[
k_*=\left\lceil\frac12\sqrt{\frac{\log(1/(np^3))}{2np^3}}\right\rceil\ll 2k.
\]
The following lemma refines the proof of Lemma~\ref{lem:max} by determining the asymptotic scale of the maximum of the family $(Z_\ell)_{\ell=k_*}^{2k}$,which will allow us to compare the proxy variables indexed by $[k_*,2k]$ with the largest proxy values arising among the indices in $[k_*-1]$.
\begin{lemma}\label{lem:generalization}
  As $n\to \infty$,
\begin{equation*}
\frac{\max_{\ell\in [k_*, 2k]} Z_\ell - np}{\sqrt{np \log(1/(np^3))}} \xrightarrow{\mathbb P}1-\frac{1}{\sqrt{8}}\approx 0.646, 
\end{equation*}
and 
$$\frac{\max_{\ell\in [\lfloor k_*/2\rfloor, 2k]} Z_\ell - np}{\sqrt{np \log(1/(np^3))}} \xrightarrow{\mathbb P} 1-\frac{1}{\sqrt{32}}\approx 0.823.
$$
\end{lemma}
\begin{proof}
\added[id=ai]Repeating the calculation leading to~\eqref{eq:that I need}, but restricting the
indices to $[k_*,2k]$, reduces the problem to estimating
\[
\sum_{\ell=k_*}^{2k}
\exp\Big(-(1+o(1))\Big(
(\ell-1)\sqrt{\zeta np^3\log(1/(np^3))}
+\frac12(\ell-1)^2np^3\Big)\Big).
\]

The summands form a decreasing sequence. As in~\eqref{eq i_*}, retaining the first $\ell_*$ terms after $k_*$ gives a lower bound. Thus, on the one hand,
\begin{align*}
& \sum_{\ell=k_*+1}^{2k}\exp\Big(-(1+o(1))\Big((\ell-1)\sqrt{\zeta np^3\log(1/(np^3))}+\frac{1}{2}(\ell-1)^2 np^3\Big)\Big) \\
& \quad \geq 
\sum_{\ell=k_*+1}^{k_*+\ell_*}\exp\Big(-(1+o(1))\Big((\ell-1)\sqrt{\zeta np^3\log(1/(np^3))}+\frac{1}{2}(\ell-1)^2np^3\Big)\Big) \\
& \quad \geq 
\ell_* \exp\Big(-(1+o(1))\Big((k_*+\ell_*)\sqrt{\zeta np^3\log(1/(np^3))}+\frac{1}{2}(k_*+\ell_*)^2np^3\Big)\Big)\\
& \quad =
\Big(\frac{1}{np^3}\Big)^{(1-\sqrt{\zeta/2}-1/8+o(1))/2}
\end{align*}
and, on the other hand, 
\begin{align*}
& \sum_{\ell=k_*+1}^{2k}\exp\Big(-(1+o(1))\Big((\ell-1)\sqrt{\zeta np^3\log(1/(np^3))}+\tfrac{1}{2}(\ell-1)^2np^3\Big)\Big) \\
& \quad \leq 
\sum_{\ell=k_*+1}^{2k}\exp\Big(-(1+o(1))(\ell-1)\Big(\sqrt{\zeta np^3\log(1/(np^3))}+\tfrac{1}{2}k_* np^3\Big)\Big) \\ 
& \quad \leq
\frac{\exp\Big(-(1+o(1))k_*\Big(\sqrt{\zeta np^3\log(1/(np^3))}+k_* np^3/2\Big)\Big)}{1-\exp\Big(-(1+o(1))\Big(\sqrt{\zeta np^3\log(1/(np^3))}+k_* np^3/2\Big)\Big)} \\
& \quad =
\Big(\frac{1}{np^3}\Big)^{(1-\sqrt{\zeta/2}-1/8+o(1))/2}.
\end{align*}
The omitted term with $\ell=k_*$ has order $(np^3)^{\sqrt{\zeta/8}+1/16+o(1)}=o(1)$ and is therefore negligible compared to the total sum above (which is of order $\omega(1)$).
Restoring that term therefore leaves the estimate unchanged. Thus, 
\[
\Prob\Big(\max_{\ell\in[k_*,2k]}Z_\ell\le T_n\Big) =\prod_{\ell=k_*}^{2k}\Prob(Z_\ell\leq T_n) = \exp(-(np^3)^{(\zeta - 1+\sqrt{\zeta/2}+1/8)/2 + o(1)}).
\]
The exponent changes sign when $\zeta+\sqrt{\zeta/2}-7/8=0$, or equivalently when $(\sqrt\zeta+1/\sqrt8)^2=1$. 
Thus, the threshold is $\sqrt\zeta=1-1/\sqrt8$. Since $T_n=np+\sqrt\zeta\sqrt{np\log(1/(np^3))}$, for $\sqrt\zeta<1-1/\sqrt8$, the displayed probability tends to $0$, and for $\sqrt\zeta>1-1/\sqrt8$ the probability tends to $1$, proving the first limit. 
Then, replacing $k_*$ by $\lfloor k_*/2\rfloor$ divides the linear contribution with respect to $k_*$ (see the third lines of the last multi-line display) by 2 and the quadratic contribution by 4. This leads to the similar equation $\zeta+\sqrt{\zeta/8}-31/32=0$, or equivalently $(\sqrt\zeta+1/\sqrt{32})^2=1$, which gives the second limit.
\end{proof}
Lemma~\ref{lem:generalization} indicates that the two largest values of the the sequence $(Z_\ell)_{\ell=1}^{2k}$ should both occur among the variables with relatively small indices. We now make this localization precise.
To this end, we introduce the largest and second-largest values of $Z_\ell$ within the initial segment $[k_*]$ and define $\ell_1$ to be the first index at which the maximum is attained. 
The next corollary shows that, a.a.s., the first and second maxima of the families $(Z_\ell)_{\ell=1}^{2k}$ and $(Z_\ell)_{\ell=1}^{k_*}$ coincide.

Let
\begin{equation}\label{def l1}
  Z^{(1)}_*=\max_{\ell\in[k_*]} Z_\ell
  \qquad\text{and}\qquad
  \ell_1=\min\{\ell\in[k_*]: Z_\ell=Z_*^{(1)}\}.
\end{equation}
Thus $\ell_1$ is the first index at which the maximum over the initial segment $[k_*]$ is attained.
We then define
\[
Z^{(2)}_*=\max_{\ell\in[k_*]\setminus\{\ell_1\}} Z_\ell.
\]
Thus, $Z_*^{(2)}$ is the second-largest value among the variables
$(Z_\ell)_{\ell\in[k_*]}$, counting multiplicities. In particular, if the maximum is attained more than once, then $Z_*^{(2)}=Z_*^{(1)}$.

\begin{corollary}\label{cor second max}
A.a.s.\ $Z^{(1)} = Z^{(1)}_*$ (equivalently, $\ell_1=\min\{\ell\in [2k] \colon Z_\ell=Z^{(1)}\}$) and $Z^{(2)} = Z^{(2)}_*$.
\end{corollary}
\begin{proof}
  We show that a.a.s.\ $Z^{(2)}_* > \max_{\ell>k_*} Z_\ell$, which implies the statement of the corollary. Firstly, Lemma~\ref{lem:max} together with the second conclusion of Lemma~\ref{lem:generalization} implies that a.a.s.
$Z^{(1)}>\max_{\ell\in[\lfloor k_*/2\rfloor,2k]}Z_\ell$.
Hence every index $\ell$ such that $Z_\ell=Z^{(1)}$ is smaller than 
$\lfloor k_*/2\rfloor$.

On the other hand, the two conclusions of Lemma~\ref{lem:generalization} imply that a.a.s.~$\max_{\ell\in[\lfloor k_*/2\rfloor,2k]}Z_\ell>
\max_{\ell\in[k_*,2k]}Z_\ell$.
Consequently, there exists an index $\lfloor k_*/2\rfloor\le j<k_*$, such that
$Z_j>\max_{\ell\in[k_*,2k]}Z_\ell$.
Since $j$ is different from every index at which $Z^{(1)}$ is attained, it follows that
\[
Z_*^{(2)}\ge Z_j>\max_{\ell\in[k_*,2k]}Z_\ell.
\]
This proves the corollary.
\end{proof}

The localization step is now complete. It remains to show that $Z_*^{(1)}-Z_*^{(2)}$ is sufficiently large.
The next lemma provides the monotonicity used in this estimate: after fixing the first index at which the maximum is attained, conditioning on a larger maximum makes the desired separation more likely.
\begin{lemma}\label{lem decreasing}
  Fix $n\in\mathbb N$, $t\in[n]$, and $\ell\in[k_*]$.
  For those integers $s\ge t$ for which the conditioning event has positive probability, the function
\[
s\longmapsto
\Prob\big(
Z_*^{(2)}\le Z_*^{(1)}-t
\mid Z_*^{(1)}=s,\ell_1=\ell
\big)
\]
is increasing.
\end{lemma}
\begin{proof}
Fix $\ell\in[k_*]$ and condition on the event $\ell_1=\ell$. We have 
\begin{align*}
\Prob(Z_*^{(2)}\le Z_*^{(1)} - t\mid Z_*^{(1)} = s, \ell_1=\ell) 
&= \prod_{j\in [\ell-1]} \Prob(Z_j\le s-t\mid Z_j\le s-1)\\ &\hspace{10em}\cdot \prod_{j=\ell+1}^{k_*} \Prob(Z_j\le s-t\mid Z_j\le s),
\end{align*}
where we used the independence of the random variables $(Z_j)_{j=1}^{k_*}$. Now, fix $j\in [\ell+1, k_*]$. On the one hand, given positive integers $t$ and $s\ge t$,
\begin{equation*}
    \Prob(Z_j\le s-t\mid Z_j\le s) = \frac{\Prob(Z_j\le s-t)}{\Prob(Z_j\le s)}
\end{equation*}
and
\begin{equation*}
\Prob(Z_j\le s+1-t\mid Z_j\le s+1) = \frac{\Prob(Z_j\le s+1-t)}{\Prob(Z_j\le s+1)}.
\end{equation*}
On the other hand, it is well-known that binomial random variables have log-concave probability mass functions. Hence, the cumulative distribution function of $Z_j$, say $F$, is also log-concave (see for instance Proposition~1-1~(ii) in~\cite{Ros02}). Then, observing that for $\lambda=(t+1)^{-1}$ we have that
$s=\lambda(s-t)+(1-\lambda)(s+1)$ and $s+1-t=(1-\lambda)(s-t)+\lambda (s+1)$, it follows that
\[
\log F(s-t)+\log F(s+1) \leq \log F(s+1-t)+\log F(s),
\]
which is equivalent to 
\[\frac{\Prob(Z_j\le s-t)}{\Prob(Z_j\le s)} = \frac{F(s-t)}{F(s)} \le \frac{F(s+1-t)}{F(s+1)} = \frac{\Prob(Z_j\le s+1-t)}{\Prob(Z_j\le s+1)},\]
thereby concluding the analysis of the $j\in[\ell+1,k_*]$ case.

The case $j\in[\ell-1]$ is analogous: for $t\ge2$, log-concavity gives $F(s-t)/F(s-1)\le F(s+1-t)/F(s)$; for $t=1$, this holds with equality. Hence every factor in the product above is increasing in $s$, which establishes the lemma.
\end{proof}

By Lemma~\ref{lem:max} one can define a sequence of positive real numbers $(\varepsilon_n)_{n \ge 1}$ converging to 0 and such that, on the one hand, $\eps_n\ge (\log(np))^{-1/2}$ for all sufficiently large $n$, and moreover
$$M_n = np+\sqrt{(1-\tfrac{\varepsilon_n}{2})np\log(1/(np^3))}$$
is a.a.s.\ smaller than $Z^{(1)}_*$ (the rounding of $M_n$ to an integer is implicit, as usual). Also, define 
\[
\gamma_n=(np)^{1/2-\varepsilon_n}.
\]
The scale $\gamma_n$ is slightly below the natural fluctuation scale $(np)^{1/2}$ and will serve as the separation margin between the largest and second-largest proxy basins.

\begin{lemma}\label{lem fst snd max}
$\displaystyle\Prob\big( Z_*^{(2)} \le Z_*^{(1)}-2\gamma_n\big) =1-o(1).$
\end{lemma}
\begin{proof}
The proof idea is to condition on the maximum being at least a deterministic level $M_n$ slightly below its typical value. After fixing the first index at which the maximum is attained, Lemma~\ref{lem decreasing} allows us to reduce to the case where the maximum equals $M_n$. We then compare the probabilities that the remaining variables lie at the nearby levels $M_n-j$, $0\le j\le2\gamma_n$, and finally show that their total contribution is $o(1)$. This implies that a.a.s.\ no other $Z_\ell$ lies within distance $2\gamma_n$ of the maximum.

\smallskip\noindent
\textbf{Reduction to a deterministic level.}
By Corollary~\ref{cor second max}, a.a.s.~$\ell_1=\min\{\ell\in [2k] \colon Z^{(1)} = Z_\ell\}$.
Using Lemma~\ref{lem decreasing} and the definition of the sequence $(\eps_n)_{n\ge1}$, we obtain
\begin{align}
&\Prob\big(Z_*^{(2)}\le Z_*^{(1)}-2\gamma_n\big)
\nonumber\\
&\qquad\ge
\Prob(Z_*^{(1)}\ge M_n)
\sum_{\ell=1}^{k_*}
\Prob(\ell_1=\ell\mid Z_*^{(1)}\ge M_n)
\Prob\big(
Z_*^{(2)}\le Z_*^{(1)}-2\gamma_n
\mid Z_*^{(1)}\ge M_n,\ell_1=\ell
\big)
\nonumber\\
&\qquad\ge
(1-o(1))
\sum_{\ell=1}^{k_*}
\Prob(\ell_1=\ell\mid Z_*^{(1)}\ge M_n)
\Prob\big(
Z_*^{(2)}\le M_n-2\gamma_n
\mid Z_*^{(1)}=M_n,\ell_1=\ell
\big)
\nonumber\\
&\qquad =
(1-o(1))
\sum_{\ell=1}^{k_*}
\Prob(\ell_1=\ell\mid Z_*^{(1)}\ge M_n)
\prod_{j\in[k_*]\setminus\{\ell\}}
\Prob(Z_j\le M_n-2\gamma_n\mid \ell_1=\ell,Z_\ell=M_n)
\nonumber\\
&\qquad \ge
(1-o(1))
\prod_{\ell=1}^{k_*}
\Prob(Z_\ell\le M_n-2\gamma_n\mid Z_\ell\le M_n).
\label{eq M_n optimal}
\end{align}
where in the last inequality we used that, by independence of $(Z_\ell)_{\ell=1}^{k_*}$, for every $\ell\in[k_*]$, the product in the preceding line rewrites as
\begin{align*}
\prod_{j=1}^{\ell-1}\Prob(Z_j\le M_n-2\gamma_n\mid Z_j\le M_n-1)
\prod_{j=\ell+1}^{k_*}\Prob(Z_j\le M_n-2\gamma_n\mid Z_j\le M_n).
\end{align*}
In particular, since for $j<\ell$ the probability of $Z_j\le M_n-2\gamma_n$
conditional on $Z_j\le M_n-1$ is at least as large as conditioning on $Z_j\le M_n$, this product is at least
\[
\prod_{j\in[k_*]\setminus\{\ell\}}
\Prob(Z_j\le M_n-2\gamma_n\mid Z_j\le M_n),
\]
and hence, since all factors lie in $[0,1]$, is uniformly bounded from below by
\[\prod_{j=1}^{k_*}
\Prob(Z_j\le M_n-2\gamma_n\mid Z_j\le M_n).\]
Since the sum of $\Prob(\ell_1=\ell\mid Z_*^{(1)}\ge M_n)$
over $\ell\in [k_*]$ adds to 1,
this yields~\eqref{eq M_n optimal}.
Moreover, since $\Prob(Z_\ell\le M_n)=1-o(1)$ uniformly over $\ell\in[k_*]$, the product in~\eqref{eq M_n optimal} can be written as
\begin{equation}\label{comp:terms1}
\begin{split}
\prod_{\ell=1}^{k_*} \big(1-\Prob\big(Z_\ell \in [M_n-2\gamma_n+1
&, M_n]\mid Z_\ell\le M_n\big)\big)\\
&=\prod_{\ell=1}^{k_*} \left(1-(1+o(1))
\Prob\big(Z_\ell\in[M_n-2\gamma_n+1,M_n]\big)\right).
\end{split}
\end{equation}


\smallskip\noindent
\textbf{Comparing nearby point probabilities.}
We now show that the final product in~\eqref{comp:terms1} is $1-o(1)$. The first step is to show that, uniformly in $\ell\in[k_*]$, the probabilities
$\Prob(Z_\ell=M_n-j)$, $0\le j\le 2\gamma_n$, are all of the same order.
In fact, we only show that the terms $\Prob(Z_\ell=M_n)$ and $\Prob(Z_\ell=M_n-2\gamma_n)$ are of the same order, the computation for the remaining ones being analogous. Indeed, recalling that~$q=1-p$, note that
\begin{equation}\label{compareterms}
\frac{\Prob(Z_\ell=M_n - 2\gamma_n)}{\Prob(Z_\ell=M_n)}=\frac{M_n(M_n-1)\cdots(M_n-2\gamma_n+1)\onemp^{2\gamma_n}}{(z_\ell-M_n+1)\cdots(z_\ell-M_n+2\gamma_n)p^{2\gamma_n}},
\end{equation}
and also
$$M_n^{2\gamma_n}\Big(1-\frac{1}{M_n}\Big)\cdots\Big(1-\frac{2\gamma_n-1}{M_n}\Big)=M_n^{2\gamma_n}\exp\Big(-(1+o(1))\frac{4\gamma_n^2}{2M_n}\Big)=(1-o(1))M_n^{2\gamma_n}.$$
Furthermore, since $\gamma_np=o(1)$, $\onemp^{2\gamma_n} = \exp(-(1+o(1))2\gamma_np)=1-o(1)$. Therefore,~\eqref{compareterms} is equal to
\begin{align*}
(1-o(1))\frac{M_n^{2\gamma_n}}{p^{2\gamma_n}}\prod_{i=1}^{2\gamma_n} \frac{1}{z_\ell - M_n+i}
&= (1-o(1)) \prod_{i=1}^{2\gamma_n} \frac{np+\sqrt{(1-\tfrac{\varepsilon_n}{2})np\log(1/(np^3))}}{np-(1+o(1))\ell np^2} \\
&= \Big(1+\big(1+o(1)\big)\Big(\sqrt{\tfrac{1}{np}(1-\tfrac{\varepsilon_n}{2})\log(1/(np^3))}+\ell p\Big)\Big)^{2\gamma_n} \\
&= \exp\Big(\big(2\gamma_n+o(\gamma_n)\big)\Big(\sqrt{\tfrac{1}{np}(1-\tfrac{\varepsilon_n}{2})\log(1/(np^3))}+\ell p\Big)\Big) \\
&= \exp(O((np)^{-\eps_n/2})) \\
& = 1+o(1),
\end{align*}
where in the second-to-last equality we used that $\gamma_n\ell p\le \gamma_nk_*p=o(1)$.

Using this observation in~\eqref{comp:terms1} implies that
\begin{equation}\label{eq:compterms2}
\prod_{\ell=1}^{k_*} \left(1-(1+o(1))\Prob\left(Z_\ell \in [M_n- 2\gamma_n+1, M_n]\right)\right)=\prod_{\ell=1}^{k_*}(1-(1+o(1))2\gamma_n \Prob\big(Z_\ell=M_n\big)).
\end{equation}

\noindent
\smallskip\noindent
\textbf{Bounding the total mass near $M_n$.}
Finally, to show that the total mass near $M_n$ is negligible, we estimate the order of $2\gamma_n\Prob(Z_\ell=M_n)$ uniformly over $\ell\in[k_*]$. Recall that for $m=m(s)$ such that $1\ll m\ll s$ as $s\to\infty$, it holds that
\[
\binom{s}{m}\sim \Big(\frac{se}{m}\Big)^m\frac{1}{\sqrt{2\pi m}}\exp\Big(-\frac{m^2}{2s}+O\Big(\frac{m^3}{s^2}\Big)\Big).
\]
Thus, since $M_n=(1+o(1))np$, we have 
\begin{align}
2\gamma_n\Prob(Z_\ell=M_n) 
& = 2\gamma_n\binom{z_\ell}{M_n}p^{M_n}q^{z_\ell-M_n}\nonumber\\ 
& = 2(np)^{1/2-\varepsilon_n}\Big(\frac{z_\ell pe}{M_n}\Big)^{M_n}\frac{1+o(1)}{\sqrt{2\pi np}}q^{z_\ell-M_n}\exp\Big(-\frac{M_n^2}{2z_\ell}
+O\Big(\frac{M_n^3}{z^2_\ell}\Big)\Big) \nonumber \\
& = (np)^{-\varepsilon_n}\Big(\frac{z_\ell p}{M_n}\Big)^{M_n}(1+o(1))q^{z_\ell-M_n}\sqrt{\frac{2}{\pi}}\exp\Big(M_n-\frac{M_n^2}{2z_\ell}
+O\Big(\frac{M_n^3}{z^2_\ell}\Big)\Big).
\label{eqn Z equal M}
\end{align}
Using that $1+x=\exp(x-\tfrac{x^2}{2}+O(x^3))$ as $x\to 0$ in order to bound from above  $\frac{z_\ell p}{M_n}=1+\frac{z_\ell p-M_n}{M_n}$,  we have
\begin{align}
\Big(\frac{z_{\ell}p}{M_n}\Big)^{M_n} & 
= \exp\Big(z_\ell p-M_n-\frac{(z_\ell p-M_n)^2}{2M_n}+O\Big(\frac{|z_\ell p-M_n|^3}{M_n^2}\Big)\Big), 
\label{eqn bnd binom} \\
q^{z_\ell-M_n} 
& = \exp\Big(-\Big(p+\frac{p^2}{2}+O(p^3)\Big)|z_\ell-M_n|\Big).
\label{eqn bnd q}
\end{align}
Using that $p^3=o(1/n)$, $z_\ell=(1+o(1))n$, $M_n=(1+o(1))np$,
and observing that both $p^3z_{\ell}$ and $p^2M_n$ are of order $O(np^3)=o(1)$, 
the exponent in the right-hand side of~\eqref{eqn bnd q} is 
\[
-\Big(p+\frac{p^2}{2}\Big)(z_\ell-M_n)+o(1)=-z_\ell p-\frac{(z_{\ell}p-M_n)^2}{2z_\ell}+\frac{M_n^2}{2z_\ell}+o(1).
\]
Hence, combining~\eqref{eqn Z equal M},~\eqref{eqn bnd binom} and~\eqref{eqn bnd q}, we obtain that $2\gamma_n\mathbb P(Z_\ell=M_n)$ is equal to
\begin{align*}
(np)^{-\eps_n}\exp\Big(-\frac12(z_\ell p-M_n)^2\Big(\frac{1}{z_\ell}+\frac{1}{M_n}\Big)+O\Big(\frac{M_n^3}{z_\ell^2}+\frac{|z_\ell p-M_n|^3}{M_n^2}\Big)+O(1)\Big).
\end{align*}
Next, observe that $\frac{M_n^3}{z_\ell^2}=(1+o(1))np^3=o(1)$ and

\begin{align*}
z_\ell p - M_n 
&=\; -(\ell-1)np^2+\tfrac12(\ell-1)(\ell-2)np^3-\sqrt{(1-\tfrac{\eps_n}{2})np\log(\tfrac{1}{np^3})}\\
&=\; -(\ell-1)np^2\big(1 - \tfrac12(\ell-2)p\big)
-\sqrt{(1-\tfrac{\eps_n}{2})np\log(\tfrac{1}{np^3})}.
\end{align*}
Therefore, ince $\ell p\le k_*p=o(1)$, uniformly over $\ell\le k_*$, we have
$|z_\ell p- M_n| \ge \sqrt{(1-\frac{\eps_n}{2})np\log(\frac{1}{np^3})}$.
Since $M_n=np+O(\sqrt{np\log(\frac{1}{np^3})})$ and
$(\log(\frac{1}{np^3}))^{3/2}/\sqrt{np}=o(1)$, 
we get that
\[\frac{(z_{\ell}p - M_n)^2}{2M_n}
\ge \frac{(1-\eps_n/2)np\log(\frac{1}{np^3})}{2M_n}
= \frac{1}{2}(1-\eps_n/2)\log\big(\frac{1}{np^3}\big) + o(1).
\]

Hence,
\begin{equation}\label{eq:=Mn}
\begin{split}
2\gamma_n\Prob(Z_\ell = M_n) 
&\le (np)^{-\eps_n}\exp\big(-\tfrac12(1-\tfrac{\eps_n}{2})\log(\tfrac{1}{np^3})+o(1)\big)\\
&= O((np)^{-\eps_n}(np^3)^{1/2 - \eps_n/4}) = O\Big(\frac{(np)^{-\eps_n/2}}{k_*}\Big),
\end{split}    
\end{equation}
where in the last equality we used that $(n^2p^4)^{-\eps_n/4}\le 1$ and $(np)^{-\eps_n/4}\ll \frac{1}{\sqrt{\log(1/(np^3))}}$.

Using that $1-x = \exp(-(1+o(1)) x)$ as $x\to 0$, we conclude that~\eqref{eq:compterms2} rewrites as 
\begin{equation*}
    \exp\Big(-(1+o(1)) \sum_{\ell=1}^{k_*} 2\gamma_n\mathbb P(Z_\ell = M_n)\Big) = \exp(-O((np)^{-\eps_n/2})) = 1-o(1),
\end{equation*}
which finishes the proof of the lemma.
\end{proof}

We now transfer the separation established in Lemma~\ref{lem fst snd max} for the independent variables $(Z_\ell)_{\ell=1}^{2k}$ to the first-round basin sizes $(\sLvltl{B}{1}{\ell})_{\ell=1}^{2k}$. Recall that Lemma~\ref{lem approx} couples these two families so that, a.a.s.,
$\max_{\ell\in[2k]}\big|\sLvltl{B}{1}{\ell}-Z_\ell\big|
\le \Delta_n$ where
$\Delta_n=(np)^{2/5}$.
Thus, when comparing two labels, the difference between the corresponding basin sizes can differ from the difference between the corresponding $Z_\ell$'s by at most $2\Delta_n$. Since $\Delta_n=o(\gamma_n)$, the separation obtained in Lemma~\ref{lem fst snd max} is large enough to yield the required gap between the two largest first-round basin sizes.

Recall that by definition $\sLvl{B^{(1)}} = \max_{\ell\in [2k]} \sLvltl{B}{1}{\ell}$, and define also 
\[
\sLvl{B^{(2)}} = \max_{\ell\in [2k]\setminus \{\ell_1\}} \sLvltl{B}{1}{\ell},
\]
where $\ell_1 = \min\{\ell\in [k_*] \colon Z^{(1)}_* = Z_\ell\}$.
In particular, since $k_*\ll k$, we have $\ell_1<k$ for all sufficiently large $n$.
\begin{corollary}\label{cor fst snd max}
\hspace{-0.7em} Under the coupling in Lemma~\ref{lem approx}, a.a.s.~$\sLvltl{B}{1}{\ell_1} = \sLvl{B^{(1)}}$ and $\sLvl{B^{(1)}} - \sLvl{B^{(2)}}\ge \gamma_n$.
\end{corollary}
\begin{proof}
Let $\Delta_n=(np)^{2/5}=o(\gamma_n)$.
By Lemma~\ref{lem approx}, on an a.a.s.\ event,
$\max_{\ell\in[2k]}|\sLvltl{B}{1}{\ell}-Z_\ell|\le \Delta_n$.
By Corollary~\ref{cor second max} and Lemma~\ref{lem fst snd max}, on another a.a.s.\ event, 
$Z_{\ell_1}=Z^{(1)}$ and $Z^{(1)}-Z^{(2)}\ge 2\gamma_n$.
On the intersection of these events, for every $\ell\in[2k]\setminus\{\ell_1\}$,
\[
\sLvltl{B}{1}{\ell}
\le Z_\ell+\Delta_n
\le Z^{(2)}+\Delta_n
\le Z^{(1)}-2\gamma_n+\Delta_n
\le \sLvltl{B}{1}{\ell_1}-2\gamma_n+2\Delta_n.
\]
Since $\Delta_n=o(\gamma_n)$, it follows that
$\sLvltl{B}{1}{\ell_1}=\sLvl{B^{(1)}}$
and
$\sLvl{B^{(1)}}-\sLvl{B^{(2)}}
\ge 2\gamma_n-2\Delta_n
\ge \gamma_n$
for all sufficiently large $n$.
\end{proof}


Corollary~\ref{cor fst snd max} completes our analysis of the first round of ALAP. We collect below the notation for the winning label and the two largest first-round basins, which will be used in the analysis of the subsequent rounds. These descriptions are understood to hold on the a.a.s.\ event of that corollary.
\begin{center}
\renewcommand{\arraystretch}{1.15}
\begin{tabular}{@{}p{0.29\textwidth}p{0.65\textwidth}@{}}
\hline
Notation & Meaning \\
\hline
$\ell_1$ & Label of the largest first-round basin. \\
$\Lvl{B^{(1)}}$ & Winning (largest) first round basin. \\
$\Lvl{B^{(2)}}$ & Largest first-round basin among labels other than $\ell_1$. \\
$\gamma_n=(np)^{1/2-\varepsilon_n}$ & Guaranteed lower bound on $\sLvl{B^{(1)}}-\sLvl{B^{(2)}}$, with $\varepsilon_n\to 0$. \\
\hline
\end{tabular}
\end{center}

\paragraph{Second-round amplification in level~$C$.}

We next show that the first-round basin gap from Corollary~\ref{cor fst snd max} is amplified during the second round of ALAP inside level~$C$. For a vertex in $C$, the comparison between labels $\ell_1$ and $\ell$ is governed by the number of edges it sends to the corresponding first-round basins. Since $B_1(\ell_1)$ is larger than every competing basin by at least $\gamma_n$, the monotonicity lemma for tie-weighted binomial comparisons (Lemma~\ref{lem:bis}) gives a bias toward label~$\ell_1$ of order
\[
\frac{\gamma_n p}{\sqrt{np^2}}
= p^{1/2}(np)^{-\eps_n}.
\]
The next result turns this pointwise bias into a simultaneous comparison of the second-round label counts in level~$C$ and establishes that $\ell_1$ is in fact the most represented label among level~$C$ vertices after the second round of ALAP.
\begin{lemma}\label{claim 3'}
There exists a constant $c_2>0$ such that after the second round of ALAP, a.a.s.\ $\sLvltl{C}{2}{\ell_1}\ge \frac{n}{4k}$
and, for all $\ell\in[k]\setminus\{\ell_1\}$,
\[
\sLvltl{C}{2}{\ell}
\le
\big(1-c_2p^{1/2}(np)^{-\eps_n}\big)
\sLvltl{C}{2}{\ell_1}.
\]
\end{lemma}
\begin{proof}
Let us condition on the a.a.s.\ event that $\sLvl{B}\le 3knp$ (see Lemma~\ref{claim F}). 
Conditionally on the first-round exposure, the binomial random variable counting the number of vertices in $C$ with (the dominant) label $\ell_1$ after two rounds stochastically dominates the number of vertices in $C$ with any other label.
Comparing with $2k$ i.i.d.\ counts of this largest size shows that every vertex in $C$ receives label $\ell_1$ independently with probability at least $1/2k$. 
Consequently, $\sLvltl{C}{2}{\ell_1}$ dominates a binomial variable with parameters $|C|\ge n-2k-3knp=(1-o(1))n$ and $1/2k$. By Chernoff's bound,
\begin{equation*}
    \Prob(\sLvltl{C}{2}{\ell_1}\le n/4k)
    \le \exp\Big(-\frac{(\mathbb E \sLvltl{C}{2}{\ell_1}-n/4k)^2}{2\mathbb E \sLvltl{C}{2}{\ell_1}}\Big) = \exp(-\Omega(n/k)) = o(1).
\end{equation*}
Fix $\ell\in [k]\setminus \{\ell_1\}$.
Then, the previous display shows a.a.s.\ we have $|\Lvltl{C}{2}{\{\ell_1,\ell\}}|\ge n/4k$. Next, we condition on this event and on the set $\Lvltl{C}{2}{\{\ell_1,\ell\}}$ itself.

By Corollary~\ref{cor fst snd max}, we know that a.a.s.\ $\sLvl{B^{(1)}} - \sLvl{B^{(2)}}\ge \gamma_n$, and moreover, combining Lemma~\ref{lem:max} and Corollary~\ref{cor fst snd max} implies that $\sLvl{B^{(1)}}\le M_n^+ = np + \sqrt{\frac32 np\log(1/(np^3))}$. Let us condition on these events as well. Then, for a vertex $u$ in level~$C$, Lemma~\ref{lem:bis} implies that
\begin{align*}
\frac{\mathbb P(u\in\Lvltl{C}{2}{\ell_1}\mid u\in \Lvltl{C}{2}{\{\ell_1,\ell\}})}{\mathbb P(u\in\Lvltl{C}{2}{\ell}\mid u\in \Lvltl{C}{2}{\{\ell_1,\ell\}})}
& \\
&\hspace{-10em}\ge \frac{\mathbb P(\mathrm{Bin}(\sLvltl{B}{1}{\ell_1}, p) > \mathrm{Bin}(\sLvltl{B}{1}{\ell}, p)) + \tfrac{1}{2}\mathbb P(\mathrm{Bin}(\sLvltl{B}{1}{\ell_1}, p) = \mathrm{Bin}(\sLvltl{B}{1}{\ell}, p))}{\mathbb P(\mathrm{Bin}(\sLvltl{B}{1}{\ell}, p) > \mathrm{Bin}(\sLvltl{B}{1}{\ell_1}, p)) + \tfrac{1}{2}\mathbb P(\mathrm{Bin}(\sLvltl{B}{1}{\ell}, p) = \mathrm{Bin}(\sLvltl{B}{1}{\ell_1}, p))}.    
\end{align*}
By replacing the denominators in the latter fraction with the sum of the corresponding numerator and denominator and using the inequality $\sLvl{B^{(1)}} - \sLvl{B^{(2)}}\ge \gamma_n$ we obtain that, in our procedure, the probability that a vertex in $\Lvltl{C}{2}{\{\ell_1,\ell\}}$ gets label $\ell_1$, conditionally on obtaining either $\ell$ or $\ell_1$,
is bounded from below by 
\begin{equation*}
    \alpha = \Prob\big(\mathrm{Bin}(\sLvl{B^{(1)}}, p) > \mathrm{Bin}(\sLvl{B^{(1)}}-\gamma_n, p)\big) + \tfrac{1}{2} \Prob\big(\mathrm{Bin}(\sLvl{B^{(1)}}, p) = \mathrm{Bin}(\sLvl{B^{(1)}}-\gamma_n, p)\big).
\end{equation*} 

Remark~\ref{rem binomials'} then implies that $\alpha$ is bounded from below by 
\[\frac{1}{2} + \Omega\bigg(\frac{\gamma_n p}{\sqrt{\sLvl{B^{(1)}}p}}\bigg) = \frac{1}{2} + \Omega\bigg(\frac{p^{1/2}}{(np)^{\eps_n}}\bigg).\] 
Hence, using that conditionally on $\sLvltl{C}{2}{\{\ell_1,\ell\}}$ we have $\mathbb E \sLvltl{C}{2}{\ell_1} \geq \alpha \sLvltl{C}{2}{\{\ell_1,\ell\}}$, by Chernoff's bound the number of vertices in $\Lvltl{C}{2}{\{\ell_1,\ell\}}$ getting label $\ell_1$ in our procedure satisfies
\begin{equation*}
\Prob\Big(\sLvltl{C}{2}{\ell_1} \le \frac{1}{2}\Big(\frac{1}{2} + \alpha\Big)\sLvltl{C}{2}{\{\ell_1,\ell\}}\Big)=\exp\Big(-\Omega\Big(\Big(\alpha - \frac{1}{2}\Big)^2\sLvltl{C}{2}{\{\ell_1,\ell\}}\Big)\Big) = o(1/n).
\end{equation*}

Hence, by a union bound we conclude that a.a.s.\ for every $\ell\in [k]\setminus \{\ell_1\}$, the difference between the number of vertices with label $\ell_1$ and $\ell$ in level~$C$ after the second round is at least 
\[(\alpha - 1/2)\sLvltl{C}{2}{\{\ell_1,\ell\}} = \Omega(p^{1/2} (np)^{-\eps_n} \sLvltl{C}{2}{\ell_1}).\qedhere\]
\end{proof}

It remains to analyze the effect of the second round of ALAP on the vertices in level~$B$. 
We start with a simple claim. Recall the constant $K$ from Lemma~\ref{lemma 2}.
\begin{claim}\label{cl:S}
The following event holds with probability $1-o(1)$ in the probability space generated by the edges incident to $A$: conditionally on all edges incident to $A$, every vertex in $B\cap ([n]\setminus [K])$ has probability $1-o(1)$ (in the probability space generated by the edges in $B$) to receive a label in $[k]$ in the second round of ALAP. 
\end{claim}
\begin{proof}
Denote by $E_A$ the random set of edges incident to $A$.
Also, denote by $\cA_v$ the event that the vertex $v\notin \{v_1,\ldots,v_K\}$ receives a label outside $[k]$ after the second round of ALAP.
Then, with very minor modifications in the proof of Lemma~\ref{lemma 2} (which was proved for LPA and not for ALAP, but the argument transfers),
$\mathbb P(\cA_v)=o(n^{-1})$. Since $\mathbb P(\cA_v) = \mathbb E[\mathbb P(\cA_v\mid E_A)]$, we must have that with probability $1-o(n^{-1})$ (in the probability space generated by $E_A$), $\mathbb P(\cA_v\mid E_A)=o(1)$. 
Finally, a union bound over all vertices $v\in [n]\setminus [K]$ shows that, with probability $1-o(1)$ (in the probability space generated by $E_A$), every vertex in $B\cap ([n]\setminus [K])$ has probability only $o(1)$ to obtain a label outside $[k]$ after round 2 of ALAP (in the probability space generated by the edges in $B$), as desired.
\end{proof}

For a fixed competing label $\ell\in [k]\setminus\{\ell_1\}$, we separate level~$B$ into a good part, where the comparison between labels $\ell_1$ and $\ell$ is still governed by fresh edges to the corresponding first-round basins, and an exceptional part, where previously exposed edges incident to level~$A$ may affect this comparison.

For every $\ell\in[k]\setminus\{\ell_1\}$, define the good part of level~$B$ by removing from $B$ the basins $B_1(\ell_1)$ and $B_1(\ell)$, as well as the basins indexed by vertices of $A$ that are adjacent inside $A$ to either $v_{\ell_1}$ or $v_\ell$:
\[
\overline{\Lvl{B}}_{\ell_1,\ell}
=
B\setminus
\Big(
B_1(\ell_1)\cup B_1(\ell)
\cup
\bigcup_{i\in[2k] \colon \{v_i v_{\ell_1},v_i v_\ell\}\cap E(G_n)\neq\varnothing}
B_1(i)
\Big).
\]
Thus, for vertices in $\overline{\Lvl{B}}_{\ell_1,\ell}$, the comparison between labels $\ell_1$ and $\ell$ is not affected by the few level~$A$ adjacencies involving $v_{\ell_1}$ or $v_\ell$.

\begin{lemma}\label{lem Level B}
There exists a constant $c_3>0$ such that a.a.s., after the second round of ALAP, for every
$\ell\in[k]\setminus\{\ell_1\}$,
\[
|\Lvltl{B}{2}{\ell_1}\cap \overline{\Lvl{B}}_{\ell_1,\ell}|
-
|\Lvltl{B}{2}{\ell}\cap \overline{\Lvl{B}}_{\ell_1,\ell}|
\ge
c_3\gamma_n n^{-1/2}
|\Lvltl{B}{2}{\ell_1}\cap \overline{\Lvl{B}}_{\ell_1,\ell}|
\ge
\tfrac12 c_3\gamma_n p\sqrt n.
\]
\end{lemma}
\begin{proof}
We start by listing several events and random variables measurable in terms of the edges incident to $A$. 
\begin{itemize}
    \item For every label $\ell\in [2k]$, $\sLvltl{A}{1}{\ell}\le 10$ (this holds a.a.s.\ by Claim~\ref{ob 3.5}).
    \item For every $\ell\in [2k]$, $2np\ge \sLvltl{B}{1}{\ell}$, and $\sLvltl{B}{1}{[k+1,2k]}\ge 2knp/3$. Both of these events are a.a.s.\ by Chernoff's bound (applied as in Lemma~\ref{claim F}), and a union bound over all $2k$ vertices in $A$ in the first case. 
    \item Denote by $S\subseteq [k+1,2k]$ the set of indices $i$ of vertices $v_i$ such that $N(v_i)\cap \{v_1,\ldots,v_k\}=\emptyset$, and note that a.a.s.\ $|S|=k-o(k)$ thanks to Corollary~\ref{ob 3.5}.
    \item Conditionally on the edges incident to $A$ exposed over the first round of the ALAP, at the second round, every vertex in $[n]\setminus [K]$ has probability $1-o(1)$ to receive a label in $[k]$. This event was shown to hold a.a.s.\ in Claim~\ref{cl:S}.
    \item The event from Corollary~\ref{cor fst snd max} and the (random) index $\ell_1$ of the largest basin. Note that, using Lemma~\ref{lem approx}, we consciously abuse notation and refer to $\ell_1$ as the index of the largest basin, despite it being defined only for the coupled binomial process in~\eqref{def l1}. 
\end{itemize}
Note that all events above hold a.a.s.\ and measurable with respect to the first-round exposure. We thus reveal the edges incident to $A$ and condition on a configuration where all of them are satisfied.

We first obtain the lower bound on the number of vertices carrying label~$\ell_1$ that will be needed below. By the events on which we conditioned above, $\sLvltl{B}{1}{S}\ge (2k/3-o(k))np$. 
Conditionally on the first-round information, the corresponding neighbor counts in the basins $(B_1(i))_{i\in[k]}$ are independent binomial random variables, and $\sLvltl{B}{1}{\ell_1}$ is at least as large as any of the basins $\sLvltl{B}{1}{i}$, $i\in[k]$. 
Moreover, the round-$2$ choices of labels of the vertices $v\in B_1(S)$ are independent, since the corresponding sets of still-unexposed edges are disjoint.

We argue that the probability that a vertex in $B_1(S)$ picks the label $\ell_1$ at the second round of ALAP is at least $(1-o(1))/k$. 
On the one hand, conditionally on the first-round exposure, we know that every vertex $v\in B_1([k+1,2k])$ a.a.s.\ obtains a label in $[k]$ at the second round of ALAP. 
On the other hand, since no vertex in $S$ is incident to vertices with labels in $[k]$ via edges exposed at the first round, the probability of receiving label $\ell_1$ is highest among all such labels.
Combining the above two observations shows the desired conclusion.

Now, since
$|V([K]\setminus [2k])|\le K=
O\big(\frac1p\log n\big)
=o(np)$,
the number $|B_1(S)\cap B_2(\ell_1)|$ is at least the number of vertices of $B_1(S)$ which obtain label $\ell_1$ at round 2, minus $K$.
Recalling that $\sLvltl{B}{1}{S}\ge (2k/3-o(k))np$, by Chernoff's bound, with probability $1-\exp(-\Omega(np))$, the number of vertices of $B_1(S)$ obtaining label $\ell_1$ is at least $3np/5$. Since $K=o(np)$, a union bound over $\ell\in[k]\setminus\{\ell_1\}$ gives, a.a.s.\ simultaneously for every such $\ell$,
\[
|\Lvltl{B}{2}{\{\ell_1,\ell\}}\cap \overline{\Lvl{B}}_{\ell_1,\ell}|
  \ge
|\Lvltl{B}{2}{\ell_1}\cap\overline{\Lvl{B}}_{\ell_1,\ell}|
\ge
|B_1(S)\cap\Lvltl{B}{2}{\ell_1}|
\ge
3np/5-K\ge
np/2.
\]

We next establish the required quantitative advantage of label $\ell_1$ over label $\ell$ among the vertices in $\Lvltl{B}{2}{\{\ell_1,\ell\}}\cap \overline{\Lvl{B}}_{\ell_1,\ell}$. Fix $\ell\in [k]\setminus \{\ell_1\}$.
We condition on the set $\Lvltl{B}{2}{\{\ell_1,\ell\}}\cap\overline{\Lvl{B}}_{\ell_1,\ell}$ and on the event $|\Lvltl{B}{2}{\{\ell_1,\ell\}}\cap\overline{\Lvl{B}}_{\ell_1,\ell}|\ge np/2$.
Given a vertex $v\in \Lvltl{B}{2}{\{\ell_1,\ell\}}\cap\overline{\Lvl{B}}_{\ell_1,\ell}$,
recall that apart from the level-$A$ incidences already exposed during the first round, its edges towards vertices carrying labels $\ell_1$ and $\ell$ on levels $A$ and $B$ remain unexposed, and the numbers $a_1$ and~$a_2$ of these vertices are among $\sLvltl{B}{1}{\ell_1}\pm 11$ and $\sLvltl{B}{1}{\ell}\pm 11$, respectively, where, thanks to the conditioning in the beginning of the proof, 11 is an upper bound on the number of vertices on level~$A$ with any given label.
By Corollary~\ref{cor fst snd max},
\[
a_1-a_2
\ge
\sLvltl{B}{1}{\ell_1}-\sLvltl{B}{1}{\ell}-22
\ge \gamma_n-22.
\]
Set $t_n=\lfloor\gamma_n/2\rfloor$. For all sufficiently large $n$, we therefore have $a_2\le a_1-t_n$. 
Moreover, the events we conditioned on at the beginning of the proof imply $a_1\leq 2np+11=O(np)$, whereas
$\sLvltl{B}{1}{\ell_1}
\ge
\frac{1-o(1)}{k}|B_1(S)|
\ge \frac{2np}{3}=\Omega(np)$, so $a_1=\Theta(np)$.
Thus $(a_1-t_n)p=\Theta(np^2)\to\infty$, while
$t_np=O(\gamma_np)=o(1)$,
where the last relation follows from $np\ll n^{2/3}$. 
Hence, Remark~\ref{rem binomials'} applies with parameters $a_1$ and $a_1-t_n$. Using Lemma~\ref{lem:bis} as in the proof of Lemma~\ref{claim 3'} and stochastic monotonicity in the number of trials of the competing binomial, we obtain
\begin{align*}
& \Prob(v\in\Lvltl{B}{2}{\ell_1} \mid v\in \Lvltl{B}{2}{\{\ell_1,\ell\}}\cap\overline{\Lvl{B}}_{\ell_1,\ell}) \\
& \qquad \ge
\Prob\big(\mathrm{Bin}(a_1,p)>\mathrm{Bin}(a_1-t_n,p)\big)
+\frac12\Prob\big(\mathrm{Bin}(a_1,p)=\mathrm{Bin}(a_1-t_n,p)\big)
\\
&\qquad =
\frac12+\Omega\Big(\frac{t_np}{\sqrt{(a_1-t_n)p}}\Big)\ge \frac{1}{2}+2c_3\gamma_n n^{-1/2},
\end{align*}
where $c_3 > 0$ is a sufficiently small absolute constant.

We will in the sequel use the following strategy: before conditioning on $\Lvltl{B}{2}{\{\ell_1,\ell\}}\cap\overline{\Lvl{B}}_{\ell_1,\ell}$, expose all remaining edges except those joining vertices of $\overline{\Lvl{B}}_{\ell_1,\ell}$ to vertices carrying labels $\ell_1$ or $\ell$ on levels $A$ and $B$. 
Since $\overline{\Lvl{B}}_{\ell_1,\ell}$ is disjoint from $B_1(\ell_1)\cup B_1(\ell)$, the collections of edges left unexposed for distinct vertices of $\overline{\Lvl{B}}_{\ell_1,\ell}$ are disjoint and hence independent. Conditional on this exposure and on the set $\Lvltl{B}{2}{\{\ell_1,\ell\}}\cap\overline{\Lvl{B}}_{\ell_1,\ell}$, the choices between $\ell_1$ and $\ell$ for its vertices are therefore independent, and each chooses $\ell_1$ with probability at least $\frac12+2c_3\gamma_nn^{-1/2}$.

Consequently, using the strategy described in the preceding paragraph and conditionally on the set $\Lvltl{B}{2}{\{\ell_1,\ell\}}\cap\overline{\Lvl{B}}_{\ell_1,\ell}$, the number of its vertices receiving label $\ell_1$ stochastically dominates
\[
\mathrm{Bin}\big(
|\Lvltl{B}{2}{\{\ell_1,\ell\}}\cap\overline{\Lvl{B}}_{\ell_1,\ell}|,
\tfrac12+2c_3\gamma_nn^{-1/2}
\big).
\]
Furthermore, conditionally on the set $\Lvltl{B}{2}{\{\ell_1,\ell\}}\cap \overline{\Lvl{B}}_{\ell_1,\ell}$ and the event $|\Lvltl{B}{2}{\{\ell_1,\ell\}}\cap \overline{\Lvl{B}}_{\ell_1,\ell}|\geq np/2$, and since $\gamma_n = (np)^{1/2-o(1)}$, Chernoff's bound implies that 
\begin{align*}
&\Prob\big(|\Lvltl{B}{2}{\ell_1}\cap\overline{\Lvl{B}}_{\ell_1,\ell}|\le |\Lvltl{B}{2}{\ell}\cap\overline{\Lvl{B}}_{\ell_1,\ell}|
+2c_3\gamma_n n^{-1/2}|\Lvltl{B}{2}{\{\ell_1,\ell\}}\cap\overline{\Lvl{B}}_{\ell_1,\ell}|\big)
\\
&\qquad=\Prob\big(|\Lvltl{B}{2}{\ell_1}\cap\overline{\Lvl{B}}_{\ell_1,\ell}|\le  (\tfrac12+c_3\gamma_n n^{-1/2})|\Lvltl{B}{2}{\{\ell_1,\ell\}}\cap\overline{\Lvl{B}}_{\ell_1,\ell}|\big)
\\
&\qquad\leq\Prob\big(\mathrm{Bin}(|\Lvltl{B}{2}{\{\ell_1,\ell\}}\cap \overline{\Lvl{B}}_{\ell_1,\ell}|, \tfrac{1}{2}+2c_3\gamma_n n^{-1/2})\le (\tfrac{1}{2}+c_3\gamma_n n^{-1/2})|\Lvltl{B}{2}{\{\ell_1,\ell\}}\cap \overline{\Lvl{B}}_{\ell_1,\ell}|\big)
\\
&\qquad\leq\exp\Big(-\frac{(c_3\gamma_n n^{-1/2} |\Lvltl{B}{2}{\{\ell_1,\ell\}}\cap \overline{\Lvl{B}}_{\ell_1,\ell}|)^2}{3|\Lvltl{B}{2}{\{\ell_1,\ell\}}\cap \overline{\Lvl{B}}_{\ell_1,\ell}|}\Big)
\\
& \qquad \le \exp(-(np)^{1-o(1)} p)
\\
& \qquad
= o(1/n),
\end{align*}
where for the last estimate, we used that
$
\gamma_n^2p
=
\frac1n (np)^{2-2\eps_n}
\ge
n^{1/4+2\eps-o(1)}
\gg\log n$ since $np\ge n^{5/8+\eps}$ and $\eps_n\to0$. Thus the displayed failure probability is uniformly $o(1/n)$.

On the complement of the event in the last displayed probability bound,
\[
|\Lvltl{B}{2}{\ell_1}\cap\overline{\Lvl{B}}_{\ell_1,\ell}|
-
|\Lvltl{B}{2}{\ell}\cap\overline{\Lvl{B}}_{\ell_1,\ell}|
\ge
2c_3\gamma_n n^{-1/2}
|\Lvltl{B}{2}{\{\ell_1,\ell\}}\cap\overline{\Lvl{B}}_{\ell_1,\ell}|
\ge
c_3\gamma_n n^{-1/2}
|\Lvltl{B}{2}{\ell_1}\cap\overline{\Lvl{B}}_{\ell_1,\ell}|.
\]
Moreover, as established above,
$|\Lvltl{B}{2}{\ell_1}\cap\overline{\Lvl{B}}_{\ell_1,\ell}|
\ge np/2$,
so the right-hand side is at least $\frac12c_3\gamma_np\sqrt n$. A union bound over all $\ell\in[k]\setminus\{\ell_1\}$ completes the proof.
\end{proof}

Lemma~\ref{lem Level B} controls the comparison between $\ell_1$ and $\ell$ on the good part $\overline{\Lvl{B}}_{\ell_1,\ell}$. We now turn to the exceptional part $B\setminus\overline{\Lvl{B}}_{\ell_1,\ell}$. These vertices are more delicate because some relevant edges to level~$A$ have already been exposed, so the clean comparison used in Lemma~\ref{lem Level B} does not apply. The next lemma shows that every competing label occupies only a sufficiently small number of vertices in this exceptional part and thus, together with the preceding level-$C$ estimate, cannot offset the advantage of $\ell_1$ created in level~$C$.

\begin{lemma}\label{lem snd Level B}  
A.a.s.\ for every $\ell\in [k]\setminus \{\ell_1\}$, there are at most $200 n^{3/2}p^{5/2+\eps/100}$ vertices with label $\ell$ in $B\setminus \overline{B}_{\ell_1, \ell}$ after the second round of ALAP.
\end{lemma}
\begin{proof}
Fix $m=\lceil n^{-1/2}p^{-3/2-\eps/100} \rceil$, and observe that $m \ll k$.
The proof has two parts. First, we derive a uniform upper bound on the probability that a fixed vertex in the exceptional part receives a given label in $[k]\setminus\{\ell_1\}$ in the second round. This is the local estimate, and it is based on comparing near-tied binomial variables. 
Second, we use this local bound and a second-moment argument to show that, uniformly in $\ell$, the number of exceptional vertices receiving label $\ell$ is at most of order~$np/m$.

\smallskip\noindent
\textbf{Local estimate.} 
By Chernoff's bound together with a union bound, a.a.s.\ 
\[
\sLvltl{B}{1}{i} = np+O(mnp^2+\sqrt{np}\log m)
\] 
uniformly for all $i\in[m]$.
Moreover, Lemmas~\ref{lem approx} and~\ref{lem:max}
imply that a.a.s.\ 
\[\max_{i\in[2k]}|B_1(i)|
\le np+O(
\sqrt{np\log(1/(np^3))}).\]
Moreover, a.a.s.\ every vertex in $B$ has only up to 10 neighbors in $A$ by Claim~\ref{ob 3.5}.
Moreover, conditionally on the latter a.a.s.\ event, another Chernoff bound implies that the (up to 22) basins of vertices $v_j$ which are incident to $\{v_{\ell},v_{\ell_1}\}$ in $A$ have total size up to $22\cdot 3np$.
We condition on the entire first-round exposure,
restricting to outcomes for which the preceding
high-probability estimates hold.
In particular, the first-round basins $(B_1(i))_{i\in[2k]}$ are now fixed.

Now, we fix a vertex $v \in B\setminus \overline{B}_{\ell_1, \ell}$ for some~$\ell \le k$, and bound from above the probability of the event $\{ v\in B_2(\ell)\}$.
To do so, we expose the edges from $v$ to $B_1(i)$, $i\in[m]$; note that $v$ has up to 10 neighbors in $A$ but we ignore these for the moment.

For every $i\in[m]\cup\{\ell\}$, define
\[
X_i:=|N(v)\cap(B_1(i)\setminus\{v\})|,
\]
Conditionally on the first-round exposure, these random variables
are independent, with $X_i\sim \mathrm{Bin}(\sLvltl{B}{1}{i}-\mathds{1}_{v\in \Lvltl{B}{1}{i}},p)$.
Our first step is to probabilistically locate the maximum $M$ of $(X_i)_{i=1}^m$ and to control its fluctuations;
while this index $i$ of the maximal~$X_i$ does not necessarily correspond to the label obtained by $v$
after the second round, it will be sufficient to estimate the probability that a label $\ell$ is attributed to $v$.
For $c\in \mathbb R$, fix 
\[
x_c = x_c(n) = np^2+\sqrt{2np^2 \log (m/\sqrt{\log m})}+c\sqrt{2np^2/\log(m/\sqrt{\log m})}.
\]
Also, for brevity, set
$a_i=\sLvltl{B}{1}{i}-\mathds{1}_{v\in\Lvltl{B}{1}{i}}$.
Then, 
\begin{equation}\label{eq:x_c}
\mathbb P(X_i\ge x_c) = \sum_{j=x_c}^{a_i} \binom{a_i}{j} p^j (1-p)^{a_i-j}.\end{equation}
Instead of directly computing the above sum, we will compare consecutive terms first. Given $j$, we have
\begin{equation}\label{eq:ratio_new}
\binom{a_i}{j+1}p^{j+1}(1-p)^{a_i-j-1}
\bigg/
\binom{a_i}{j}p^j(1-p)^{a_i-j}
=
\bigg(\frac{a_i+1}{j+1}-1\bigg)\frac{p}{1-p}.
\end{equation}
Moreover, uniformly for $i\in[m]$,
\[
a_ip = np^2 + O\big(mnp^3+\sqrt{np^3}\log n\big)
= np^2+o\Bigg(\sqrt{\frac{np^2}{\log m}}\Bigg)
\]
and $p = o(1/\sqrt{np^2\log m})$, which implies that~\eqref{eq:ratio_new} is equal to
\begin{align*}
\frac{np^2}{j+1}+o\bigg(\frac{1}{\sqrt{np^2\log m}}\bigg) = 1+\frac{np^2-j-1}{j+1}+o\bigg(\frac{1}{\sqrt{np^2\log m}}\bigg).   
\end{align*}
In particular, for any $d > c$, we have that 
\begin{align*}
\prod_{j=x_c}^{x_d} \bigg(1+\frac{np^2-j-1}{j+1}+o\bigg(\frac{1}{\sqrt{np^2\log m}}\bigg)\bigg) 
&= \exp\bigg((1+o(1)) \sum_{j=x_c}^{x_d} \frac{np^2-j-1}{j+1}\bigg)\\
&= \exp\bigg((1+o(1)) \sum_{j=x_c}^{x_d} \frac{np^2-j-1}{np^2}\bigg)\\
&= \exp\bigg((1+o(1)) \frac{(np^2-x_c)^2-(np^2-x_d)^2}{2np^2}\bigg)\\ 
&= \mathrm{e}^{2(c-d)+o(1)}.
\end{align*}
As a result, using a classical approximation of the sum~\eqref{eq:x_c} by an integral and the fact that $px_c=o(1)$, we obtain that 
\begin{align}
\mathbb P(X_i\ge x_c)
&= (1+o(1)) \binom{a_i}{x_c} p^{x_c} (1-p)^{a_i-x_c} \cdot \sqrt{\frac{2np^2}{\log(m/\sqrt{\log m})}} \int_{0}^{\infty} \mathrm{e}^{-2u} \mathrm{d}u\nonumber\\
&= (1+o(1)) \bigg(\frac{\mathrm{e}a_ip}{x_c}\bigg)^{x_c} \exp\bigg(-\frac{x_c^2}{2a_i}\bigg) \frac{1}{\sqrt{2\pi x_c}} \mathrm{e}^{-pa_i} \cdot \sqrt{\frac{np^2}{2\log m}}\nonumber\\
&= (1+o(1)) \bigg(\frac{\mathrm{e}np^2}{x_c}\bigg)^{x_c} \mathrm{e}^{-np^2} \cdot \sqrt{\frac{1}{4\pi\log m}}\nonumber\\
&= (1+o(1)) \bigg(1+\frac{np^2-x_c}{x_c}\bigg)^{x_c} \mathrm{e}^{x_c-np^2}\cdot \sqrt{\frac{1}{4\pi\log m}}.\label{eq:taylor}
\end{align}
Finally, Taylor expansion of the first term in~\eqref{eq:taylor} shows that 
\begin{align*}
\bigg(1+\frac{np^2-x_c}{x_c}\bigg)^{x_c} 
&= \mathrm{e}^{np^2-x_c - (x_c-np^2)^2/2x_c + O((x_c-np^2)^3/x_c^2)}\\
&= (1+o(1))\mathrm{e}^{np^2-x_c}\cdot \mathrm{e}^{-2c}\sqrt{\log m}/m,
\end{align*}
which means that~\eqref{eq:x_c} equals $(1+o(1)) \mathrm{e}^{-2c}/(\sqrt{4\pi} m)$, and the leading order of this expression does not depend on $i\in [m]$.

Moreover, for every $i\in [m]$ and every integer value, $X_i$ takes this value with probability $O(1/\sqrt{np^2}) = o(1/m)$.
Next, fix constants $c,c'>0$. Uniformly in $j\in[m]$ and in integers
$r\in[x_{-c},x_c]$, we have
\[
\Prob(X_j>r)\le \Prob(X_j\ge x_{-c})=O(1/m)
\]
by the estimate above for~\eqref{eq:x_c}. Hence
$\Prob(X_j\le r)=1-O(1/m)$ and therefore,
\[
\Prob(X_j\ge r-c'\mid X_j\le r)
=
\frac{\Prob(r-c'\le X_j\le r)}{\Prob(X_j\le r)}
\le
\frac{(c'+1)\sup_s\Prob(X_j=s)}{1-O(1/m)}
=
O\Big(\frac{1}{\sqrt{np^2}}\Big).
\]
Since, conditionally on $X_i=r$ and on $X_j\le r$ for every $j\neq i$,
the variables $(X_j)_{j\neq i}$ remain independent, a union bound gives
\begin{equation}\label{eq:max2}
\sup_{\substack{i\in[m]\\
r\in[x_{-c},x_c]\cap\mathbb Z}}
\Prob\big(
\exists j\in[m]\setminus\{i\},\ X_j\ge r-c'
\,\big|\,
X_i=r,\ X_j\le r\ \text{for all }j\neq i
\big)
=
O\Big(\frac{m}{\sqrt{np^2}}\Big)
=
o(1).
\end{equation}
The constant $c'$ allows us to account for the
bounded contribution from neighbors in level~$A$
and, when needed, for the effect of fixing the
status of a single edge within level~$B$.

We will also use that, for every $i\in[m]$,
\begin{equation}\label{eq:max-prob}
\Prob\Big(X_i=\max_{j\in [m]} X_j, \max_{j\in [m]\setminus \{i\}} X_j\le X_i-12\Big)=(1+o(1))/m,
\end{equation}
where the $o(1)$ is uniform over all $i\in [m]$. In fact,
the estimates above remain uniform if the constant~$c$ is replaced by
$c_n=\log\log m$. They imply that
$\Prob(M<x_{-c_n})=o(1/m)$ and
$\Prob(X_i>x_{c_n})=o(1/m)$ uniformly in~$i$.
Moreover, repeating the Stirling expansion used in the calculation
leading to~\eqref{eq:taylor}, with $x_c$ replaced by $r$, and using
$a_i p=np^2+o(\sqrt{np^2/\log m})$, gives, uniformly for
$r\in[x_{-c_n},x_{c_n}]$ and $i,j\in[m]$,
$\Prob(X_i=r)=(1+o(1))\Prob(X_j=r)$.
Thus, the probability that any given $X_i$ is the unique maximum
within this window and exceeds the second maximum by at least 12 is $(1+o(1))/m$, with the $o(1)$-factor uniform over $i\in [m]$.
Note also that the sum of these probabilities is $1-o(1)$, which proves~\eqref{eq:max-prob}.
Moreover, when $\ell > m$, similar estimates allow to establish that vertex $v$ receives label $\ell$ with probability at most $(1+o(1))/m$ (the bound is only one-sided this time, as $\ell$ could be much larger than $m$).

For $\ell\in[m]$, the bounded contribution
from neighbors in level~$A$ implies that,
if $v$ receives label~$\ell$, then
$X_\ell\ge M-10$.
The estimates leading to~\eqref{eq:max-prob},
together with the near-tie bound~\eqref{eq:max2},
therefore imply that $v$ receives label~$\ell$
in the second round of ALAP with probability at most 
$(1+o(1))/m$,
uniformly in $\ell\in[m]$.
Together with the bound for $\ell>m$
mentioned at the end of the previous paragraph, this gives the required
uniform local estimate.

\smallskip\noindent
\textbf{Second-moment bound.} We now focus on bounding the exact number of vertices $Y_{\ell}$ in $B\setminus \overline{B}_{\ell_1, \ell}$ which obtain label $\ell$ at the second round.
By the uniform local estimate above and the fact that
$|B\setminus\overline B_{\ell_1,\ell}|\le 22\cdot3np$, for every $\ell\in [k]\setminus\{\ell_1\}$,
\[
\mathbb EY_\ell
\le (1+o(1))\frac{|B\setminus\overline B_{\ell_1,\ell}|}{m}
\le \frac{100np}{m}.
\]

Next, we show that a.a.s., for every $\ell\in [k]\setminus\{\ell_1\}$, $Y_\ell$ is at most $200np/m$. The argument is similar to Step 2 in the proof of Lemma~\ref{claim 4}. Note that
\[
Y_\ell=\displaystyle\sum_{v\in B\setminus\overline{B}_{\ell_1,\ell}}\mathds{1}_{v\in\Lvltl{B}{2}{\ell}}.
\]
As before, all expectations and variances below
are conditional on the first-round exposure;
we suppress this conditioning from the notation.
Then, we have that
\begin{align}
    \mathbb V(Y_\ell) 
    &=\; \sum_{u,v\in B\setminus\overline{B}_{\ell_1,\ell}}\Big(\mathbb E[\mathds{1}_{u\in \Lvltl{B}{2}{\ell}}\mathds{1}_{v\in \Lvltl{B}{2}{\ell}}] - \mathbb E[\mathds{1}_{u\in \Lvltl{B}{2}{\ell}}]\mathbb E[\mathds{1}_{v\in \Lvltl{B}{2}{\ell}}]\Big)\nonumber \\
    &\le \mathbb E[Y_\ell] + \sum_{\substack{u,v\in B\setminus\overline{B}_{\ell_1,\ell}\\ u\neq v}}
     \Big(\mathbb E[\mathds{1}_{u\in \Lvltl{B}{2}{\ell}}\mathds{1}_{v\in \Lvltl{B}{2}{\ell}}] - \mathbb E[\mathds{1}_{u\in \Lvltl{B}{2}{\ell}}]\mathbb E[\mathds{1}_{v\in \Lvltl{B}{2}{\ell}}]\Big).\label{eq:snd line}
\end{align}
Denote the host binomial graph by $G$. Using that the events $u\in \Lvltl{B}{2}{\ell}$ and $v\in \Lvltl{B}{2}{\ell}$ are independent conditionally on each of the events $uv\in G$ and $uv\notin G$, we deduce that for all pairs of different vertices $u,v$ in $B\setminus\overline{B}_{\ell_1,\ell}$, $\mathbb E[\mathds{1}_{u\in \Lvltl{B}{2}{\ell}}\mathds{1}_{v\in \Lvltl{B}{2}{\ell}}]$ rewrites as
\begin{align*}
&q \mathbb E[\mathds{1}_{u\in \Lvltl{B}{2}{\ell}}\mathds{1}_{v\in \Lvltl{B}{2}{\ell}}\mid uv\notin G] + p \mathbb E[\mathds{1}_{u\in \Lvltl{B}{2}{\ell}}\mathds{1}_{v\in \Lvltl{B}{2}{\ell}}\mid uv\in G]\\
&\qquad =\; q \mathbb E[\mathds{1}_{u\in \Lvltl{B}{2}{\ell}}\mid uv\notin G]\mathbb E[\mathds{1}_{v\in \Lvltl{B}{2}{\ell}}\mid uv\notin G]
+ p \mathbb E[\mathds{1}_{u\in \Lvltl{B}{2}{\ell}}\mid uv\in G]\mathbb E[\mathds{1}_{v\in \Lvltl{B}{2}{\ell}}\mid uv\in G],
\end{align*}
while $\mathbb E[\mathds{1}_{u\in \Lvltl{B}{2}{\ell}}]\mathbb E[\mathds{1}_{v\in \Lvltl{B}{2}{\ell}}]$ rewrites as
\begin{align*}
&(q\mathbb E[\mathds{1}_{u\in \Lvltl{B}{2}{\ell}}\mid uv\notin G]+p\mathbb E[\mathds{1}_{u\in \Lvltl{B}{2}{\ell}}\mid uv\in G])
\\
&\qquad\cdot
(q\mathbb E[\mathds{1}_{v\in \Lvltl{B}{2}{\ell}}\mid uv\notin G]+p\mathbb E[\mathds{1}_{v\in \Lvltl{B}{2}{\ell}}\mid uv\in G]).
\end{align*}
This implies that the general term in the sum in~\eqref{eq:snd line} rewrites as
\begin{equation}\label{eqn:covIdn}
\begin{split}
&pq (\mathbb E[\mathds{1}_{u\in \Lvltl{B}{2}{\ell}}\mid uv\notin G] - \mathbb E[\mathds{1}_{u\in \Lvltl{B}{2}{\ell}}\mid uv\in G])
\\&\qquad
\cdot (\mathbb E[\mathds{1}_{v\in \Lvltl{B}{2}{\ell}}\mid uv\notin G] - \mathbb E[\mathds{1}_{v\in \Lvltl{B}{2}{\ell}}\mid uv\in G]).
\end{split}
\end{equation}

We now bound the covariance terms in~\eqref{eq:snd line}
directly. For distinct
$u,v\in B\setminus\overline B_{\ell_1,\ell}$ and $e\in\{0,1\}$, write
\[
a_v^{(e)}
=
\mathbb E\big[
\mathds{1}_{v\in\Lvltl{B}{2}{\ell}}
\mid \mathds{1}_{uv\in G}=e
\big].
\]
The preceding local estimate also applies after
conditioning on the presence or absence of the edge~$uv$:
fixing the status of this edge changes at most one
second-round score by one, which can be absorbed into
the constant shift (that is, $c'$) in the binomial comparison.
Consequently, for every $\ell\in [k]\setminus\{\ell_1\}$ and
$e\in\{0,1\}$,
$a_v^{(e)}=O(1/m)$,
uniformly in $u$ and $v$.
Hence,
$|a_v^{(1)}-a_v^{(0)}|=O(1/m)$.
It follows from~\eqref{eqn:covIdn} that,
for every pair of distinct vertices
$u,v\in B\setminus\overline B_{\ell_1,\ell}$,
\[
\big|
\operatorname{Cov}\big(
\mathds{1}_{u\in\Lvltl{B}{2}{\ell}},
\mathds{1}_{v\in\Lvltl{B}{2}{\ell}}
\big)
\big|
=O(p/m^2),
\]
uniformly in $\ell\in[k]\setminus\{\ell_1\}$.

Since
$|B\setminus\overline B_{\ell_1,\ell}|\le 22\cdot3np$
and $\mathbb EY_\ell\le100np/m$, equation~\eqref{eq:snd line} now yields
\[
\mathbb V(Y_\ell)
\le
\mathbb EY_\ell
+
O\bigg(p\Big(\frac{np}{m}\Big)^2\bigg)
=
O\bigg(p\Big(\frac{np}{m}\Big)^2\bigg),
\]
where in the last equality we used
$np/m=o(p(np/m)^2)$. Hence, by Chebyshev's inequality,
\[
\Prob\Big(Y_\ell\ge\frac{200np}{m}\Big)
\le
\frac{\mathbb V(Y_\ell)}{(100np/m)^2}
=
O(p),
\qquad \ell\in [k]\setminus\{\ell_1\}.
\]
Since $kp=o(1)$, a union bound shows that a.a.s.
$Y_\ell\le 200np/m$ for every
$\ell\in[k]\setminus\{\ell_1\}$.
Finally, by the definition of $m$,
$\frac{np}{m}
\le n^{3/2}p^{5/2+\eps/100}$,
which completes the proof.
\end{proof}

\paragraph{Third-round amplification and completion of the proof.}
We now show that the second-round advantage of label~$\ell_1$
in level~$C$ leads to its adoption by every vertex in that
level during the third round.
The main task is to show that the contribution from neighbors
in level~$B$ cannot offset this advantage.
Once this is established, the still-unexposed edges within
level~$C$ allow us to amplify the advantage through
concentration estimates.
Finally, we transfer the conclusion from ALAP to LPA and
deduce consensus after five rounds.

To control the contribution from level~$B$, we will apply
the following lemma with
$\Omega_\ell=\Lvltl{B}{2}{\ell}$, the set of vertices
in level~$\Lvl{B}$ carrying label~$\ell$ after round~2,
and $G_\ell=\overline{\Lvl{B}}_{\ell_1,\ell}$.
Lemmas~\ref{lem Level B} and~\ref{lem snd Level B}
established two complementary properties of these sets:
the first gives a relative advantage for label~$\ell_1$
over label~$\ell$ within $G_\ell$, while the second
bounds the number of vertices carrying label~$\ell$
in $\Omega_{\ell}\setminus G_{\ell}$.
The following lemma shows that these properties suffice
to control the corresponding neighborhood counts
uniformly over vertices in level~$\Lvl{C}$.

\begin{lemma}\label{lem technical}
Let $(\Omega_\ell)_{\ell=1}^k$ be subsets of level $\Lvl{B}$ 
such that, for every $\ell\in[k]\setminus\{\ell_1\}$, there exists a set $G_\ell\subseteq\Lvl{B}$ satisfying
\[
|\Omega_{\ell_1}\cap G_\ell|-|\Omega_\ell\cap G_\ell|
\ge c_3\gamma_n n^{-1/2}|\Omega_{\ell_1}\cap G_\ell|,
\qquad
|\Omega_\ell\setminus G_\ell|
\le200n^{3/2}p^{5/2+\eps/100}.
\]
Assume that the sets $G_\ell$ 
and $\Omega_\ell$, as well as the partition into levels $A,B,C$, are determined without exposing the edges between $\Lvl{B}$ and $\Lvl{C}$.
Then, a.a.s.~for every vertex $v$ in $\Lvl{C}$ and every $\ell\in [k]$, it holds that $|N(v)\cap \Omega_{\ell}|-|N(v)\cap \Omega_{\ell_1}|\leq 400 n^{3/2} p^{7/2+\eps/100}$.
\end{lemma}
\begin{proof}
The claim is immediate for $\ell=\ell_1$, so fix $\ell\neq\ell_1$ and $v\in\Lvl{C}$. We condition on the information determining the sets in the statement, leaving all edges between $B$ and $C$ unexposed. Put
$D=200n^{3/2}p^{5/2+\eps/100}$,
$\delta=c_3\gamma_n n^{-1/2}$,
$g=|\Omega_{\ell_1}\cap G_\ell|$, and
$h=|\Omega_{\ell_1}\setminus G_\ell|$.
For all sufficiently large $n$, we have $0<\delta\le1$. By the hypotheses,
$|\Omega_\ell|\le(1-\delta)g+D$, and
$|\Omega_{\ell_1}|=g+h$.
Let
\[
X=|N(v)\cap\Omega_\ell|, 
\qquad
Y=|N(v)\cap\Omega_{\ell_1}|.
\]
Conditionally on the sets $\Omega_\ell$ and $\Omega_{\ell_1}$,
the edges from $v$ to their union are independent Bernoulli
random variables with parameter $p$. Thus,
$X$ and
$Y$ are binomial random variables,
although they need not be independent. By the hypotheses,
\[
\mathbb E[X-Y]
=p(|\Omega_\ell|-|\Omega_{\ell_1}|)
\le p(D-\delta g-h).
\]
Set $s=\frac p2(D+\delta g+h)$.
If $X-Y>2pD$, then necessarily
$X>\mathbb EX+s$ or $Y<\mathbb EY-s$.
Since
\[
\mathbb E X\le p(D+g),\qquad
\mathbb E Y=p(g+h),\qquad
s\le\frac p2(D+g+h),
\]
the additive Chernoff bounds give, for some absolute constant $c>0$,
\[
\Prob(X-Y>2pD)
\le
2\exp\Big(-4c\frac{s^2}{p(D+g+h)}\Big)
\le
2\exp\Big(-cp\frac{(D+\delta g+h)^2}{D+g+h}\Big)
\le 2\exp(-cp\delta D).
\]
Here the last inequality follows from
$(D+h+\delta g)^2\ge\delta(D+h)(D+h+g)$,
which holds because $0<\delta\le1$. Finally, since $\gamma_n=(np)^{1/2-\eps_n}$ and $np\ge n^{5/8+\eps}$,
\[
p\delta D
=200c_3n^{3/2}p^{4+\eps/100}(np)^{-\eps_n}
\ge n^{4\eps-3\eps/800-o(1)}\gg\log n.
\]
Thus $\Prob(X-Y>2pD)=o(1/(kn))$, uniformly in the conditioned configuration. A union bound over all $v\in\Lvl{C}$ and $\ell\in[k]\setminus\{\ell_1\}$ completes the proof.
\end{proof}

As indicated before the lemma, we apply Lemma~\ref{lem technical} with
$\Omega_\ell=\Lvltl{B}{2}{\ell}$ for $\ell\in[k]$ and
$G_\ell=\overline{\Lvl{B}}_{\ell_1,\ell}$
for $\ell\in[k]\setminus\{\ell_1\}$.
The two inequalities required by Lemma~\ref{lem technical} follow from
Lemmas~\ref{lem Level B} and~\ref{lem snd Level B}.
We conclude that the contribution from neighbors in level~$B$
cannot offset the second-round advantage of label~$\ell_1$
inside level~$C$ established in Lemma~\ref{claim 3'}.
Since the edges within level~$C$ remain unexposed,
this advantage can now be amplified in the third round.

\begin{proof}[Proof of the second point in Theorem~\ref{thm 1}]
We first prove the desired conclusion for the third round of ALAP and then transfer it to LPA.
We show that the advantage of $\ell_1$ over every label $\ell$ in level~$C$ after round 2 of ALAP produces, for every vertex, a neighborhood surplus much larger than the advantage of $\ell$ in level $B$ allowed by Lemma~\ref{lem technical}. The second-round labels in level~$C$ are determined by edges towards $B$, so this step leaves all edges within $C$ unexposed.
Indeed, fix a vertex $v\in \Lvl{C}$ and, for every $\ell\in [k]$, let $Y_\ell$ be the number of neighbors of $v$ in $C_2(\ell)$, that is, $Y_\ell=|N(v)\cap\Lvltl{C}{2}{\ell}|$.
We condition on the a.a.s.\ event of Lemma~\ref{claim 3'}
and on the second-round labels of all vertices in $C$.
Since these labels are determined without exposing the edges
within $C$, the latter remain independent Bernoulli
random variables with parameter $p$.
Under this conditioning the edges within $C$ are still unexposed, and $Y_\ell$ has distribution
$\mathrm{Bin}(|\Lvltl{C}{2}{\ell}\setminus\{v\}|,p)$.
Write $\mu_\ell=\mathbb EY_\ell$ and set $\eta=p^{1/2}(np)^{-2\eps_n}$, where $(\eps_n)_n$ is the sequence chosen before Lemma~\ref{lem fst snd max}. In particular,
since Lemma~\ref{claim 3'} gives
$\sLvltl{C}{2}{\ell_1}\ge n/4k$, we have $\mu_{\ell_1}\ge p(n/4k-1)$ and
\[
\mu_\ell\le
\big(1-c_2p^{1/2}(np)^{-\eps_n}\big)\mu_{\ell_1}+p
\qquad \text{for all $\ell\in[k]\setminus\{\ell_1\}$}.
\]
By Chernoff's bound, for each fixed $v$ and $\ell\neq\ell_1$,
\[
\Prob(Y_{\ell_1}\le\mu_{\ell_1}-\eta\mu_{\ell_1})
\le\exp(-\eta^2\mu_{\ell_1}/4),
\qquad
\Prob(Y_\ell\ge\mu_\ell+\eta\mu_{\ell_1})
\le\exp(-\eta^2\mu_{\ell_1}/4).
\]
Both probabilities are $o(1/(kn))$. Indeed, since
$k=\lceil 15p^{-2}\sqrt{\log n/n}\rceil$,
$p\ge n^{-3/8+\eps}$ and $\eps_n=o(1)$,
\[
\eta^2\mu_{\ell_1}
\ge (1-o(1))\frac{np^2}{4k}(np)^{-4\eps_n}
\ge n^{4\eps-o(1)}\gg\log(kn).
\]
On the a.a.s.\ event furnished by Lemma~\ref{lem technical}, the preceding two concentration estimates imply that, for each fixed $v\in C$ and $\ell\in[k]\setminus\{\ell_1\}$, with conditional failure probability $o(1/(kn))$, the difference between the number of neighbors of $v$ in $B\cup C$
carrying label $\ell_1$ and the number carrying label $\ell$
after the second round is at least
\begin{align*}
&
|N(v)\cap(\Lvltl{B}{2}{\ell_1}\cup\Lvltl{C}{2}{\ell_1})|
-|N(v)\cap(\Lvltl{B}{2}{\ell}\cup\Lvltl{C}{2}{\ell})|
\\
&\qquad\ge Y_{\ell_1}-Y_\ell-400n^{3/2}p^{7/2+\eps/100}\\
&\qquad\ge \mu_{\ell_1}-\mu_\ell
-2\eta\mu_{\ell_1}-400n^{3/2}p^{7/2+\eps/100}\\
&\qquad\ge 
\big(c_2p^{1/2}(np)^{-\eps_n}-2p^{1/2}(np)^{-2\eps_n}\big)
\mu_{\ell_1}-p-400n^{3/2}p^{7/2+\eps/100}\\
&\qquad= \Omega\big(p^{1/2}(np)^{-\eps_n}\mu_{\ell_1}\big).
\end{align*}
For the final estimate, observe that $(np)^{-\eps_n}=o(1)$ by the choice of $\eps_n$, while $\mu_{\ell_1}\ge p(n/(2k)-1)$ and the definition of $k$ give
\[
p^{1/2}(np)^{-\eps_n}\mu_{\ell_1}
\ge n^{3/2-o(1)}p^{7/2}
\gg n^{3/2}p^{7/2+\eps/100}.
\]

We finally transfer our conclusions from the ALAP to the LPA using the constructed coupling between the two.
By a union bound over $v\in C$ and $\ell\in[k]\setminus\{\ell_1\}$, a.a.s.\ every vertex in $C$ has a surplus of
\begin{equation}\label{eq:surplus}
\Omega(p^{1/2}(np)^{-\eps_n}\tfrac{np}{4k})\gg Kp
\end{equation}
neighbors carrying label $\ell_1$ over those carrying any other label in $[k]$ in the ALAP, where $K$ is defined in Lemma~\ref{lemma 2}.
By that lemma and our coupling between the ALAP and the LPA, a.a.s.\ every label $i$ outside $[k]$ can occur only on the vertex $v_i$ after round 2 of the LPA (and only when $i\in [K]$), and every other vertex receives identical labels in both procedures in round 2. 
Moreover, a.a.s.\ every vertex has at most $O(Kp)$ neighbors in $V([K])$, by Chernoff's bound and a union bound. 
Since the number of neighbors carrying $\ell_1$ is itself much larger than $Kp$ (and the label of the vertex itself changes the label counts by at most 1), under the a.a.s.\ surplus event described in~\eqref{eq:surplus} and the coupling established between the LPA and the ALAP, every vertex in $C$ receives label $\ell_1$ in round 3 of the LPA as well.
Since $|C|>0.9n$, Lemma~\ref{claim 5} justifies that a.a.s.\ only label $\ell_1$ survives after 5 rounds of LPA. 
Finally, Lemmas~\ref{lem approx} and~\ref{lem:max} imply
that a.a.s.\ 
$\max_{\ell\in[2k]}|B_1(\ell)|
\ge np+\frac12\sqrt{np\log(1/(np^3))}$.
On the other hand, since
$|B_1(1)|\sim\mathrm{Bin}(n-2k,p)$,
Chernoff's inequality implies that a.a.s.\
$|B_1(1)|\le np+\frac14\sqrt{np\log(1/(np^3))}$.
It follows that a.a.s.\ $\ell_1>1$,
which completes the proof.
\end{proof}

\end{document}